\documentclass[a4paper, 11pt]{article}
\usepackage{latexsym,amsfonts, amsmath, amssymb, amsthm}
\usepackage[utf8]{inputenc}
\usepackage[final=true,protrusion=true,expansion=true]{microtype}
\usepackage[noadjust]{cite}
\usepackage{enumitem}
\usepackage{color}
\usepackage{mathtools}

\usepackage{booktabs}
\usepackage{multirow}

\usepackage{subcaption}
\usepackage[font=scriptsize,labelfont=bf]{caption}
\usepackage[affil-it]{authblk}

\usepackage{geometry}
\usepackage{hyperref}
\hypersetup{
	final,
    colorlinks=true,
    linkcolor=blue,
    filecolor=magenta,
    urlcolor=cyan,
}

\theoremstyle{plain}
\newtheorem{theorem}{Theorem}
\newtheorem{proposition}[theorem]{Proposition}
\newtheorem{lemma}[theorem]{Lemma}
\newtheorem{corollary}[theorem]{Corollary}

\theoremstyle{plain}
\newtheorem{definition}[theorem]{Definition}

\newtheorem{remark}[theorem]{Remark}

\providecommand{\Supp}{\operatorname{supp}}                            
\providecommand{\supp}{\Supp}

\renewcommand{\Im}{\operatorname{Im}}             
\providecommand{\argmin}{\operatorname*{\arg\min}}  
\providecommand{\Id}{\Op{Id}}                     

\providecommand{\diag}{\operatorname{diag}}

\providecommand{\CC}{{\cal C}}

\providecommand{\CE}{{\cal E}}

\providecommand{\CN}{{\cal N}}

\providecommand{\CP}{{\cal P}}

\providecommand{\CT}{{\cal T}}

\providecommand{\CV}{{\cal V}}

\providecommand{\CX}{{\cal X}}
\providecommand{\CY}{{\cal Y}}

\providecommand{\bbE}{\mathbb{E}}

\providecommand{\bbR}{\mathbb{R}}

\providecommand*{\abs}[1]{\left|{#1}\right|} 
\providecommand*{\N}[1]{\left\|{#1}\right\|} 
\providecommand*{\Nnormal}[1]{\|{#1}\|} 

\newcommand*{\SN}[1]{\left|{#1}\right|}      

\providecommand*{\abs}[1]{\left|{#1}\right|} 

\newcommand*{\Op}[1]{\mathsf{#1}} 

\usepackage{bbm}
\usepackage{arydshln}
\usepackage{subcaption}
\usepackage{mdframed}
\usepackage{xcolor}
\usepackage{ifdraft}
\usepackage{comment}

\usepackage{tikz}
\usepackage{pgfplots}
\usetikzlibrary{spy,calc}
\usepackage{hyperref}

\newif\ifrevised
\newcommand{\revised}[1]{%
	\ifrevised
		\color{purple} #1 \color{black} %
	\else
		#1%
	\fi}
\revisedfalse

\newcommand{\overbar}[1]{\makebox[0pt]{$\phantom{#1}\mkern 1.5mu\overline{\mkern-1.5mu\phantom{#1}\mkern-1.5mu}\mkern 1.5mu$}#1}
\renewcommand{\underbar}[1]{\makebox[0pt]{$\phantom{#1}\mkern 1.5mu\underline{\mkern-1.5mu\phantom{#1}\mkern-1.5mu}\mkern 1.5mu$}#1}

\DeclareMathOperator*{\Law}{Law}
\DeclareMathOperator{\divergence}{div}

\newcommand{\globmin}{x^*}
\newcommand{\minobj}{\underbar \CE}

\newcommand{\indivmeasure}[0]{\varrho} 

\newcommand{\empmeasurenoarg}[0]{\widehat\rho^N}
\newcommand{\empmeasure}[1]{\widehat\rho_{#1}^N}

\newcommand{\omegaa}[0]{\omega_{\alpha}}

\newcommand{\conspoint}[1]{y_{\alpha}({#1})}
\newcommand{\conspointnoarg}{y_{\alpha}}

\makeatother

\title{\usefont{OT1}{bch}{b}{n}
	\huge Leveraging Memory Effects and Gradient Information in Consensus-Based Optimization: On Global Convergence in Mean-Field Law \\
}

\date{}

\author[1,2]{Konstantin Riedl\thanks{Email: \texttt{konstantin.riedl@ma.tum.de}}}
\affil[1]{Technical University of Munich, School of Computation, Information and Technology, Department of Mathematics, Munich, Germany}
\affil[2]{Munich Center for Machine Learning, Munich, Germany}

\begin{document}
\maketitle
\begin{abstract}
\noindent
	In this paper we study consensus-based optimization~(CBO), a versatile, flexible and customizable optimization method suitable for performing nonconvex and nonsmooth global optimizations in high dimensions.
	CBO is a multi-particle metaheuristic, which is effective in various applications and at the same time amenable to theoretical analysis thanks to its minimalistic design.
	The underlying dynamics, however, is flexible enough to incorporate different mechanisms widely used in evolutionary computation and machine learning, as we show by analyzing a variant of CBO which makes use of memory effects and gradient information.
	We rigorously prove that this dynamics converges to a global minimizer of the objective function in mean-field law for a vast class of functions under minimal assumptions on the initialization of the method.
	The proof in particular reveals how to leverage further, in some applications advantageous, forces in the dynamics without loosing provable global convergence.
	To demonstrate the benefit of the herein investigated memory effects and gradient information in certain applications, we present numerical evidence for the superiority of this CBO variant in applications such as machine learning and compressed sensing, which en passant widen the scope of applications of CBO.
\end{abstract}

{\noindent\small{\textbf{Keywords:} high-dimensional global optimization, derivative-free optimization, nonsmoothness, nonconvexity, metaheuristics, consensus-based optimization, mean-field limit, Fokker-Planck equations, memory effects, gradient information}}\\

{\noindent\small{\textbf{AMS subject classifications:} 65K10, 90C26, 90C56, 35Q90, 35Q84}}

\section{Introduction} \label{sec:introduction}

Interacting multi-particle systems are ubiquitous in a wide variety of scientific disciplines with application areas reaching from atomic scales over the human scale to the astronomical scale.
For instance, large-scale multi-agent models are used to understand the coordinated movement of animal groups~\cite{parrish1999complexity,couzin2005effective} or crowds of people~\cite{cristiani2011multiscale,albi2015invisible}.
Especially fascinating in this context is that such complex and often intelligent behavior\,---\,phenomena known as self-organization and swarm intelligence\,---\,emerge from seemingly simple rules of interaction~\cite{vicsek2012collective}.
This intriguing capabilities have drawn researchers' attention towards designing interacting particle systems for specific purposes in various disciplines.
In applied mathematics in particular, agent-based optimization algorithms look back on a long and successful history of empirically achieving state-of-the-art performance on challenging global unconstrained problems of the form
\begin{equation*}
	\globmin = \argmin_{x\in\bbR^d}\CE(x).
\end{equation*}
Here, $\CE : \bbR^d\rightarrow \bbR$ denotes a possibly nonconvex and nonsmooth high-dimensional objective function, whose global minimizer~$\globmin$ is assumed to exist and be unique for the remainder of this work.
Well-known representatives of this family are Evolutionary Programming~\cite{fogel2006evolutionary}, Genetic Algorithms~\cite{holland1992adaptation}, Particle Swarm Optimization \cite{kennedy1995particle} and Ant Colony Optimization~\cite{dorigo2005ant}.
They belong to the broad class of so-called metaheuristics~\cite{blum2003metaheuristics,back1997handbook}, which are methods orchestrating an interaction between local improvement procedures and global strategies, deterministic and stochastic processes, to eventually design an efficient and robust procedure for searching the solution space of the objective function~$\CE$.

Motivated by both the substantiated success of metaheuristics in applications and the lack of rigorous theoretical guarantees about their convergence and performance, the authors of~\cite{pinnau2017consensus} proposed consensus-based optimization~(CBO), which follows the spirit of metaheuristics but allows for a rigorous theoretical analysis~\cite{carrillo2018analytical,fornasier2021consensus,carrillo2019consensus,fornasier2021convergence,ha2020convergenceHD,ha2021convergence}. 
By taking inspiration from consensus formation in opinion dynamics~\cite{hegselmann2002opinion}, CBO methods use $N$ particles~$X^1,\dots,X^N$ to explore the energy landscape of the objective~$\CE$ and to eventually form a consensus about the location of the global minimizer~$\globmin$.
In its original form~\cite{pinnau2017consensus}, the dynamics of each particle~$X^i$, which is governed by a stochastic differential equation~(SDE), is subject to two competing forces.
A deterministic drift term pulls the particles towards a so-called consensus point, which is an instantaneously computed weighted average of the positions of all particles and approximates the global minimizer~$\globmin$ the best possible given the currently available information.
The resulting contractive behavior is counteracted by the second term which is stochastic in nature and thereby features the exploration of the energy landscape of the objective function.
Its magnitude and therefore its explorative power scales with the distance of the individual particle from the consensus point, which encourages particles far away to explore larger regions of the domain, while particles already close advance their position only locally.

In this work, motivated by the numerical evidence presented below as well as other recent papers such as~\cite{grassi2020particle,totzeck2020consensus,schillings2022Ensemble}, we consider a more elaborate variant of this dynamics which exhibits the two following additional drift terms.
\begin{itemize}[leftmargin=*,labelindent=5ex,labelsep=3ex,topsep=0.5ex]
	\item \noindent The first is a drift towards the historical best position of the particular particle.
		To store such information, we follow the work~\cite{grassi2020particle}, where the authors introduce for each particle an additional state variable~$Y^i$, which can be regarded as the memory of the respective particle~$X^i$.
		In contrast to the original dynamics, an individual particle is therefore described by the tuple~$(X^i,Y^i)$.
		Moreover, the consensus point is no longer computed from the instantaneous positions~$X^i$, but the historical best positions~$Y^i$.
	\item \noindent The second term is a drift in the direction of the negative gradient of $\CE$ evaluated at the current position of the respective particle~$X^i$.
\end{itemize}
%
Both terms are accompanied by associated noise terms.
We now make the CBO dynamics with memory effects and gradient information rigorous by providing a formal description of the interacting particle system.
A visualization of the dynamics with all relevant quantities and forces is provided in Figure~\ref{fig:illustration_dynamics}.
Given a finite time horizon~$T>0$, and user-specified parameters $\alpha,\beta,\theta,\kappa,\lambda_1,\sigma_1>0$ and $\lambda_2,\lambda_3,\sigma_2,\sigma_3\geq0$, the dynamics is given by the system of SDEs
\begin{subequations} \label{eq:CBO_micro_with_memory}
\begin{align}
	&\begin{aligned} \label{eq:CBO_micro_with_memory_X}
	dX_{t}^i = \begin{aligned}[t] 
		 &\!-\lambda_1\!\left(X_{t}^i-y_{\alpha}(\empmeasure{Y,t})\right) dt 
		    -\lambda_2\!\left(X_{t}^i-Y_{t}^i\right) dt
		    -\lambda_3\nabla\CE(X_{t}^i) \,dt\\
		 &\!+\sigma_1 D\!\left(X_{t}^i-y_{\alpha}(\empmeasure{Y,t})\right) dB_{t}^{1,i}
		    +\sigma_2 D\!\left(X_{t}^i-Y_{t}^i\right) dB_{t}^{2,i}
		    +\sigma_3 D\!\left(\nabla\CE(X_{t}^i)\right) dB_{t}^{3,i},
		\end{aligned}
	\end{aligned}\\
	&\mspace{1mu}dY_{t}^i = \kappa \left(X_{t}^i-Y_{t}^i\right) S^{\beta,\theta}\!\left(X_{t}^i, Y_{t}^i\right) dt
	\label{eq:CBO_micro_with_memory_Y}
\end{align}
\end{subequations}
for $i=1,\dots,N$ and where $((B_{t}^{m,i})_{t\geq0})_{i=1,\dots,N}$ are independent standard Brownian motions in~$\bbR^d$ for $m\in\{1,2,3\}$.
The system is complemented with independent initial data~$(X_0^i,Y_0^i)_{i=1,\dots,N}$, typically such that \revised{$X_0^i=Y_0^i$} for all $i=1,\dots,N$.
A numerical implementation of the scheme usually originates from an Euler-Maruyama time discretization of Equation~\eqref{eq:CBO_micro_with_memory}.
\begin{figure}[htp!]
	\centering
	\includegraphics[width=0.75\textwidth, trim=0 327 0 318, clip]{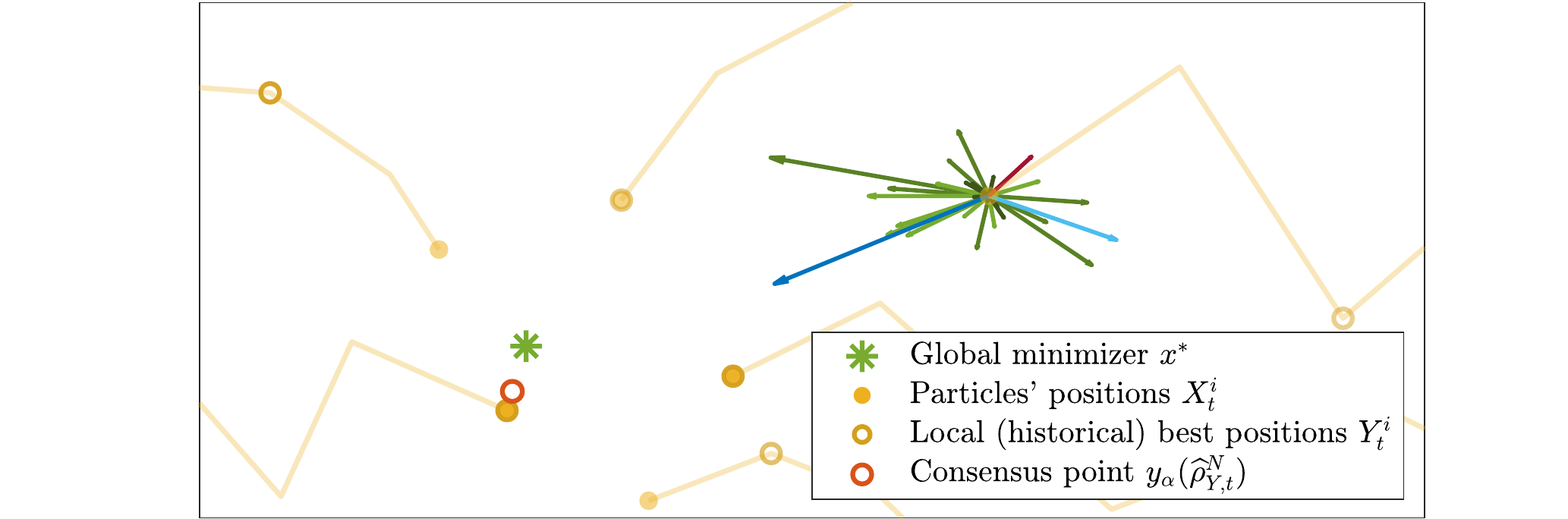}
	\caption{A visualization of the CBO dynamics~\eqref{eq:CBO_micro_with_memory} with memory effects and gradient information.
		Particles with positions~$X^1,\dots,X^N$ (yellow dots with their trajectories) explore the energy landscape of the objective~$\CE$ in search of the global minimizer~$\globmin$ (green star).
		Each particle stores its local historical best position~$Y^i_t$ (yellow circles).
		The dynamics of the position~$X^i_t$ of each particle is governed by three deterministic terms with associated random noise terms (visualized by depicting eight possible realizations with differently shaded green arrows).
		A global drift term (dark blue arrow) drags the particle towards the consensus point~$y_\alpha(\widehat{\rho}_{Y,t}^N)$ (orange circle), which is computed as a weighted (visualized through color opacity) average of the particles' historical best positions.
		A local drift term (light blue arrow) imposes movement towards the respective local best position~$Y^i_t$.
		A gradient drift term (purple arrow) exerts a force in the direction~$-\nabla\CE(X^i_t)$.}
	\label{fig:illustration_dynamics}
\end{figure}
The first term appearing in the SDE for the position $X_{t}^i$, i.e., in the first line of Equation~\eqref{eq:CBO_micro_with_memory_X}, is the drift towards the consensus point
\begin{align} \label{eq:consensus_point}
	\conspoint{\empmeasure{Y,t}}
	:= \int y \frac{\omegaa(y)}{\N{\omegaa}_{L_1(\empmeasure{Y,t})}}\,d\empmeasure{Y,t}(y),
	\quad
	\text{with}\quad \omegaa(y) := \exp(-\alpha \CE(y)).
\end{align}
Here, $\empmeasure{Y,t}$ denotes the random empirical measure of the particles' historical best positions, i.e., $\empmeasure{Y,t}:=\revised{\frac{1}{N}}\sum_{i=1}^N\delta_{Y_{t}^i}$.
Definition~\eqref{eq:consensus_point} is motivated by the fact that $\conspoint{\empmeasure{Y,t}}\approx\argmin_{i\in\{1,\dots,N\}}\CE(Y_t^i)$ as $\alpha\rightarrow\infty$ under reasonable assumptions.
The first term in the second line of Equation~\eqref{eq:CBO_micro_with_memory_X} is the with the consensus drift associated diffusion term, which injects randomness into the dynamics and thereby features the explorative nature of the algorithm.
The two commonly studied diffusion types are isotropic~\cite{pinnau2017consensus,carrillo2018analytical,fornasier2021consensus} and anisotropic~\cite{carrillo2019consensus,fornasier2021convergence} diffusion with
\begin{equation} \label{eq:diffustion_types}
	D\!\left(\,\cdot\,\right) =
	\begin{cases}
		\N{\,\cdot\,}_2 \Id ,& \text{for isotropic diffusion,}\\
		\diag\left(\,\cdot\,\right)\!,& \text{for anisotropic diffusion},
			\end{cases}
\end{equation} 
where $\Id\in\bbR^{d\times d}$ is the identity matrix and $\diag:\bbR^{d} \rightarrow \bbR^{d\times d}$ the operator mapping a vector onto a diagonal matrix with the vector as its diagonal.
\revised{Despite the potential of the dynamics getting trapped in affine subspaces,}
the coordinate-dependent scaling of anisotropic diffusion has proven to be beneficial for the performance of the method in high-dimensional applications by allowing for dimension-independent convergence rates~\cite{carrillo2019consensus,fornasier2021convergence}.
For this reason, we restrict our attention to the case of anisotropic noise in what follows.
Nevertheless, theoretically similar results as the ones presented in this work can be obtained also for the isotropic case.
The second term in the first line of Equation~\eqref{eq:CBO_micro_with_memory_X} is the drift towards the historical best position of the respective particle.
In contrast to the global nature of the consensus drift, which incorporates information from all $N$~particles, this term depends only on the past of the specific particle.
To store such information about the history of each particle~\cite{grassi2020particle}, an additional state variable~$Y^i$ is introduced for every particle, which evolves according to Equation~\eqref{eq:CBO_micro_with_memory_Y}, \revised{where}
\begin{align} \label{eq:S_beta}
	S^{\beta,\theta}(x,y) = \frac{1}{2}\big(1+\theta+\tanh\!\big(\beta\left(\CE(y)-\CE(x)\right)\!\big)\big)
\end{align}
\revised{is chosen throughout this article,}
which is an approximation to the Heaviside function $H(x,y) = \mathbbm{1}_{\CE(x)<\CE(y)}$ as $\theta\rightarrow0$ and $\beta\rightarrow\infty$.
The variable $Y_{t}^i$ can therefore be regarded as the memory of the $i$th particle, i.e., as the location of the in-time best-seen position of $X^i$ up to time~$t$.
This can be understood most easily when discretizing \eqref{eq:CBO_micro_with_memory_Y} as
\begin{align*} 
	Y_{k+1}^i = Y_{k}^i + \Delta t\kappa \left(X_{k+1}^i-Y_{k}^i\right) S^{\beta,\theta}\!\left(X_{k+1}^i, Y_{k}^i\right) 
\end{align*}
and noting that with parameter choices~$\kappa=1/\Delta t$, $\theta=0$ and $\beta\gg1$ it holds $Y_{k+1}^i = X_{k+1}^i$ if $\CE(X_{k+1}^i) < \CE(Y_{k}^i)$ and $Y_{k+1}^i = Y_{k}^i$ else.
The third term in the first line of Equation~\eqref{eq:CBO_micro_with_memory_X} is the drift in the direction of the negative gradient of $\CE$, which is a local and instantaneous contribution.
The remaining two terms are noise terms, which are associated with the formerly described memory and gradient drifts.

A theoretical convergence analysis of CBO can be carried out either by directly investigating the microscopic system~\eqref{eq:CBO_micro_with_memory} or its numerical time-discretization, as promoted for instance \revised{in a simplified setting} in the works~\cite{ha2020convergenceHD,ha2021convergence}, or alternatively, as done for example in~\cite{carrillo2018analytical,carrillo2019consensus,fornasier2021consensus,fornasier2021convergence,fornasier2020consensus_sphere_convergence}, by analyzing the macroscopic behavior of the particle density through a mean-field limit associated with~\eqref{eq:CBO_micro_with_memory}.
Formally, such mean-field limit is given by the self-consistent nonlinear and nonlocal SDE
\begin{subequations} \label{eq:CBO_macro_with_memory}
\begin{align}
	&\begin{aligned} \label{eq:CBO_macro_with_memory_X}
	d\overbar{X}_{t} = \begin{aligned}[t] 
		 &\!-\lambda_1\!\left(\overbar{X}_{t}-y_{\alpha}(\rho_{Y,t})\right) dt 
		    -\lambda_2\!\left(\overbar{X}_{t}-\overbar{Y}_{t}\right) dt
		    -\lambda_3\nabla\CE(\overbar{X}_{t}) \,dt\\
		 &\!+\sigma_1 D\!\left(\overbar{X}_{t}-y_{\alpha}(\rho_{Y,t})\right) d B_{t}^{1}
		    +\sigma_2 D\!\left(\overbar{X}_{t}-\overbar{Y}_{t}\right) d B_{t}^{2}
		    +\sigma_3 D\!\left(\nabla\CE(\overbar{X}_{t})\right) d B_{t}^{3},
		\end{aligned}
	\end{aligned}\\
	&\mspace{1mu}d\overbar{Y}_{t} = \kappa \left(\overbar{X}_{t}-\overbar{Y}_{t}\right) S^{\beta,\theta}\!\left(\overbar{X}_{t}, \overbar{Y}_{t}\right) dt, \label{eq:CBO_macro_with_memory_Y}
\end{align}
\end{subequations}
which is complemented with initial datum $(\overbar{X}_{0},\overbar{Y}_{0})\sim\rho_0$, and where $\rho_t = \rho(t) = \Law\left(\overbar{X}_{t},\overbar{Y}_{t}\right)$ with marginal law $\rho_{Y,t}$ of $\overbar{Y}_t$ given by $\rho_{Y,t} = \rho_{Y}(t, \,\cdot\,) = \int d\rho_t(\,\cdot\,,y)$.
The measure~$\rho\in\CC([0,T],\CP(\bbR^d\times\bbR^d))$ in particular weakly satisfies the Fokker-Planck equation
\begin{equation} \label{eq:CBO_with_memory_weak}
\begin{split}
	\partial_t\rho_t\!
	&= \divergence_x \!\big( \big( \lambda_1\!\left(x - \conspoint{\rho_{Y,t}}\right) + \lambda_2\!\left(x-y\right) + \lambda_3\nabla\CE(x) \big)\rho_t \big) 
	+ \divergence_y \!\big(\big(\kappa(y-x)S^{\beta,\theta}(x,y)\big)\rho_t\big) \\
	&\quad\, + \frac{1}{2}\sum_{k=1}^d\partial^2_{x_kx_k}\left(\left(\sigma_1^2D\!\left(x-\conspoint{\rho_{Y,t}}\right)_{kk}^2 + \sigma_2^2D\!\left(x-y\right)_{kk}^2 + \sigma_3^2D\!\left(\nabla\CE(x)\right)_{kk}^2\right)\rho_t\right),
\end{split}	
\end{equation}
see Definition~\ref{def:fokker_planck_weak_sense}.
Working with the partial differential equation (PDE)~\eqref{eq:CBO_with_memory_weak} instead of the interacting particle system~\eqref{eq:CBO_micro_with_memory} typically permits to employ more powerful technical tools, which result in stronger and deterministic statements about the long-time behavior of the average agent density~$\rho$.
This analysis approach is rigorously justified by the mean-field approximation, i.e., the fact that the empirical particle measure~$\empmeasure{t}:=\revised{\frac{1}{N}}\sum_{i=1}^N\delta_{(X_{t}^i,Y_{t}^i)}$ converges in some sense to the mean-field law~$\rho_t$ as the number of particles~$N$ tends to infinity.
For the original CBO dynamics, a qualitative result about convergence in distribution is provided in~\cite{huang2021MFLCBO}, which is based on a tightness argument in the path space.
More precisely, the authors of that work show that the sequence~$\{\empmeasurenoarg\}_{N\geq2}$ of $\CC([0,T],\CP(\bbR^d))$-valued random variables is tight, which permits to employ Prokhorov's theorem to obtain, up to a subsequence, some limiting measure, which turns out to be deterministic and at the same time satisfy the associated mean-field PDE.
A more desirable quantitative approximation result, on the other hand, can be established by proving propagation of chaos, i.e., by establishing for instance
\begin{align*}
	\max_{i=1,\dots,N} \sup_{t\in[0,T]} \bbE\N{(X_{t}^i,Y_{t}^i)-(\overbar{X}_{t}^i,\overbar{Y}_{t}^i)}_2^2 \leq CN^{-1} \qquad \text{ as } N\rightarrow\infty,
\end{align*}
where $(\overbar{X}_{t}^i,\overbar{Y}_{t}^i)$ denote $N$ i.i.d.\@ copies of the mean-field dynamcis~\eqref{eq:CBO_macro_with_memory}.
For the original variant of \revised{unconstrained} CBO this was first done in \cite[Section~3.3]{fornasier2021consensus}.
To keep the focus of this work on the long-time behavior of the CBO variant~\eqref{eq:CBO_with_memory_weak}, a rigorous analysis of the mean-field approximation is left for future considerations.

Before summarizing the contributions of the present paper, let us put our work into context by providing a comprehensive literature overview about the history, developments and achievements of CBO.

\paragraph{Versatility and Flexibility of CBO\,---\,A Literature Overview.}
Since its introduction in the work~\cite{pinnau2017consensus}, CBO has gained a significant amount of attention from various research groups.
This has led to a vast variety of different developments, of both theoretical and applied nature, as well as what concerns the mathematical modeling and numerical analysis of the method.
\revised{By interpreting CBO as a stochastic relaxation of gradient descent, the recent work~\cite{riedl2023gradient} even establishes a connection between the worlds of derivative-free and gradient-based optimization.}

A first rigorous but local convergence proof of the mean-field limit of CBO to global minimizers is provided \revised{for the cases of isotropic and anisotropic diffusion} in \cite{carrillo2018analytical,carrillo2019consensus}, \revised{respectively.}
By analyzing the time-evolution of the variance of the law of the mean-field dynamics~$\rho_t$ and proving its exponential decay towards zero, the authors first establish consensus formation at some stationary point before they ensure that this consensus is actually close to the global minimizer.
A similarly flavored approach is pursued in~\cite{ha2020convergenceHD,ha2021convergence}, however, directly for the fully in-time discrete microscopic system \revised{and in the simplified setting, where the same Brownian motion is used for all agents, which limits the exploration capabilities of the method.}
In contrast, in the recent works~\cite{fornasier2021consensus,fornasier2021convergence}, the authors, \revised{again for the isotropic and anisotropic CBO variant, respectively,} investigate the time-evolution of the Wasserstein-$2$ distance between the law~$\rho_t$ and a Dirac delta at the global minimizer. 
This is also the strategy which we pursue in this paper.
By proving the exponential decay of~$W_2(\rho_t,\delta_{\globmin})$ to zero, consensus at the desired location follows immediately.
Moreover, by providing a probabilistic quantitative result about the mean-field approximation, the authors give a first, and so far unique, holistic global convergence proof for the \revised{implementable, i.e., discretized} numerical CBO algorithm \revised{in the unconstrained case.} 
The results about the mean-field approximation of the latter papers were partially inspired by the series of works~\cite{fornasier2020consensus_hypersurface_wellposedness,fornasier2020consensus_sphere_convergence,fornasier2021anisotropic}, in which the authors constrain the particle dynamics of CBO to compact hypersurfaces and prove local convergence of the numerical scheme to minimizers by adapting the technique of~\cite{carrillo2018analytical,carrillo2019consensus}.
This ensures a beneficial compactness of the stochastic processes, which simplifies the convergence of the interacting particle dynamics to the mean-field dynamics.
\revised{In the unconstrained case,} such intrinsic compactness is replaced by the fact that the dynamics are bounded with high probability, which is sufficient to establish convergence in probability.
Further related works about CBO for optimizations with constraints include the papers~\cite{kim2020stochastic,ha2021emergent}, where a problem on the Stiefel manifold is approached, and the works~\cite{carrillo2021consensus,borghi2021constrained}, where the constrained optimization is recast into a penalized problem.
The philosophy of using an interacting swarm of particles to approach various relevant problems in science and engineering has promoted several variations of the original CBO algorithm for minimization.
Amongst them are methods based on consensus dynamics to tackle multi-objective optimization problems~\revised{\cite{borghi2022consensus,borghi2022adaptive,klamroth2022consensus}}, saddle point problems~\cite{huang2022consensus}, the search for several minimizers simultaneously~\cite{bungert2022polarized} or the sampling from certain distributions~\cite{carrillo2022consensus}.

In the same vein and also in the spirit of this work, the original CBO method itself has undergone several modifications allowing for a more complex dynamics.
This includes the use of particles with memory~\cite{totzeck2020consensus,grassi2021mean}, the integration of momentum~\cite{chen2020consensus}, the usage of jump-diffusion processes~\cite{kalise2022consensus} and the exploitation of on-the-fly extracted higher-order differential information through inferred gradients based on point evaluations of the objective function~\cite{schillings2022Ensemble}.
It moreover turned out that the renowned particle swarm optimization method (PSO)~\cite{kennedy1995particle} can be formulated and regarded as a second-order generalization of CBO~\cite{grassi2020particle,cipriani2021zero}.
This insight has enabled to adapt the for CBO developed analysis techniques to rigorously prove the convergence of PSO~\cite{qiu2022PSOconvergence}.

In the collection of formerly referenced works \revised{and beyond,} CBO has demonstrated to be a valuable method for a wide scope of applications reaching from
the phase retrieval or robust subspace detection problem in signal processing~\cite{fornasier2020consensus_sphere_convergence,fornasier2021anisotropic}, 
over the training of neural networks for image classification in machine learning~\cite{carrillo2019consensus,fornasier2021convergence} \revised{as well as in the setting of clustered federated learning~\cite{carrillo2023fedcbo},}
to asset allocation in finance~\cite{bae2022constrained}.
It has been furthermore employed to approximate low-frequency functions in the presence of high-frequency noise and to the task of solving PDEs with low-regularity solutions~\cite{chen2020consensus}.

\paragraph{Contributions.}
In view of the various developments and the wide scope of applications, a theoretical understanding of the long-time behavior of the in practical applications employed CBO methods is of paramount interest.
In this work we analyze a variant of CBO which incorporates memory effects as well as gradient information from a theoretical and numerical perspective.
As demonstrated concisely in Figure~\ref{fig:benefits_memory_gradient} and more comprehensively in Section~\ref{sec:numerics}, the herein investigated dynamics, which is more involved than standard CBO, proves to be beneficial in applications in machine learning and compressed sensing.
\begin{figure}[htp!]
	\centering
	\begin{subfigure}[b]{0.46\textwidth}
        \centering
        \includegraphics[trim=260 220 264 238,clip,height=0.2\textheight]{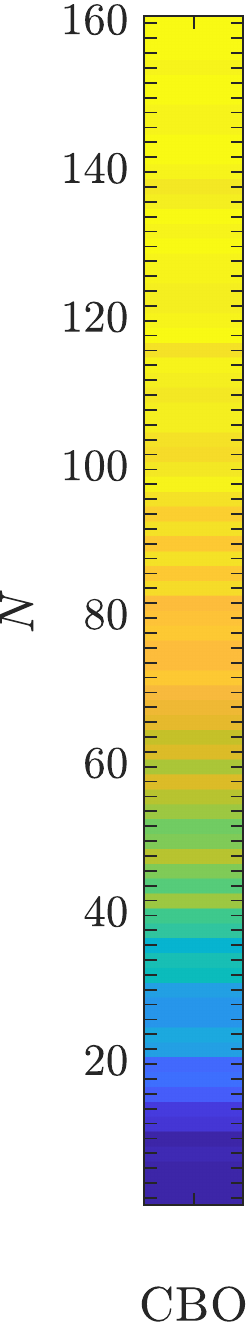}%
        \vspace{0.06cm}
        \includegraphics[trim=83 220 65 238,clip,height=0.2\textheight]{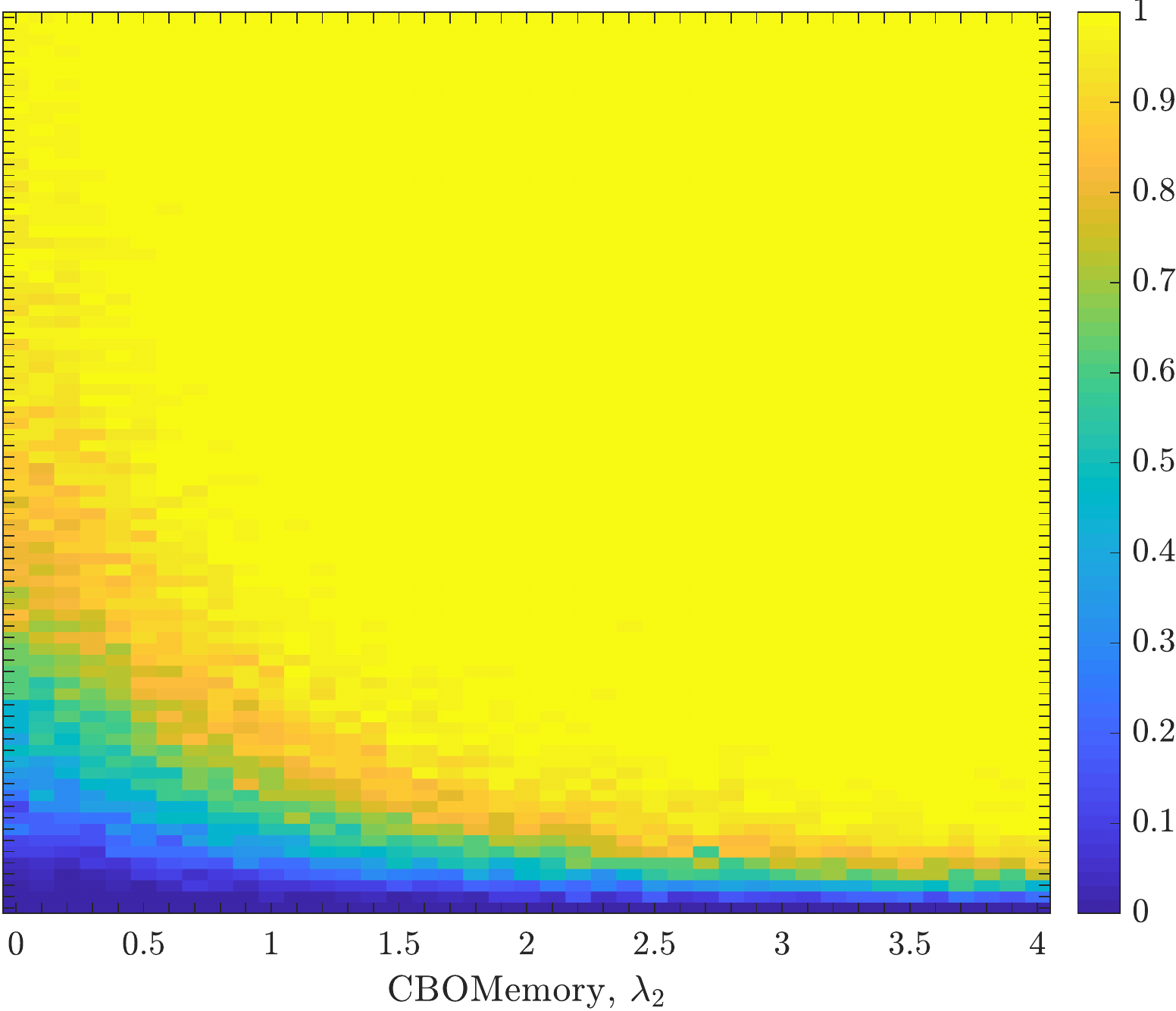}
        \caption{Memory effects and an additional drift towards the historical best position of each individual particle improve the success probability of CBO.}
        \label{fig:benefits_memory}
    \end{subfigure}~\hspace{1em}~
	\begin{subfigure}[b]{0.46\textwidth}
        \centering
        \includegraphics[trim=260 220 264 238,clip,height=0.2\textheight]{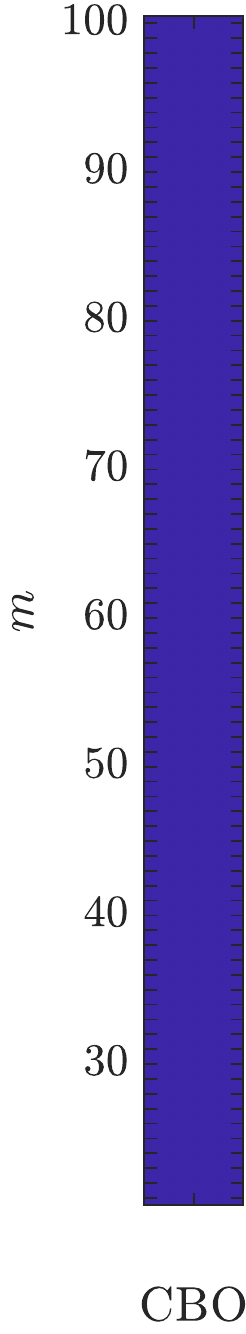}%
        \vspace{0.06cm}
        \includegraphics[trim=83 220 65 238,clip,height=0.2\textheight]{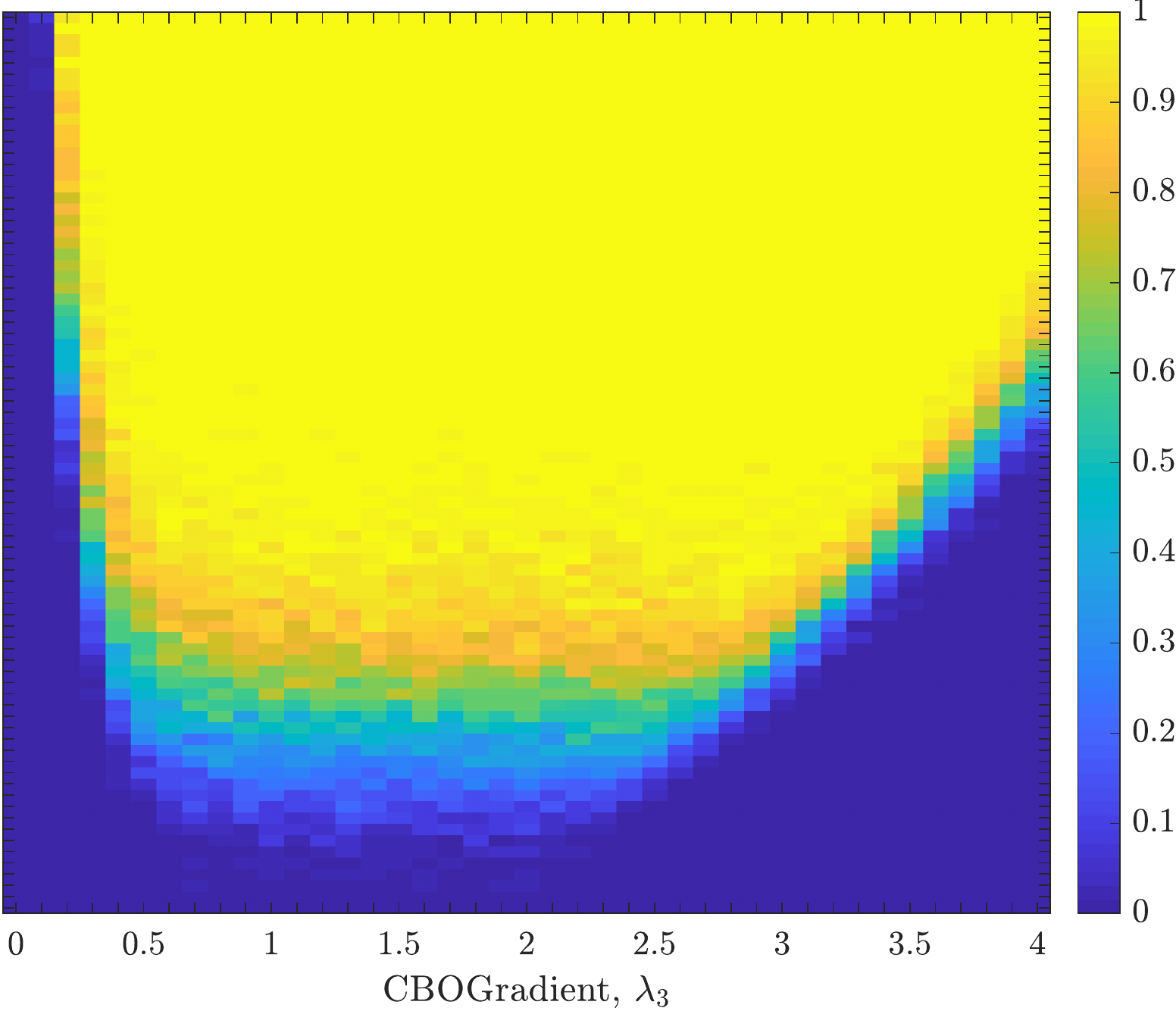}
        \caption{Gradient information and a drift in the direction of the negative gradient can be indispensable in certain applications such as compressed sensing.}
        \label{fig:benefits_gradient}
    \end{subfigure}
	\caption{\footnotesize A demonstration of the benefits of memory effects and gradient information in CBO methods.
	In both settings \textbf{(a)} and \textbf{(b)} the depicted success probabilities are averaged over $100$ runs of CBO and the implemented scheme is given by a Euler-Maruyama discretization of Equation~\eqref{eq:CBO_micro_with_memory} with time horizon~$T=20$, discrete time step size~$\Delta t=0.01$, $\alpha=100$, $\beta=\infty$, $\theta=0$, $\kappa=1/\Delta t$, $\lambda_1=1$ and $\sigma_1=\sqrt{1.6}$.
	In \textbf{(a)} we plot the success probability of CBO without (left separate column) and with (right phase diagram) memory effects for different values of the parameter~$\lambda_2$, i.e., for different strengths of the memory drift, when optimizing the Rastrigin function $\CE(x) = \sum_{k=1}^d x_k^2 + \frac{5}{2} \big(1-\cos(2\pi x_k)\big)$ in dimension~$d=4$.
	As remaining parameters we choose $\sigma_2=\lambda_1\sigma_1$ and $\lambda_3=\sigma_3=0$, i.e., no gradient information is involved.
	We observe that an increasing amount of memory drift improves the success probability significantly, even in the case where, theoretically, there are no convergence guarantees anymore, see Theorem~\ref{thm:global_convergence_main} and Corollary~\ref{cor:global_convergence_main}.
	Section~\ref{sec:numerics:Rastrigin} provides further details.
	In \textbf{(b)} we depict the success probability of CBO without (left separate column) and with (right phase diagram) gradient information for different values of the parameter~$\lambda_3$, i.e., for different strengths of the gradient drift, when solving a compressed sensing problem in dimension~$d=200$ with sparsity~$s=8$.
	On the vertical axis we depict the number of measurements~$m$, from which we try to recover the sparse signal by solving the associated $\ell_1$-regularized problem (LASSO).
	As remaining parameters we use merely $N=10$ particles, choose $\sigma_3=0$ and $\lambda_2=\sigma_2=0$, i.e., no memory drift is involved.
	We observe that gradient information is required to be able to identify the correct sparse solution and standard CBO would fail in such task.
	Section~\ref{sec:numerics:CS} provides more details.}
	\label{fig:benefits_memory_gradient}
\end{figure}
Despite this additional complexity, by employing the analysis technique devised in~\cite{fornasier2021consensus,fornasier2021convergence}, we are able to provide rigorous mean-field-type convergence guarantees to the global minimizer, which describe the behavior of the method in the large-particle limit and allow to draw conclusions about the typically observed performance in the practicable regime.
Our results for CBO with memory effects and gradient information hold for a vast class of objective functions under minimal assumptions on the initialization of the method.
Moreover, the proof reveals how to leverage further, in other applications advantageous, forces in the dynamics while still being amenable to theory and allowing for provable global convergence.

\subsection{Organization} \label{subsec:organization}
In Section~\ref{sec:main_result}, after providing details about the existence of solutions to the macroscopic SDE~\eqref{eq:CBO_macro_with_memory} and the associated PDE~\eqref{eq:CBO_with_memory_weak}, we present and discuss our main theoretical contribution.
It is about the convergence of CBO with memory effects and gradient information, as given in Equation~\eqref{eq:CBO_micro_with_memory}, to the global minimizer of the objective function in mean-field law, see~\cite[Definition~1]{fornasier2021consensus}.
More precisely, we show that the mean-field dynamics~\eqref{eq:CBO_macro_with_memory} and~\eqref{eq:CBO_with_memory_weak} converge with exponential rate to the global minimizer.
Section~\ref{sec:proof_main_theorem} contains the proof details of this result.
In Section~\ref{sec:numerics} we numerically demonstrate the benefits of the additional memory effects and gradient information of the previously analyzed CBO variant.
We in particular present applications of CBO in machine learning and compressed sensing, before we conclude the paper in Section~\ref{sec:conclusion}.

For the sake of reproducible research, in the GitHub repository \url{https://github.com/KonstantinRiedl/CBOGlobalConvergenceAnalysis} we provide the Matlab code implementing the CBO algorithm with memory effects and gradient information analyzed in this work.

\subsection{Notation} \label{subsec:notation}
\revised{Given a set $A\subset\bbR^d$, we write $(A)^c$ to denote its complement, i.e., $(A)^c:=\{z\in\bbR^d:z\not\in A\}$.}
For $\ell_\infty$ balls in $\bbR^d$ with center~$z$ and radius~$r$ we write $B^\infty_{r}(z)$.
The space of continuous functions~$f:X\rightarrow Y$ is denoted by $\CC(X,Y)$, with $X\subset\bbR^n$ and a suitable topological space $Y$.
For $X\subset\bbR^n$ open and for $Y=\bbR^m$ the function space~$\CC^k_{c}(X,Y)$ contains functions~$f\in\CC(X,Y)$ that are $k$-times continuously differentiable and have compact support.
$Y$ is omitted in the case of real-valued functions.
The operator $\nabla$ denotes the standard gradient of a function on~$\bbR^d$.

In this paper we mostly study laws of stochastic processes, $\rho\in\CC([0,T],\CP(\bbR^d))$, and we refer to a snapshot of such law at time~$t$ by writing~$\rho_t\in\CP(\bbR^d)$.
Here, $\CP(\bbR^d)$ denotes the set of all Borel probability measures~$\indivmeasure$ over $\bbR^d$.
In $\CP_p(\bbR^d)$ we moreover collect measures~$\indivmeasure \in \CP(\bbR^d)$ with finite $p$-th moment.
For any $1\leq p<\infty$, $W_p$ denotes the \mbox{Wasserstein-$p$} distance between two Borel probability measures~$\indivmeasure_1,\indivmeasure_2\in\CP_p(\bbR^d)$, see, e.g., \cite{savare2008gradientflows}.
$\bbE(\indivmeasure)$ denotes the expectation of a probability measure $\indivmeasure$.

\section{Global Convergence in Mean-Field Law} \label{sec:main_result}

In the first part of this section we provide an existence result about solutions of the nonlinear macroscopic SDE~\eqref{eq:CBO_macro_with_memory}, respectively, the associated Fokker-Planck equation~\eqref{eq:CBO_with_memory_weak}.
Thereafter we specify the class of studied objective functions and present the main theoretical result about the convergence of the dynamics~\eqref{eq:CBO_macro_with_memory} and~\eqref{eq:CBO_with_memory_weak} to the global minimizer.

Throughout this work we consider the\,---\,in typical applications beneficial\,---\,case of CBO with anisotropic diffusion, i.e.,~$D\!\left(\,\cdot\,\right) = \diag\left(\,\cdot\,\right)$ in Equations~\eqref{eq:CBO_micro_with_memory}, \eqref{eq:CBO_macro_with_memory} and \eqref{eq:CBO_with_memory_weak}, and also Equation~\eqref{eq:weak_solution_identity} below.
However, up to minor modifications, analogous results can be obtained for isotropic diffusion.

\subsection{Definition and Existence of Weak Solutions} \label{subsec:well_posedness}

Let us begin by rigorously defining weak solutions of the Fokker-Planck equation~\eqref{eq:CBO_with_memory_weak}.

\begin{definition} \label{def:fokker_planck_weak_sense}
	Let $\rho_0 \in \CP(\bbR^d\times\bbR^d)$, $T > 0$.
	We say $\rho\in\CC([0,T],\CP(\bbR^d\times\bbR^d))$ satisfies the Fokker-Planck equation~\eqref{eq:CBO_with_memory_weak} with initial condition $\rho_0$ in the weak sense in the time interval $[0,T]$, if we have for all $\phi \in \CC_c^2(\bbR^d\times\bbR^d)$ and all $t \in (0,T)$
	\begin{equation} \label{eq:weak_solution_identity}
	\begin{aligned}
		&\frac{d}{dt}\!\iint \phi(x,y) \,d\rho_t(x,y) =
		-\! \iint \kappa S^{\beta,\theta}(x,y) \left\langle y-x,\nabla_y\phi(x,y) \right\rangle d\rho_t(x,y)\\
		&\quad -\! \iint \lambda_1\!\left\langle x \!-\! \conspoint{\rho_{Y,t}}, \!\nabla_x \phi(x,y) \right\rangle \!+\! \lambda_2\!\left\langle x\!-\!y,\!\nabla_x\phi(x,y) \right\rangle \!+\! \lambda_3\!\left\langle \nabla\CE(x), \!\nabla_x \phi(x,y) \right\rangle d\rho_t(x,y)\\
		&\quad + \frac{1}{2} \iint \!\sum_{k=1}^d \!\left(\sigma_1^2D\!\left(x\!-\!\conspoint{\rho_{Y,t}}\right)_{kk}^2 \!\!+\! \sigma_2^2D\!\left(x\!-\!y\right)_{kk}^2 \!\!+\! \sigma_3^2D\!\left(\nabla\CE(x)\right)_{kk}^2 \!\right) \!\partial^2_{x_kx_k} \phi(x,y) \,d\rho_t(x,y)
	\end{aligned}
	\end{equation}
	and $\lim_{t\rightarrow 0}\rho_t = \rho_0$ \revised{(in the sense of weak convergence of measures).}
\end{definition}

For solutions of the mean-field dynamics~\eqref{eq:CBO_macro_with_memory} and~\eqref{eq:CBO_with_memory_weak} we have the following existence result.

\begin{theorem} \label{thm:well-posedness_FP}
	Let $T > 0$, $\rho_0 \in \CP_4(\bbR^d\times\bbR^d)$. Let $\CE : \bbR^d\rightarrow \bbR$ with $\minobj > -\infty$ satisfy for some constants $C_1,C_2 > 0$ the conditions
	\begin{align}
		\SN{\CE(x)-\CE(x')}
		&\leq C_1\left(\N{x}_2 + \Nnormal{x'}_2\right)\Nnormal{x-x'}_2,
		\quad \text{for all } x,x' \in \bbR^d, \label{asm:local_lipschitz}\\
		\CE(x) - \minobj
		&\leq C_2 \left(1 + \Nnormal{x}_2^2\right),
		\quad \text{for all } x \in \bbR^d, \label{asm:quadratic_growth_1}
	\end{align}
	and either $\sup_{x \in \bbR^d}\CE(x) < \infty$ or
	\begin{align}
		\CE(x) - \minobj \geq C_3\N{x}_2^2,\quad \text{for all } \N{x}_2 \geq C_4 \label{asm:quadratic_growth_2}
	\end{align}
	for some~$C_3,C_4 > 0$.
	Furthermore, in the case of an active gradient drift in the CBO dynamcis~\eqref{eq:CBO_macro_with_memory}, i.e., if $\lambda_3\not=0$, let $\CE\in\CC^1(\bbR^d)$ and obey additionally
	\begin{align}
		\N{\nabla\CE(x)-\nabla\CE(x')}_2 &\leq \widetilde{L}_{\nabla\CE} \Nnormal{x-x'}_2, \quad \text{for all } x,x' \in\bbR^d \label{asm:Lsmooth}
	\end{align}
	for some~$\widetilde{L}_{\nabla\CE}>0$.
	Then, \revised{if $(\overbar X_0, \overbar Y_0)$ is distributed according to $\rho_0$,} there exists a nonlinear process $(\overbar X, \overbar Y) \in \CC([0,T],\bbR^d\times\bbR^d)$ satisfying \eqref{eq:CBO_macro_with_memory}
	with associated law $\rho = \Law\big((\overbar X, \overbar Y)\big)$ having regularity $\rho \in \CC([0,T], \CP_4(\bbR^d\times\bbR^d))$ and being a weak solution to the Fokker-Planck equation~\eqref{eq:CBO_with_memory_weak} \revised{with $\rho(0)=\rho_0$.}
\end{theorem}

\revised{Assumption~\eqref{asm:local_lipschitz} requires that $\CE$ is locally Lipschitz-continuous with the Lipschitz constant being allowed to have linear growth.
This entails in particular that the objective has at most quadratic growth at infinity as formulated explicitly in Assumption~\eqref{asm:quadratic_growth_1},
which can be seen when choosing $x'=\globmin$ and $C_2=2C_1\max\{1,\N{\globmin}_2^2\}$ in \eqref{asm:local_lipschitz}. 
Assumption~\eqref{asm:quadratic_growth_2}, on the other hand, assumes that $\CE$ also has at least quadratic growth in the farfield, i.e., overall it grows quadratically far away from~$\globmin$.
Alternatively, $\CE$ may be bounded from above.
Since the objective $\CE$ can be usually modified for the purpose of analysis outside a sufficiently large region, these growth conditions are not really restrictive.
In case of an additional gradient drift term in the dynamics, i.e., $\lambda_3\not=0$, the objective naturally needs to be continuously differentiable.
Furthermore, Assumption~\eqref{asm:Lsmooth} imposes $\CE$ to be $\widetilde{L}_{\nabla\CE}$-smooth, i.e., having an $\widetilde{L}_{\nabla\CE}$-Lipschitz continuous gradient.}

\revised{\begin{remark} \label{rem:test_functions_redefine}
	The regularity $\rho \in \CC([0,T], \CP_4(\bbR^d\times\bbR^d))$ obtained in Theorem~\ref{thm:well-posedness_FP} above is an immediate consequence of the regularity of the initial condition $\rho_0 \in \CP_4(\bbR^d\times\bbR^d)$.
	It allows to extend the test function space $\CC^{\infty}_{c}(\bbR^d\times\bbR^d)$ in Definition~\ref{def:fokker_planck_weak_sense} to the larger space
	\begin{equation} \label{eq:function_space}
		\begin{split}
			\CC^2_{*}(\bbR^d\times\bbR^d)
			:=\big\{\phi\in\CC^2(\bbR^d\times\bbR^d):\, &\abs{\partial_{x_k}\phi(x,y)} \leq C_\phi(1\!+\!\Nnormal{x}_2\!+\!\Nnormal{y}_2) \text{ for some } C_\phi\!>\!0\\
			&\,\text{and } \sup_{(x,y)\in \bbR^d\times\bbR^d}\max_{k=1,\dots,d}|\partial^2_{x_kx_k} \phi(x,y)| < \infty\big\},
		\end{split}
		\end{equation}
		as can be seen from the proof of Theorem~\ref{thm:well-posedness_FP}, which we sketch in what follows.
\end{remark}}

\begin{proof}[Proof sketch of Theorem~\ref{thm:well-posedness_FP}]
	The proof is based on the Leray-Schauder fixed point theorem and follows the steps taken in~\cite[Theorems~3.1, 3.2]{carrillo2018analytical}.

	\noindent\textbf{Step 1:} For a given function $u\in\mathcal{C}([0,T],\mathbb{R}^d)$ and an initial measure~$\rho_0\in\CP_4(\bbR^d)$, according to standard SDE theory~\cite[Chapters~6]{arnold1974stochasticdifferentialequations}, we can uniquely solve the auxiliary SDE
		\begin{subequations}
		\begin{align*}
			&\begin{aligned}
			d\widetilde{X}_{t} = \begin{aligned}[t] 
				 &\!-\lambda_1\big(\widetilde{X}_{t}-u_t\big)\, dt 
				    -\lambda_2\big(\widetilde{X}_{t}-\widetilde{Y}_{t}\big)\, dt
				    -\lambda_3\nabla\CE(\widetilde{X}_{t}) \,dt\\
				 &\!+\sigma_1 D\big(\widetilde{X}_{t}-u_t\big)\, d B_{t}^{1}
				    +\sigma_2 D\big(\widetilde{X}_{t}-\widetilde{Y}_{t}\big)\, d B_{t}^{2}
				    +\sigma_3 D\big(\nabla\CE(\widetilde{X}_{t})\big)\, d B_{t}^{3}
				\end{aligned}
			\end{aligned}\\
			&\mspace{1mu}d\widetilde{Y}_{t} = \kappa \big(\widetilde{X}_{t}-\widetilde{Y}_{t}\big)\, S^{\beta,\theta}\big(\widetilde{X}_{t}, \widetilde{Y}_{t}\big)\, dt
		\end{align*}
		\end{subequations}
		with $(\widetilde{X}_0,\widetilde{Y}_0)\sim\rho_0$.
		This is due to the fact that the coefficients of the drift and diffusion terms are locally Lipschitz continuous and have at most linear growth,
		which, in turn, is a consequence of the assumptions on $\CE$ as well as the smoothness of $S^{\beta,\theta}$ \revised{as defined in~\eqref{eq:S_beta}.}
		This induces $\widetilde\rho_t=\Law\big((\widetilde{X}_t,\widetilde{Y}_t)\big)$.
		Moreover, the assumed regularity of the initial distribution $\rho_0 \in \CP_4(\bbR^d\times\bbR^d)$ allows to obtain a fourth-order moment estimate of the form $\mathbb{E}\big[\Nnormal{\widetilde{X}_t}_2^4+\Nnormal{\widetilde{Y}_t}_2^4\big] \leq \big(1+2\mathbb{E}\big[\Nnormal{\widetilde{X}_0}_2^4+\Nnormal{\widetilde{Y}_0}_2^4\big]\big)e^{ct}$, see, e.g.\@~\cite[Chapter~7]{arnold1974stochasticdifferentialequations}.
		So, in particular, $\widetilde\rho\in\CC([0,T],\CP_4(\bbR^d\times\bbR^d))$.

	\noindent\textbf{Step 2:}
		\revised{For some test function $\phi\in\CC^2_{*}(\bbR^d\times\bbR^d)$ as defined in \eqref{eq:function_space},} by It\^o's formula, we derive
		\begin{equation*}
		\begin{split}
			&d\phi(\widetilde{X}_t,\widetilde{Y}_t)
			= \nabla_x\phi(\widetilde{X}_t,\widetilde{Y}_t) \cdot \left(-\lambda_1\big(\widetilde{X}_{t}-u_t\big) -\lambda_2\big(\widetilde{X}_{t}-\widetilde{Y}_{t}\big) -\lambda_3\nabla\CE(\widetilde{X}_{t})\right) dt \\
			&\quad \quad\, + \nabla_y\phi(\widetilde{X}_t,\widetilde{Y}_t) \cdot \left(\kappa \big(\widetilde{X}_{t}-\widetilde{Y}_{t}\big)\, S^{\beta,\theta}\big(\widetilde{X}_{t}, \widetilde{Y}_{t}\big)\right) dt \\
			&\quad \quad\, + \frac{1}{2} \sum_{k=1}^d \partial^2_{x_kx_k}\phi(\widetilde{X}_t,\widetilde{Y}_t)\left(\sigma_1^2 D\big(\widetilde{X}_{t}-u_t\big)_{kk}^2 \!+\! \sigma_2^2 D\big(\widetilde{X}_{t}-\widetilde{Y}_{t}\big)_{kk}^2 \!+\! \sigma_3^2 D\big(\nabla\CE(\widetilde{X}_{t})\big)_{kk}^2\right)dt \\
			&\quad \quad\, + \nabla_x\phi(\widetilde{X}_t,\widetilde{Y}_t) \cdot\left(\sigma_1 D\big(\widetilde{X}_{t}-u_t\big)\, d B_{t}^{1} + \sigma_2 D\big(\widetilde{X}_{t}-\widetilde{Y}_{t}\big)\, d B_{t}^{2} + \sigma_3 D\big(\nabla\CE(\widetilde{X}_{t})\big)\, d B_{t}^{3}\right)
		\end{split}
		\end{equation*}
		After taking the expectation, applying Fubini's theorem and observing that the stochastic integrals of the form~$\bbE\int_0^t\nabla_x\phi(\widetilde{X}_t,\widetilde{Y}_t)\cdot D(\,\cdot\,)\,dB_t$ vanish as a consequence of \cite[Theorem 3.2.1(iii)]{oksendal2013stochastic} due to the established regularity $\widetilde\rho\in\CC([0,T],\CP_4(\bbR^d\times\bbR^d))$ and $\phi\in\CC^2_{*}(\bbR^d\times\bbR^d)$, we obtain
		\begin{equation*}
		\begin{split}
			&\frac{d}{dt}\bbE\phi(\widetilde{X}_t,\widetilde{Y}_t)
			= -\bbE\nabla_x\phi(\widetilde{X}_t,\widetilde{Y}_t) \cdot \left(\lambda_1\big(\widetilde{X}_{t}-u_t\big) +\lambda_2\big(\widetilde{X}_{t}-\widetilde{Y}_{t}\big) +\lambda_3\nabla\CE(\widetilde{X}_{t})\right) \\
			&\quad\ + \bbE \nabla_y\phi(\widetilde{X}_t,\widetilde{Y}_t) \cdot \left(\kappa\big(\widetilde{X}_{t}-\widetilde{Y}_{t}\big)  S^{\beta,\theta}\big(\widetilde{X}_{t}, \widetilde{Y}_{t}\big)\right) \\
			&\quad\ + \frac{1}{2} \sum_{k=1}^d \partial^2_{x_kx_k}\phi(\widetilde{X}_t,\widetilde{Y}_t)\left(\sigma_1^2 D\big(\widetilde{X}_{t}-u_t\big)_{kk}^2 \!+\! \sigma_2^2 D\big(\widetilde{X}_{t}-\widetilde{Y}_{t}\big)_{kk}^2 \!+\! \sigma_3^2 D\big(\nabla\CE(\widetilde{X}_{t})\big)_{kk}^2\right)
		\end{split}
		\end{equation*}
		according to the fundamental theorem of calculus. This shows that $\widetilde\rho\in\CC([0,T],\CP_4(\bbR^d\times\bbR^d))$ satisfies the Fokker-Planck equation
		\begin{equation} \label{proof:auxiliary_Fokker-Planck}
		\begin{aligned}
			&\frac{d}{dt}\iint \phi(x,y) \,d\widetilde\rho_t(x,y) =
			- \iint \kappa S^{\beta,\theta}(x,y) \left\langle y-x,\nabla_y\phi(x,y) \right\rangle d\widetilde\rho_t(x,y)\\
			&\quad\ - \iint \lambda_1\left\langle x - u_t, \nabla_x \phi(x,y) \right\rangle + \lambda_2\left\langle x-y,\nabla_x\phi(x,y) \right\rangle + \lambda_3\left\langle \nabla\CE(x), \nabla_x \phi(x,y) \right\rangle d\widetilde\rho_t(x,y)\\
			&\quad\ + \frac{1}{2} \iint \sum_{k=1}^d \left(\sigma_1^2D\!\left(x-u_t\right)_{kk}^2 + \sigma_2^2D\!\left(x-y\right)_{kk}^2 + \sigma_3^2D\!\left(\nabla\CE(x)\right)_{kk}^2 \right) \partial^2_{x_kx_k} \phi(x,y) \,d\widetilde\rho_t(x,y)
		\end{aligned}
		\end{equation}
		
	\noindent The remainder is identical to the cited reference and is summarized below for completeness.

	\noindent\textbf{Step 3:} Setting $\CT u:=\conspoint{\widetilde\rho_Y}\in\mathcal{C}([0,T],\mathbb{R}^d)$ provides the self-mapping property of the map
		\begin{align*}
			\CT:\CC([0,T],\mathbb{R}^d)\rightarrow\CC([0,T],\bbR^d), \quad u\mapsto\mathcal{T}u=\conspoint{\widetilde\rho_Y},
		\end{align*}
		which is compact as a consequence of a stability estimate for the consensus point~\cite[Lemma~3.2]{carrillo2018analytical}.
		More precisely, as shown in the cited result, it holds $\Nnormal{\conspoint{\widetilde\rho_{Y,t}}-\conspoint{\widetilde\rho_{Y,s}}}_2 \lesssim W_2(\widetilde\rho_{Y,t},\widetilde\rho_{Y,s})$ for $\widetilde\rho_{Y,t},\widetilde\rho_{Y,s}\in\CP_4(\bbR^d)$.
		Together with the \mbox{H\"{o}lder-$1/2$} continuity of the Wasserstein-$2$ distance~$W_2(\widetilde\rho_{Y,t},\widetilde\rho_{Y,s})$, this ensures the claimed compactness of~$\CT$.

	\noindent\textbf{Step 4:} Then, for $u=\vartheta\CT u$ with $\vartheta\in[0,1]$, there exists $\rho\in\CC([0,T],\CP_4(\bbR^d\times\bbR^d))$ satisfying~\eqref{proof:auxiliary_Fokker-Planck} with marginal~$\rho_Y$ such that $u=\vartheta \conspoint{\rho_Y}$.
		For such $u$, a uniform bound can be obtained either thanks to the boundedness or the growth condition of~$\CE$ required in the statement.
		An application of the Leray-Schauder fixed point theorem concludes the proof by providing a solution to~\eqref{eq:CBO_macro_with_memory}.
\end{proof}

\subsection{Main Result} \label{subsec:main_results}

We now present the main theoretical result about global mean-field law convergence of CBO with memory effects and gradient information for objectives that satisfy the following conditions.
\begin{definition}[Assumptions] \label{def:assumptions}
	Throughout we are interested in functions $\CE \in \CC(\bbR^d)$, for which
	\begin{enumerate}[label=A\arabic*,labelsep=10pt,leftmargin=35pt]
		\item\label{asm:zero_global} there exists \revised{a unique} $\globmin\in\bbR^d$ such that $\CE(\globmin)=\inf_{x\in\bbR^d} \CE(x)=:\underbar{\CE}$, and
		\item\label{asm:icp} there exist $\CE_{\infty},R_0,\eta > 0$, and  $\nu \in (0,\infty)$ such that
		\begin{align}
			\label{eq:asm_icp_vstar}
			\N{x-\globmin}_\infty &\leq \frac{1}{\eta}(\CE(x)-\minobj)^{\nu} \quad \text{for all } x \in B^\infty_{R_0}(\globmin),\\
			\label{eq:asm_icp_farfield}
			\CE_{\infty} &< \CE(x)-\minobj\quad \text{for all } x \in \big(B^\infty_{R_0}(\globmin)\big)^c.
		\end{align}
	\end{enumerate}
	Furthermore, for the case of an additional gradient drift component, i.e., if $\lambda_3\not=0$, we additionally require that $\CE\in\CC^1(\bbR^d)$ and that
	\begin{enumerate}[label=A\arabic*,labelsep=10pt,leftmargin=35pt]
		\setcounter{enumi}{2}
		\item\label{asm:Lipschitz_gradient} there exist $C_{\nabla\CE} > 0$ such that
			\begin{align}
				\N{\nabla\CE(x)}_2 \leq C_{\nabla\CE}\N{x-\globmin}_2 \quad \text{for all } x \in\bbR^d.
		\end{align}
	\end{enumerate}
\end{definition}

In the case, where no gradient drift is present, i.e., $\lambda_3=0$ in Equations~\eqref{eq:CBO_micro_with_memory}, \eqref{eq:CBO_macro_with_memory} and~\eqref{eq:CBO_with_memory_weak}, the objective function~$\CE$ is only required to be continuous and satisfy Assumptions~\ref{asm:zero_global} and~\ref{asm:icp}.
While the former merely imposes that the infimum is attained at $\globmin$, the latter can be regarded as a tractability condition of the energy landscape of~$\CE$~\cite{fornasier2021consensus,fornasier2020consensus_sphere_convergence}.
More precisely, the inverse continuity condition~\eqref{eq:asm_icp_vstar} ensures that~$\CE$ is locally coercive in some neighborhood of the global minimizer~$\globmin$.
Condition~\eqref{eq:asm_icp_farfield}, on the other hand, guarantees that in the farfield $\CE$ is bounded away from the minimal value by at least $\CE_{\infty}$.
This in particular excludes objectives for which $\CE(x)\approx\minobj$ far away from~$\globmin$.
\revised{Note that \ref{asm:icp} actually already implies the uniqueness of $\globmin$ requested in \ref{asm:zero_global}.}
In case of an additional gradient drift term in the dynamics, i.e., $\lambda_3\not=0$, the objective naturally needs to be continuously differentiable.
Furthermore, in Assumption~\ref{asm:Lipschitz_gradient} we impose that the gradient~$\nabla\CE$ grows at most linearly.
This is a significantly weaker assumption compared to typical smoothness assumptions about $\CE$ in the optimization literature (in particular in the analysis of stochastic gradient descent), where Lipschitz-continuity of the gradient of $\CE$ is required~\cite{moulines2011non}.

We are now ready to state the main theoretical result.
Its proof is deferred to Section~\ref{sec:proof_main_theorem}.
\revised{For the reader's convenience let us recall that 
\begin{equation*}
	W_2^2\left(\rho_t,\delta_{(x^*,x^*)}\right)
	= \iint\! \left(\N{x-\globmin}_2^2 + \N{y
	-\globmin}_2^2\right) d\rho_t(x,y),
\end{equation*}
which motivates to investigate the behavior of the Lyapunov functional $\CV(\rho_t)$ as introduced in \eqref{eq:functional_V} below.}

\begin{theorem} \label{thm:global_convergence_main}
	Let $\CE\in\CC(\bbR^d)$ satisfy \ref{asm:zero_global} and \ref{asm:icp}.
	Furthermore, in the case of an active gradient drift in the CBO dynamcis~\eqref{eq:CBO_macro_with_memory}, i.e., if $\lambda_3\not=0$, let $\CE\in\CC^1(\bbR^d)$ obey in addition \ref{asm:Lipschitz_gradient}.
	Moreover, let $\rho_0 \in \CP_4(\bbR^d\times\bbR^d)$ be such that $(\globmin,\globmin)\in\supp(\rho_0)$.
	Let us define the functional
	\begin{equation} \label{eq:functional_V}
		\CV(\rho_t) := \frac{1}{2} \iint\! \left(\N{x-\globmin}_2^2 + \N{y-x}_2^2\right) d\rho_t(x,y),
	\end{equation}
	and the \revised{rates}
	\begin{subequations}
	\begin{align}
		\chi_1 &:= \min\left\{\lambda_1 \!-\! \lambda_2 \!-\! 3\lambda_3 C_{\nabla\CE} \!-\! 2\sigma_1^2 \!-\! 2\sigma_3^2C_{\nabla\CE}^2, 2\kappa\theta \!+\! \lambda_2 \!-\! \lambda_1 \!-\! \lambda_3 C_{\nabla\CE} \!-\! 2\sigma_2^2\right\}, \label{eq:chi_1} \\
		\revised{\chi_2} &:= \revised{\max\left\{3\lambda_1 \!+\! \lambda_2 \!+\! 3\lambda_3 C_{\nabla\CE} \!-\! 2\sigma_1^2 \!+\! 2\sigma_3^2C_{\nabla\CE}^2, 2\kappa\theta \!+\! 3\lambda_2 \!+\! \lambda_1 \!+\! \lambda_3 C_{\nabla\CE} \!-\! 2\sigma_2^2\right\},} \label{eq:chi_2}
	\end{align}
	\end{subequations}
	which we assume to be strictly positive through a sufficient choice of the parameters of the CBO dynamics.
	Furthermore, \revised{provided that $\CV(\rho_0)>0$,} fix any $\varepsilon \in (0,\CV(\rho_0))$, $\vartheta \in (0,1)$ and define the time horizon
	\begin{equation} \label{eq:end_time_star_statement}
		T^* := \frac{1}{(1-\vartheta)\revised{\chi_1}}\log\left(\frac{\CV(\rho_0)}{\varepsilon}\right).
	\end{equation}
	Then there exists $\alpha_0 > 0$, depending (among problem dependent quantities) also on $\varepsilon$ and $\vartheta$, such that for all $\alpha > \alpha_0$, if $\rho \in \CC([0,T^*], \CP_4(\bbR^d\times\bbR^d))$ is a weak solution to the Fokker-Planck equation~\eqref{eq:CBO_with_memory_weak} on the time interval $[0,T^*]$ with initial condition $\rho_0$,
	we have
	\revised{\begin{equation} \label{eq:thm:global_convergence_main:T_characterization}
		\CV(\rho_T) = \varepsilon
		\quad\text{with}\quad
		T \in \left[\frac{(1-\vartheta)\chi_1}{(1+\vartheta/2)\chi_2}T^*,T^*\right].
	\end{equation}}
	Furthermore, \revised{on the time interval $[0,T]$, $\CV(\rho_t)$ decays at least exponentially fast, i.e., for all $t\in[0,T]$ it holds}
	\begin{equation} \label{eq:thm:global_convergence_main:exponential_decay}
		W_2^2(\rho_t,\delta_{(\globmin,\globmin)})
		\leq 6\CV(\rho_t)
		\leq 6\CV(\rho_0) \exp\left(-(1-\vartheta)\chi_1 t\right).
	\end{equation}
\end{theorem}

Theorem~\ref{thm:global_convergence_main} proves the exponentially fast convergence of the law~$\rho$ of the dynamics~\eqref{eq:CBO_macro_with_memory} to the global minimizer~$\globmin$ of $\CE$ under a minimal assumption about the initial distribution~$\rho_0$.
The result in particular allows to devise a strategy for the parameter choices of the method.
Namely, fixing the parameters $\lambda_2, \lambda_3,\sigma_1,\sigma_2,\sigma_3$ and $\theta$, choosing $\lambda_1$ and consecutively $\kappa$ such that
\begin{align*}
	\lambda_1 > \lambda_2 + 3\lambda_3 C_{\nabla\CE} + 2\sigma_1^2 + 2\sigma_3^2C_{\nabla\CE}^2
	\quad  \text{and} \quad
	\kappa > \frac{1}{2\theta} \left( - \lambda_2 + \lambda_1 + \lambda_3 C_{\nabla\CE} + 2\sigma_2^2\right)
\end{align*}
ensures that the convergence rate~\revised{$\chi_1$} is strictly positive.
\revised{Since $\chi_2\geq\chi_1$, $\chi_2>0$ as well.}
\revised{Given a desired accuracy~$\varepsilon$,} by consulting the proof in Section~\ref{subsec:proof_main}, we can further derive an estimate on the lower bound of~$\alpha$, namely
\begin{align*}
	\alpha_0 \sim d+\log16d + \log\left(\frac{\CV(\rho_0)}{\varepsilon}\right) - \log c\left(\vartheta,\chi_1,\lambda_1,\sigma_1\right)- \log\rho_0\big(B^\infty_{r}(\globmin)\times B^\infty_{r}(\globmin)\big)
\end{align*}
for some suitably small~$r\in(0,R_0)$, \revised{which, like the hidden constant, may depend on $\varepsilon$.}
The choice of the first set of parameters, in particular what concerns the drift towards the historical best and in the direction of the negative gradient, requires some manual hyperparameter tuning and depends on the problem at hand.
We will see this also in Section~\ref{sec:numerics}, where we conduct numerical experiments in different application areas.

Eventually, \revised{with \eqref{eq:end_time_star_statement}} one can determine the \revised{maximal} time horizon~$T^*$, \revised{until which the Lyapunov functional $\CV(\rho_t)$ is guaranteed to have reached the prescribed $\varepsilon$.}
\revised{The exact time point~$T$, where $\CV(\rho_T) = \varepsilon$, is characterized more concretely in Equation~\eqref{eq:thm:global_convergence_main:T_characterization}.
Due to the presence of memory effects and gradient information, which might counteract the consensus drift of CBO, it seems challenging to specify $T$ more closely.
However, in the case of standard CBO, $T$ turns out to be equal to $T^*$ up to a factor depending merely on $\vartheta$, see, e.g., \cite{fornasier2021consensus}.

In fact, this result can be retrieved as a special case of the subsequent Corollary~\ref{cor:global_convergence_main}, where we state} an analogous convergence result for the CBO dynamics with gradient information but without memory effect.
Its respective proof follows the lines of the one of the richer dynamics in Section~\ref{sec:proof_main_theorem}, cf.\@ also~\cite[Theorem~12]{fornasier2021consensus} and \cite[Theorem~2]{fornasier2021convergence} and it is left as an exercise to the reader interested in the technical details of the proof technique.
More precisely, for the instantaneous CBO model with gradient drift,
\begin{align} \label{eq:CBO_micro_without_memory}
\begin{split}
	d\widetilde{X}_{t}^i = \begin{aligned}[t] 
		 &\!-\lambda_1\big(\widetilde{X}_{t}^i-y_{\alpha}(\widehat{\widetilde{\rho}}_t\!^{N})\big) \,dt 
		    -\lambda_3\nabla\CE(\widetilde{X}_{t}^i) \,dt\\
		 &\!+\sigma_1 D\big(\widetilde{X}_{t}^i-y_{\alpha}(\widehat{\widetilde{\rho}}_t\!^{N})\big) \,d\widetilde{B}_{t}^{1,i}
		    +\sigma_3 D\big(\nabla\CE(\widetilde{X}_{t}^i)\big) \,d\widetilde{B}_{t}^{3,i},
		\end{aligned}
\end{split}
\end{align}
where \revised{$\widehat{\widetilde{\rho}}_t\!^{N}:=\frac{1}{N}\sum_{i=1}^N\delta_{\widetilde{X}_{t}^i}$} and
to which the associated mean-field Fokker-Planck equation reads
\begin{align} \label{eq:CBO_without_memory_weak}
\begin{split}
	\partial_t\widetilde\rho_t
	&= \divergence \big( \big( \lambda_1\left(x - y_{\alpha}(\widetilde\rho_{t})\right) + \lambda_3\nabla\CE(x) \big)\widetilde\rho_t \big) \\
	&\quad\, + \frac{1}{2}\sum_{k=1}^d\partial^2_{x_kx_k}\left(\left(\sigma_1^2D\!\left(x-y_{\alpha}(\widetilde\rho_{t})\right)_{kk}^2 + \sigma_3^2D\!\left(\nabla\CE(x)\right)_{kk}^2\right)\widetilde\rho_t\right),
\end{split}	
\end{align}
we have the following convergence result.

\begin{corollary} \label{cor:global_convergence_main}
	Let $\CE\in\CC(\bbR^d)$ satisfy \ref{asm:zero_global} and \ref{asm:icp}.
	Furthermore, in the case of an active gradient drift, i.e., if $\lambda_3\not=0$, let $\CE\in\CC^1(\bbR^d)$ obey in addition \ref{asm:Lipschitz_gradient}.
	Moreover, let $\widetilde\rho_0 \in \CP_4(\bbR^d)$ be such that $\globmin\in\supp(\widetilde\rho_0)$.
	Let us define the functional
	\begin{equation}
		\widetilde\CV(\widetilde\rho_t) := \frac{1}{2} \int\! \N{x-\globmin}_2^2 d\widetilde\rho_t(x),
	\end{equation}
	and the \revised{rates}
	\begin{subequations}
	\begin{align}
		\revised{\widetilde\chi_1} &:= 2\lambda_1 - 2\lambda_3 C_{\nabla\CE} - \sigma_1^2 - \sigma_3^2C_{\nabla\CE}^2, \\
		\revised{\widetilde\chi_2} &:= \revised{2\lambda_1 + 2\lambda_3 C_{\nabla\CE} - \sigma_1^2 + \sigma_3^2C_{\nabla\CE}^2,}
	\end{align}
	\end{subequations}
	which we assume to be strictly positive through a sufficient choice of the parameters of the CBO dynamics.
	Furthermore, \revised{provided that $\widetilde\CV(\widetilde\rho_0)>0$,}fix any $\varepsilon \in (0,\widetilde\CV(\widetilde\rho_0))$, $\vartheta \in (0,1)$ and define the time horizon
	\begin{equation} \label{eq:end_time_star_statement_without_memory}
		\widetilde T^* := \frac{1}{(1-\vartheta)\revised{\widetilde\chi_1}}\log\bigg(\frac{\widetilde\CV(\widetilde\rho_0)}{\varepsilon}\bigg).
	\end{equation}
	Then there exists $\widetilde\alpha_0 > 0$, depending (among problem dependent quantities) also on $\varepsilon$ and $\vartheta$, such that for all $\alpha > \widetilde\alpha_0$, if $\widetilde\rho \in \CC([0,T^*], \CP_4(\bbR^d))$ is a weak solution to the Fokker-Planck equation~\eqref{eq:CBO_without_memory_weak} on the time interval $[0,\widetilde T^*]$ with initial condition $\widetilde\rho_0$, we have
	\revised{\begin{equation} \label{eq:thm:global_convergence_main:T_characterization_without_memory}
		\widetilde\CV(\widetilde\rho_{\widetilde{T}}) = \varepsilon
		\quad\text{with}\quad
		\widetilde{T} \in \left[\frac{(1-\vartheta)\widetilde\chi_1}{(1+\vartheta/2)\widetilde\chi_2}\widetilde{T}^*,\widetilde{T}^*\right].
	\end{equation}}
	Furthermore, \revised{on the time interval $[0,\widetilde{T}]$, $\widetilde\CV(\widetilde\rho_t)$ decays at least exponentially fast, i.e., for all $t\in[0,\widetilde{T}]$ it holds}
	\begin{equation} \label{eq:thm:global_convergence_main:exponential_decay_without_memory}
		W_2^2(\widetilde\rho_t,\delta_{\globmin})
		= 2\widetilde\CV(\widetilde\rho_t)
		\leq 2\widetilde\CV(\widetilde\rho_0) \exp\left(-(1-\vartheta)\widetilde\chi_1 t\right).
	\end{equation}
\end{corollary}

\section{Proof details for Section~\ref{subsec:main_results}} \label{sec:proof_main_theorem}
In what follows we provide the proof details for the global mean-field law convergence result of CBO with memory effects and gradient information, Theorem~\ref{thm:global_convergence_main}.
The entire section can be read as a proof sketch with \revised{Corollaries~\ref{cor:evolution_of_functionalV} and~\ref{cor:evolution_of_functionalV_lower},} Proposition~\ref{lem:laplace_quant} and Proposition~\ref{lem:lower_bound_probability_memoryCBO} containing the key individual statements.
How to combine these results rigorously to complete the proof of Theorem~\ref{thm:global_convergence_main} is then covered in detail in Section~\ref{subsec:proof_main}.

\begin{remark} \label{rem:wlog_minobj_zero}
	Without loss of generality we assume $\minobj = 0$ throughout this section.
\end{remark}

\subsection{Evolution of the Mean-Field Limit} \label{subsec:evolution_convex}

Recall that our overall goal is to establish the convergence of the dynamics~\eqref{eq:CBO_with_memory_weak} to a Dirac delta at the global minimizer~$\globmin$ with respect to the Wasserstein-$2$ distance, i.e.,
\begin{equation*}
	W_2\!\left(\rho_t, \delta_{(\globmin,\globmin)}\right) \rightarrow 0 \quad \text{as} \quad t\rightarrow\infty.
\end{equation*}
To this end we analyze the decay behavior of the functional~$\CV(\rho_t)$ as defined in~\eqref{eq:functional_V}, i.e., $\CV(\rho_t) = \frac{1}{2} \iint \!\big(\!\N{x-\globmin}_2^2 + \N{y-x}_2^2\!\big)\,d\rho_t(x,y)$.
More precisely, we will show its exponential decay with a rate controllable through the parameters of the CBO method.

Let us start below with deriving the evolution inequalities for the functionals
\begin{equation*} 
	\CX(\rho_t) = \frac{1}{2} \iint\! \N{x-\globmin}_2^2 d\rho_t(x,y)
	\quad\text{and}\quad
	\CY(\rho_t) = \frac{1}{2} \iint\! \N{y-x}_2^2 d\rho_t(x,y).
\end{equation*}

\begin{lemma} \label{lem:evolution_of_functionalXY}
	Let $\CE : \bbR^d \rightarrow \bbR$, and fix $\alpha,\lambda_1,\sigma_1 > 0$ and $\lambda_2,\sigma_2,\lambda_3,\sigma_3,\beta,\kappa,\theta \geq 0$.
	Moreover, let $T>0$ and let $\rho \in \CC([0,T], \CP_4(\bbR^d\times\bbR^d))$ be a weak solution to the Fokker-Planck equation~\eqref{eq:CBO_with_memory_weak}.
	Then the functionals $\CX(\rho_t)$ and $\CY(\rho_t)$ satisfy
	\begin{equation*}
	\begin{split}
		\frac{d}{dt}\!\begin{pmatrix}
			\CX(\rho_t)\\
			\CY(\rho_t)
		\end{pmatrix}
		\!\leq
		&-\!\begin{pmatrix}
			2\lambda_1 \!-\! \lambda_2 \!-\! 2\lambda_3 C_{\nabla\CE} \!-\! \sigma_1^2 \!-\! \sigma_3^2C_{\nabla\CE}^2 &\!\!\! -\lambda_2\!-\!\sigma_2^2 \\
			-\lambda_1 \!-\! \lambda_3 C_{\nabla\CE} \!-\! \sigma_1^2 \!-\! \sigma_3^2C_{\nabla\CE}^2 &\!\!\! 2\kappa\theta \!+\! 2\lambda_2 \!-\! \lambda_1 \!-\! \lambda_3 C_{\nabla\CE} - \sigma_2^2
		\end{pmatrix}\!
		\begin{pmatrix}
			\CX(\rho_t)\\
			\CY(\rho_t)
		\end{pmatrix}\\
		&+
		\sqrt{2}\begin{pmatrix}
			\left(\lambda_1\!+\!\sigma_1^2\right) \!\sqrt{\CX(\rho_t)} \\
			\lambda_1 \sqrt{\CY(\rho_t)} \!+\! \sigma_1^2 \sqrt{\CX(\rho_t)}
		\end{pmatrix}\!\N{\conspoint{\rho_{Y,t}}\!-\!\globmin}_2
		\!+\!
		\frac{\sigma_1^2}{2} \!\begin{pmatrix}
			1\\
			1
		\end{pmatrix}\!\N{\conspoint{\rho_{Y,t}}\!-\!\globmin}_2^2,
	\end{split}
	\end{equation*}
	where the inequality has to be understood \revised{component-wise.}
\end{lemma}

\begin{proof}
	We note that the functions $\phi_\CX(x,y) = 1/2\N{x-\globmin}_2^2$ and $\phi_\CY(x,y) = 1/2\N{y-x}_2^2$ are in $\CC^2_{*}(\bbR^d\times\bbR^d)$ and recall that $\rho$ satisfies the weak solution identity~\eqref{eq:weak_solution_identity} for such test functions.
	Hence, by applying \eqref{eq:weak_solution_identity} with $\phi_\CX$ and $\phi_\CY$ as above, we obtain for the evolution of~$\CX(\rho_t)$
	\begin{equation} \label{eq:proof:evolution_of_objective_X_aux}
	\begin{aligned}
		&\frac{d}{dt} \CX(\rho_t) =
		- \iint \lambda_1\!\left\langle x \!-\! \conspoint{\rho_{Y,t}}, x\!-\!\globmin \right\rangle + \lambda_2 \!\left\langle x\!-\!y, x\!-\!\globmin \right\rangle + \lambda_3 \!\left\langle \nabla\CE(x), x\!-\!\globmin \right\rangle d\rho_t(x,y)\\
		&\qquad\qquad\;\;\, + \frac{1}{2} \iint \sigma_1^2\N{x\!-\!\conspoint{\rho_{Y,t}}}_2^2 + \sigma_2^2\N{x\!-\!y}_2^2 + \sigma_3^2\N{\nabla\CE(x)}_2^2 d\rho_t(x,y)
	\end{aligned}
	\end{equation}
	and for the evolution of $\CY(\rho_t)$
	\begin{equation} \label{eq:proof:evolution_of_objective_Y_aux}
	\begin{aligned}
		&\frac{d}{dt} \CY(\rho_t) =
		- \iint \kappa S^{\beta,\theta}(x,y) \N{x-y}_2^2 d\rho_t(x,y)\\
		&\qquad\qquad\;\,\, - \iint \lambda_1\left\langle x - \conspoint{\rho_{Y,t}}, x-y \right\rangle + \lambda_2 \N{x-y}_2^2 + \lambda_3\left\langle \nabla\CE(x), x-y \right\rangle d\rho_t(x,y)\\
		&\qquad\qquad\;\,\, + \frac{1}{2} \iint  \sigma_1^2\N{x-\conspoint{\rho_{Y,t}}}_2^2 + \sigma_2^2\N{x-y}_2^2 + \sigma_3^2\N{\nabla\CE(x)}_2^2 d\rho_t(x,y).
	\end{aligned}
	\end{equation}
	Here we used $\nabla_x\phi_\CX(x,y) = x-\globmin$, $\nabla_y\phi_\CX(x,y)=0$, $\partial^2_{x_kx_k} \phi_\CX(x,y) = 1$, $\nabla_x\phi_\CY(x,y) = x-y$, $\nabla_y\phi_\CY(x,y)=y-x$ and $\partial^2_{x_kx_k} \phi_\CY(x,y) = 1$.
	Let us now collect auxiliary estimates in \eqref{eq:proof:auxiliary_estimates_a}--\eqref{eq:proof:auxiliary_estimates_g}, which turn out to be useful in establishing upper bounds for \eqref{eq:proof:evolution_of_objective_X_aux} and \eqref{eq:proof:evolution_of_objective_Y_aux}.
	Using standard tools such as Cauchy-Schwarz and Young's inequality we have
	\begin{subequations} \label{eq:proof:auxiliary_estimates}
	\begin{align}
		-\left\langle x-y, x-\globmin \right\rangle
		&\leq \N{x-y}_2\N{x-\globmin}_2
		\leq \frac{1}{2}\!\left(\N{x-y}_2^2+\N{x-\globmin}_2^2\right),
		\label{eq:proof:auxiliary_estimates_a}\\%
		-\left\langle x-\conspoint{\rho_{Y,t}}, x-\globmin \right\rangle 
		&= -\N{x-\globmin}_2^2 - \left\langle \globmin-\conspoint{\rho_{Y,t}}, x-\globmin \right\rangle \nonumber\\
		&\leq -\N{x-\globmin}_2^2 + \N{\conspoint{\rho_{Y,t}}-\globmin}_2 \N{x-\globmin}_2, 
		\label{eq:proof:auxiliary_estimates_b}\\%
		-\left\langle x-\conspoint{\rho_{Y,t}}, x-y \right\rangle 
		&= -\left\langle x-\globmin, x-y \right\rangle - \left\langle \globmin-\conspoint{\rho_{Y,t}}, x-y \right\rangle \nonumber\\
		&\leq \frac{1}{2}\!\left(\N{x-y}_2^2+\N{x-\globmin}_2^2\right) + \N{\conspoint{\rho_{Y,t}}-\globmin}_2\N{x-y}_2, 
		\label{eq:proof:auxiliary_estimates_c}\\%
		\N{x-\conspoint{\rho_{Y,t}}}_2^2
		&= \N{x-\globmin}_2^2 - 2\left\langle\conspoint{\rho_{Y,t}}-\globmin,x-\globmin\right\rangle + \N{\conspoint{\rho_{Y,t}}-\globmin}_2^2 \nonumber\\
		&\leq \N{x-\globmin}_2^2 + 2\N{\conspoint{\rho_{Y,t}}-\globmin}_2\N{x-\globmin}_2 + \N{\conspoint{\rho_{Y,t}}-\globmin}_2^2,
		\label{eq:proof:auxiliary_estimates_d}\\%
		\intertext{where in \eqref{eq:proof:auxiliary_estimates_b}--\eqref{eq:proof:auxiliary_estimates_d} we expanded the left-hand side of the scalar product and the norm by subtracting and adding $\globmin$.
		Furthermore, by means of \ref{asm:Lipschitz_gradient} we obtain}
		-\left\langle \nabla\CE(x), x-\globmin \right\rangle
		&\leq \N{\nabla\CE(x)}_2\N{x-\globmin}_2
		\leq C_{\nabla\CE}\N{x-\globmin}_2^2,
		\label{eq:proof:auxiliary_estimates_e}\\%
		-\left\langle \nabla\CE(x), x-y \right\rangle
		&\leq \N{\nabla\CE(x)}_2\N{x-y}_2
		\leq C_{\nabla\CE}\N{x-\globmin}_2\N{x-y}_2 \nonumber \\
		&\leq \frac{C_{\nabla\CE}}{2}\left(\N{x-y}_2^2+\N{x-\globmin}_2^2\right),
		\label{eq:proof:auxiliary_estimates_f}\\%
		\N{\nabla\CE(x)}_2^2
		&\leq C_{\nabla\CE}^2\N{x-\globmin}_2^2
		\label{eq:proof:auxiliary_estimates_g}.
	\end{align}
	\end{subequations}
	Integrating the bounds~\eqref{eq:proof:auxiliary_estimates_a}, \eqref{eq:proof:auxiliary_estimates_b}, \eqref{eq:proof:auxiliary_estimates_d}, \eqref{eq:proof:auxiliary_estimates_e} and \eqref{eq:proof:auxiliary_estimates_g} into Equation~\eqref{eq:proof:evolution_of_objective_X_aux} results in the upper bound 
	\begin{equation*}
	\begin{aligned}
		&\frac{d}{dt} \CX(\rho_t) 
		\leq 
		- \left( 2\lambda_1 - \lambda_2 - 2\lambda_3 C_{\nabla\CE} - \sigma_1^2 - \sigma_3^2C_{\nabla\CE}^2 \right) \CX(\rho_t)
		+ \left(\lambda_2+\sigma_2^2\right) \CY(\rho_t) \\
		&\qquad\qquad\;\;\, + \sqrt{2}\left(\lambda_1+\sigma_1^2\right) \sqrt{\CX(\rho_t)} \N{\conspoint{\rho_{Y,t}}-\globmin}_2 + \frac{\sigma_1^2}{2}\N{\conspoint{\rho_{Y,t}}-\globmin}_2^2,
	\end{aligned}
	\end{equation*}
	where we furthermore used that by Jensen's inequality
	\begin{equation} \label{eq:proof:E_rho_t-globmin}
		\iint \N{x-\globmin}_2 d\rho_t(x,y)
		\leq \sqrt{\iint \N{x-\globmin}_2^2 d\rho_t(x,y)}
		=\sqrt{2\CX(\rho_t)}.
	\end{equation}
	For Equation~\eqref{eq:proof:evolution_of_objective_Y_aux} we first note that, by definition, $S^{\beta,\theta}\geq\theta$ uniformly.
	This combined with the bounds~\eqref{eq:proof:auxiliary_estimates_c}, \eqref{eq:proof:auxiliary_estimates_d}, \eqref{eq:proof:auxiliary_estimates_f} and \eqref{eq:proof:auxiliary_estimates_g} allows to derive
	\begin{equation*}
	\begin{aligned}
		&\frac{d}{dt} \CY(\rho_t)
		\leq
		- \left(2\kappa\theta + 2\lambda_2 - \lambda_1 - \lambda_3 C_{\nabla\CE} - \sigma_2^2\right) \!\CY(\rho_t) 
		+ \left(\lambda_1 + \lambda_3 C_{\nabla\CE} + \sigma_1^2 + \sigma_3^2C_{\nabla\CE}^2 \right) \!\CX(\rho_t)\\
		&\qquad\qquad\;\,\, + \sqrt{2}\left(\lambda_1 \sqrt{\CY(\rho_t)} + \sigma_1^2 \sqrt{\CX(\rho_t)}\right) \N{\conspoint{\rho_{Y,t}}-\globmin}_2 + \frac{\sigma_1^2}{2}\N{\conspoint{\rho_{Y,t}}-\globmin}_2^2,
	\end{aligned}
	\end{equation*}
	where we used \eqref{eq:proof:E_rho_t-globmin} together with an analogous bound for $\iint \N{x-y}_2 d\rho_t(x,y)$.
\end{proof}

Recalling that $\CV(\rho_t) = \CX(\rho_t) + \CY(\rho_t)$ immediately allows to obtain an evolution inequality for~$\CV(\rho_t)$ of the following form.

\begin{corollary} \label{cor:evolution_of_functionalV}
	Under the assumptions of Lemma~\ref{lem:evolution_of_functionalXY} the functional $\CV(\rho_t)$ satisfies
	\begin{equation} \label{eq:cor:evolution_of_functionalV}
	\begin{split}
		\frac{d}{dt} \CV(\rho_t)
		&\leq 
			-\chi_1 \CV(\rho_t)
			+ 2\sqrt{2}\left(\lambda_1 + \sigma_1^2\right) \sqrt{\CV(\rho_t)}\N{\conspoint{\rho_{Y,t}} - \globmin}_2
			+ \sigma_1^2\N{\conspoint{\rho_{Y,t}} - \globmin}_2^2,
	\end{split}
	\end{equation}
	with $\chi_1$ as specified in~\eqref{eq:chi_1}.
\end{corollary}

\revised{
Analogously to the upper bounds on the time evolutions of the functionals $\CX(\rho_t)$, $\CY(\rho_t)$ and $\CV(\rho_t)$, we can derive bounds from below as follows.

\begin{lemma} \label{cor:evolution_of_functionalXY_lower}
	Under the assumptions of Lemma~\ref{lem:evolution_of_functionalXY} the functionals $\CX(\rho_t)$ and $\CY(\rho_t)$ satisfy
	\begin{equation*}
	\begin{split}
		\frac{d}{dt}\!\begin{pmatrix}
			\CX(\rho_t)\\
			\CY(\rho_t)
		\end{pmatrix}
		\!\geq
		&-\!\begin{pmatrix}
			2\lambda_1 \!+\! \lambda_2 \!+\! 2\lambda_3 C_{\nabla\CE} \!-\! \sigma_1^2 \!+\! \sigma_3^2C_{\nabla\CE}^2 &\!\!\! \lambda_2\!-\!\sigma_2^2 \\
			\lambda_1 \!+\! \lambda_3 C_{\nabla\CE} \!-\! \sigma_1^2 \!+\! \sigma_3^2C_{\nabla\CE}^2 &\!\!\! 2\kappa\theta \!+\! 2\lambda_2 \!+\! \lambda_1 \!+\! \lambda_3 C_{\nabla\CE} - \sigma_2^2
		\end{pmatrix}\!
		\begin{pmatrix}
			\CX(\rho_t)\\
			\CY(\rho_t)
		\end{pmatrix}\\
		&-
		\sqrt{2}\begin{pmatrix}
			\left(\lambda_1\!+\!\sigma_1^2\right) \!\sqrt{\CX(\rho_t)} \\
			\lambda_1 \sqrt{\CY(\rho_t)} \!+\! \sigma_1^2 \sqrt{\CX(\rho_t)}
		\end{pmatrix}\!\N{\conspoint{\rho_{Y,t}}\!-\!\globmin}_2,
	\end{split}
	\end{equation*}
	where the inequality has to be understood component-wise.
\end{lemma}

\begin{proof}
	By following the lines of the proof of Lemma~\ref{lem:evolution_of_functionalXY} and noticing that in analogy to the estimates~\eqref{eq:proof:auxiliary_estimates} it hold
	\begin{subequations} \label{eq:proof:auxiliary_estimates_lower}
	\begin{align}
		-\left\langle x-y, x-\globmin \right\rangle
		&\geq -\N{x-y}_2\N{x-\globmin}_2
		\geq -\frac{1}{2}\!\left(\N{x-y}_2^2+\N{x-\globmin}_2^2\right),
		\label{eq:proof:auxiliary_estimates_a_lower}\\%
		-\left\langle x-\conspoint{\rho_{Y,t}}, x-\globmin \right\rangle 
		&= -\N{x-\globmin}_2^2 - \left\langle \globmin-\conspoint{\rho_{Y,t}}, x-\globmin \right\rangle \nonumber\\
		&\geq -\N{x-\globmin}_2^2 - \N{\conspoint{\rho_{Y,t}}-\globmin}_2 \N{x-\globmin}_2, 
		\label{eq:proof:auxiliary_estimates_b_lower}\\%
		-\left\langle x-\conspoint{\rho_{Y,t}}, x-y \right\rangle 
		&= -\left\langle x-\globmin, x-y \right\rangle - \left\langle \globmin-\conspoint{\rho_{Y,t}}, x-y \right\rangle \nonumber\\
		&\geq -\frac{1}{2}\!\left(\N{x-y}_2^2+\N{x-\globmin}_2^2\right) - \N{\conspoint{\rho_{Y,t}}-\globmin}_2\N{x-y}_2, 
		\label{eq:proof:auxiliary_estimates_c_lower}\\%
		\N{x-\conspoint{\rho_{Y,t}}}_2^2
		&= \N{x-\globmin}_2^2 - 2\left\langle\conspoint{\rho_{Y,t}}-\globmin,x-\globmin\right\rangle + \N{\conspoint{\rho_{Y,t}}-\globmin}_2^2 \nonumber\\
		&\geq \N{x-\globmin}_2^2 - 2\N{\conspoint{\rho_{Y,t}}-\globmin}_2\N{x-\globmin}_2,
		\label{eq:proof:auxiliary_estimates_d_lower}\\%
		\intertext{as well as}
		-\left\langle \nabla\CE(x), x-\globmin \right\rangle
		&\geq -\N{\nabla\CE(x)}_2\N{x-\globmin}_2
		\geq -C_{\nabla\CE}\N{x-\globmin}_2^2,
		\label{eq:proof:auxiliary_estimates_e_lower}\\%
		-\left\langle \nabla\CE(x), x-y \right\rangle
		&\geq -\N{\nabla\CE(x)}_2\N{x-y}_2
		\geq -C_{\nabla\CE}\N{x-\globmin}_2\N{x-y}_2 \nonumber \\
		&\geq -\frac{C_{\nabla\CE}}{2}\left(\N{x-y}_2^2+\N{x-\globmin}_2^2\right),
		\label{eq:proof:auxiliary_estimates_f_lower}\\%
		\N{\nabla\CE(x)}_2^2
		&\geq -C_{\nabla\CE}^2\N{x-\globmin}_2^2
		\label{eq:proof:auxiliary_estimates_g_lower}.
	\end{align}
	\end{subequations}
	we obtain the statement by integrating the bounds into Equations~\eqref{eq:proof:evolution_of_objective_X_aux} and~\eqref{eq:proof:evolution_of_objective_Y_aux}.	
\end{proof}

\begin{corollary} \label{cor:evolution_of_functionalV_lower}
	Under the assumptions of Lemma~\ref{lem:evolution_of_functionalXY} the functional $\CV(\rho_t)$ satisfies
	\begin{equation} \label{eq:cor:evolution_of_functionalV_lower}
	\begin{split}
		\frac{d}{dt} \CV(\rho_t)
		&\geq 
			-\chi_2 \CV(\rho_t)
			- 2\sqrt{2}\left(\lambda_1 + \sigma_1^2\right) \sqrt{\CV(\rho_t)}\N{\conspoint{\rho_{Y,t}} - \globmin}_2,
	\end{split}
	\end{equation}
	with $\chi_2$ as specified in~\eqref{eq:chi_2}.
\end{corollary}
}

In order to be able to apply Gr\"onwall's inequality to~\eqref{eq:cor:evolution_of_functionalV} \revised{and \eqref{eq:cor:evolution_of_functionalV_lower}} with the aim of obtaining \revised{estimates} of the form~$\CV(\rho_t) \leq \CV(\rho_0)e^{-(1-\vartheta)\chi_1 t}$ \revised{and~$\CV(\rho_t) \geq \CV(\rho_0)e^{-(1-\vartheta/2)\chi_2 t}$ for some~$\chi_1,\chi_2>0$ and a suitable $\vartheta\in(0,1)$,} it remains to control the quantity $\N{\conspoint{\rho_{Y,t}} - \globmin}_2$ through the choice of the parameter~$\alpha$.
This is the content of the next section.

\subsection{Quantitative Laplace Principle} \label{subsec:quant_laplace}

The well-known Laplace principle~\cite{dembo2009large,miller2006applied,pinnau2017consensus} asserts that for any absolutely continuous probability distribution~$\indivmeasure\in\CP(\bbR^d)$ with $\globmin\in\supp(\indivmeasure)$ it holds 
\begin{align} \label{eq:laplace_principle}
	\lim\limits_{\alpha\rightarrow \infty}\left(-\frac{1}{\alpha}\log\N{\omegaa}_{L_1(\indivmeasure)}\right)
	= \CE(\globmin) = \underbar\CE,
\end{align}
which allows to infer that the $\alpha$-weighted measure~$\omegaa/\N{\omegaa}_{L_1(\indivmeasure)}\indivmeasure$ is concentrated in a small region around the minimizer~$\globmin$, provided that $\CE$ attains its minimum at a single point, which is however guaranteed by the inverse continuity property~\ref{asm:icp}.

The asymptotic nature of the result~\eqref{eq:laplace_principle}, however, does not permit to obtain the required quantitative estimates, which is the reason why the authors of~\cite{fornasier2021consensus} proposed a quantitative nonasymptotic variant of the Laplace principle.
In the following proposition, cf.\@~\cite[Proposition~1]{fornasier2021convergence}, we state this result for the setting of anisotropic noise considered throughout the paper.

\begin{proposition}[{\cite[Proposition~1]{fornasier2021convergence}}] \label{lem:laplace_quant}
	Let $\minobj = 0$, $\indivmeasure \in \CP(\bbR^d)$ and fix $\alpha  > 0$. For any $r > 0$ we define $\CE_{r} := \sup_{y \in B^\infty_{r}(\globmin)}\CE(y)$.
	Then, under the inverse continuity property~\ref{asm:icp}, for any $r \in (0,R_0]$ and $q > 0$  such that $q + \CE_{r} \leq \CE_{\infty}$, we have
	\begin{align*}
		\N{\conspoint{\indivmeasure} - \globmin}_2
		\leq \frac{\sqrt{d}(q + \CE_{r})^\nu}{\eta} + \frac{\sqrt{d}\exp(-\alpha q)}{\indivmeasure(B^\infty_{r}(\globmin))}\int\N{y-\globmin}_2d\indivmeasure(y).
	\end{align*}
\end{proposition}

\begin{proof}
	The proof is a mere reformulation of the one of~\cite[Propositon~1]{fornasier2021convergence}, which is presented in what follows for the sake of completeness.
	
	For any $a > 0$, Markov's inequality gives $\Nnormal{\omegaa}_{L_1(\indivmeasure)} \geq a \indivmeasure(\{y : \exp(-\alpha \CE(y)) \geq a\})$.
	By choosing $a = \exp(-\alpha\CE_{r})$ and noting that
	\begin{align*}
		\indivmeasure\left(\left\{y \in \bbR^d: \exp(-\alpha \CE(y)) \geq \exp(-\alpha\CE_{r})\right\}\right) &=
		\indivmeasure\left(\left\{y \in \bbR^d: \CE(y) \leq \CE_{r} \right\}\right) \geq \indivmeasure(B^\infty_{r}(\globmin)),
	\end{align*}
	we get $\N{\omegaa}_{L_1(\indivmeasure)} \geq \exp(-\alpha \CE_{r})\indivmeasure(B^\infty_{r}(\globmin))$.
	Now let $\widetilde r \geq r > 0$.
	With the definition of the consensus point $\conspoint{\indivmeasure} = \int y \omegaa(y)/\Nnormal{\omegaa}_{L_1(\indivmeasure)}\,d\indivmeasure(y)$ and Jensen's inequality we can decompose
	\begin{align*}
	\begin{split}
		\N{\conspoint{\indivmeasure} - \globmin}_\infty
		&\leq \int_{B^\infty_{\widetilde r}(\globmin)} \N{y-\globmin}_\infty \frac{\omegaa(y)}{\N{\omegaa}_{L_1(\indivmeasure)}}\,d\indivmeasure(y) \\
		&\quad\, + \int_{\left(B^\infty_{\widetilde r}(\globmin)\right)^c} \N{y-\globmin}_\infty \frac{\omegaa(y)}{\N{\omegaa}_{L_1(\indivmeasure)}}\,d\indivmeasure(y).
	\end{split}
	\end{align*}
	After noticing that the first term is bounded by $\widetilde r$ since $ \N{y-\globmin}_\infty \leq \widetilde r$ for all $y \in B^\infty_{\widetilde r}(\globmin)$ we can continue the former with
	\begin{equation} \label{eq:aux_laplace_1}
	\begin{split}
		\N{\conspoint{\indivmeasure} - \globmin}_\infty
		&\leq \widetilde r + \frac{1}{ \exp(-\alpha \CE_{r})\indivmeasure(B^\infty_{r}(\globmin))}\!\int_{(B^\infty_{\widetilde r}(\globmin))^c} \N{y-\globmin}_\infty \omegaa(y)\,d\indivmeasure(y)\\
		&\leq \widetilde r + \frac{\exp\left(-\alpha \inf_{y \in (B^\infty_{\widetilde r}(\globmin))^c}\CE(y)\right)}{\exp\left(-\alpha \CE_{r})\indivmeasure(B^\infty_{r}(\globmin)\right)}\!\int_{(B^\infty_{\widetilde r}(\globmin))^c}\N{y-\globmin}_\infty d\indivmeasure(y)\\
		&= \widetilde r + \frac{\exp\left(-\alpha \left(\inf_{y \in (B^\infty_{\widetilde r}(\globmin))^c}\CE(y) - \CE_{r}\right)\right)}{\indivmeasure(B^\infty_{r}(\globmin))}\!\int\N{y-\globmin}_\infty d\indivmeasure(y),
	\end{split}
	\end{equation}
	where for the second term we used $\N{\omegaa}_{L_1(\indivmeasure)} \geq \exp(-\alpha \CE_{r})\indivmeasure(B^\infty_{r}(\globmin))$ from above.
	Let us now choose $\widetilde r = (q+\CE_r)^{\nu}/\eta$, which satisfies $\widetilde r  \geq r$, since \ref{asm:icp} with $\minobj = 0$ and $r \leq R_0$ implies
	\begin{align*}
		\widetilde r
		= \frac{(q+\CE_r)^{\nu}}{\eta} \geq \frac{\CE_r^{\nu}}{\eta}
		= \frac{\left(\sup_{y \in B^\infty_{r}(\globmin)}\CE(y)\right)^{\nu}}{\eta} \geq \sup_{y \in B_{r}^\infty(\globmin)}\N{y-\globmin}_\infty
		= r.
	\end{align*}
	Furthermore, due to the assumption $q+\CE_r \leq \CE_{\infty}$ in the statement we have $\widetilde r \leq \CE_{\infty}^{\nu}/\eta$, which together with the two cases of \ref{asm:icp} with $\minobj = 0$ allows to bound the infimum in \eqref{eq:aux_laplace_1} as follows 
	\begin{align*}
		\inf_{y \in (B^\infty_{\widetilde r}(\globmin))^c}\CE(y) - \CE_r
		\geq \min\left\{\CE_{\infty}, (\eta \widetilde r)^{\frac{1}{\nu}}\right\} - \CE_r
		= (\eta \widetilde r)^{\frac{1}{\nu}} - \CE_r
		= (q + \CE_r) - \CE_r
		= q.
	\end{align*}
	Inserting this and the definition of $\widetilde r$ into \eqref{eq:aux_laplace_1}, we get the result as a consequence of the norm equivalence $\N{\,\cdot\,}_\infty\leq\N{\,\cdot\,}_2\leq\sqrt{d}\N{\,\cdot\,}_\infty$.
\end{proof}

To eventually apply Proposition~\ref{lem:laplace_quant} in the setting of Corollary~\ref{cor:evolution_of_functionalV}, i.e., to upper bound the distance of the consensus point~$\conspoint{\rho_{Y,t}}$ to the global minimizer~$\globmin$, it remains to ensure that $\rho_{Y,t}(B^\infty_{r}(\globmin))$ is bounded away from $0$ for a finite time horizon.
We ensure that this is indeed the case in what follows.

\subsection{A Lower Bound for the Probability Mass~$\rho_{Y,t}(B^\infty_{r}(\globmin))$} \label{subsec:lower_bound_prob_mass}

In this section, for any small radius $r > 0$, we provide a lower bound on the probability mass of $\rho_{Y,t}(B^\infty_{r}(\globmin))$ by defining a mollifier $\phi_r : \bbR^d\times\bbR^d \rightarrow \bbR$ so that
\begin{equation*}
	\rho_{Y,t}(B^\infty_{r}(\globmin))
	= \rho_{t}(\bbR^d\times B^\infty_{r}(\globmin))
	= \iint_{\bbR^d\times B^\infty_{r}(\globmin)} 1 \, d\rho_{t}(x,y)
	\geq \iint \phi_r(x,y) \, d\rho_t(x,y)
\end{equation*}
and studying the evolution of the right-hand side.

\begin{lemma} \label{lem:properties_mollifier}
	For $r > 0$ let $\Omega_r := \{(x,y)\in\bbR^d\times\bbR^d: \max\{\N{x-\globmin}_\infty, \N{x-y}_\infty\}< r/2\}$ and define the mollifier $\phi_r : \bbR^d\times\bbR^d \rightarrow \bbR$ by
	\begin{equation*} 
		\phi_{r}(x,y) := \begin{cases}
			\prod_{k=1}^d \exp\left(1-\frac{\left(\frac{r}{2}\right)^2}{\left(\frac{r}{2}\right)^2-\left(x-\globmin\right)_k^2}\right)\exp\left(1-\frac{\left(\frac{r}{2}\right)^2}{\left(\frac{r}{2}\right)^2-\left(x-y\right)_k^2}\right),& \text{ if } (x,y)\in\Omega_r,\\
			0,& \text{ else.}
		\end{cases}
	\end{equation*}
	We have that $\Im(\phi_r) = [0,1]$, $\supp(\phi_r) = \Omega_r \subset B^\infty_{r/2}(\globmin) \times B^\infty_{r}(\globmin) \subset \bbR^d \times B^\infty_{r}(\globmin)$, $\phi_{r} \in \CC_c^{\infty}(\bbR^d\times\bbR^d)$ and
	\begin{align*}
		\partial_{x_k} \phi_{r}(x,y)
			&= -\frac{r^2}{2}
				\left(
				\frac{\left(x-\globmin\right)_k}{\left(\left(\frac{r}{2}\right)^2-\left(x-\globmin\right)^2_k\right)^2}
				+\frac{\left(x-y\right)_k}{\left(\left(\frac{r}{2}\right)^2-\left(x-y\right)^2_k\right)^2}
				\right)
				\phi_{r}(x,y),\\
		\partial_{y_k} \phi_{r}(x,y)
			&= -\frac{r^2}{2}
				\frac{\left(y-x\right)_k}{\left(\left(\frac{r}{2}\right)^2-\left(x-y\right)^2_k\right)^2}
				\phi_{r}(x,y),\\
		\partial^2_{x_kx_k} \phi_{r}(x,y)
			&= \frac{r^2}{2}
				\left(
				\left(\frac{2\left(2\left(x-\globmin\right)^2_k-\left(\frac{r}{2}\right)^2\right)\left(x-\globmin\right)_k^2 - \left(\left(\frac{r}{2}\right)^2-\left(x-\globmin\right)^2_k\right)^2}{\left(\left(\frac{r}{2}\right)^2-\left(x-\globmin\right)^2_k\right)^4}\right)\right.\\
				&\qquad+\left.\left(\frac{2\left(2\left(x-y\right)^2_k-\left(\frac{r}{2}\right)^2\right)\left(x-y\right)_k^2 - \left(\left(\frac{r}{2}\right)^2-\left(x-y\right)^2_k\right)^2}{\left(\left(\frac{r}{2}\right)^2-\left(x-y\right)^2_k\right)^4}\right)
				\right)
				\phi_{r}(x,y).
	\end{align*}
\end{lemma}

\begin{proof}
	It is straightforward to check the properties of $\phi_r$ as it is a tensor product of classical well-studied mollifiers.
\end{proof}

\noindent
To keep the notation as concise as possible in what follows, let us introduce the decomposition
\begin{equation} \label{eq:derivatives_abbreviations}
	\partial_{x_k} \phi_{r} = \delta^{*}_{x_k} \phi_{r} + \delta^Y_{x_k} \phi_{r}
	\quad\text{and}\quad
	\partial^2_{x_kx_k} \phi_{r} = \delta^{2,*}_{x_kx_k} \phi_{r} + \delta^{2,Y}_{x_kx_k} \phi_{r},
\end{equation}
where
\begin{equation*}
	\delta^{*}_{x_k} \phi_{r}(x,y) \!=\! 
			\frac{-\frac{r^2}{2}\left(x\!-\!\globmin\right)_k}{\left(\!\left(\frac{r}{2}\right)^2\!\!-\!\left(x\!-\!\globmin\right)^2_k\right)^{\!2}}
			\phi_{r}(x,y)
	\ \ \text{and} \ \
	\delta^Y_{x_k} \phi_{r}(x,y) \!=\! 
			\frac{-\frac{r^2}{2}\left(x\!-\!y\right)_k}{\left(\!\left(\frac{r}{2}\right)^2\!\!-\!\left(x\!-\!y\right)^2_k\right)^{\!2}}
			\phi_{r}(x,y)
\end{equation*}
and analogously for $\delta^{2,*}_{x_kx_k} \phi_{r}$ and $\delta^{2,Y}_{x_kx_k} \phi_{r}$.

\begin{proposition} \label{lem:lower_bound_probability_memoryCBO}
	Let $T > 0$, $r > 0$, and fix parameters $\alpha,\lambda_1,\sigma_1 > 0$ as well as parameters $\lambda_2,\sigma_2,\lambda_3,\sigma_3,\beta,\kappa,\theta \geq 0$
	such that $\sigma_2>0$ iff $\lambda_2\not=0$ and $\sigma_3>0$ iff $\lambda_3\not=0$.
	Moreover assume the validity of Assumption~\ref{asm:Lipschitz_gradient} if $\lambda_3\not=0$.
	Let $\rho\in\CC([0,T],\CP(\bbR^d\times\bbR^d))$ weakly solve the Fokker-Planck equation~\eqref{eq:CBO_with_memory_weak} in the sense of Definition~\ref{def:fokker_planck_weak_sense} with initial condition $\rho_0 \in \CP(\bbR^d\times\bbR^d)$ and for $t \in [0,T]$.
	Then, for all $t\in[0,T]$ we have
	\begin{align} \label{eq:lower_bound_probability_rate}
		\rho_{Y,t}(B^\infty_{r}(\globmin))
		&\geq \left(\iint\phi_{r}(x,y)\,d\rho_0(x,y)\right)\exp\left(-pt\right)
	\end{align}
	with
	\begin{align} \label{eq:def_p_memory}
		p &:= d\sum_{i=1}^3\omega_i\left(\left(1\!+\!\mathbbm{1}_{i\in\{1,3\}}\right)\left(\frac{2\lambda_i C_\Upsilon\sqrt{c}}{(1\!-\!c)^2\frac{r}{2}} \!+\! \frac{\sigma_i^2C_\Upsilon^2}{(1\!-\!c)^4\!\left(\frac{r}{2}\right)^2} \!+\! \frac{4\lambda_i^2}{\tilde{c}\sigma_i^2}\right) \!+\! \mathbbm{1}_{i=2}\frac{\sigma_2^2c}{(1\!-\!c)^4} \right),
	\end{align}
	where, for any $B<\infty$ with $\sup_{t \in [0,T]}\N{\conspoint{\rho_{Y,t}}-v^*}_2 \leq B$, $C_\Upsilon = C_\Upsilon(r,B,d,C_{\nabla\CE})$ is as defined in \eqref{eq:C_Upsilon}.
	Moreover, $\omega_i=\mathbbm{1}_{\lambda_i>0}$ for $i\in\{1,2,3\}$ and $c \in (1/2,1)$ can be any constant that satisfies $(1-c)^2 \leq (2c-1)c$.
\end{proposition}

\begin{remark}
	In order to ensure a finite decay rate $p<\infty$ in Proposition~\ref{lem:lower_bound_probability_memoryCBO} it is crucial to have non-vanishing diffusions $\sigma_1>0$, $\sigma_2>0$ if $\lambda_2\not=0$ and $\sigma_3>0$ if $\lambda_3\not=0$.
	As apparent from the formulation of the statement as well as the proof below, $\sigma_2$ or $\sigma_3$ may be $0$ if the corresponding drift parameter, $\lambda_2$ or $\lambda_3$, respectively, vanishes.
\end{remark}

\begin{proof}[Proof of Proposition~\ref{lem:lower_bound_probability_memoryCBO}]
	By the definition of the marginal $\rho_Y$ and the properties of the mollifier~$\phi_r$ defined in Lemma~\ref{lem:properties_mollifier} we have
	\begin{align*}
		\rho_{Y,t}(B^\infty_{r}(\globmin))
		= \rho_t\big(\bbR^d \times B^\infty_{r}(\globmin)\big)
		\geq \rho_t(\Omega_r)
		\geq \iint \phi_{r}(x,y)\,d\rho_t(x,y).
	\end{align*}
	Our strategy is to derive a lower bound for the right-hand side of this inequality.
	Using the weak solution property of $\rho$ as in Definition~\ref{def:fokker_planck_weak_sense} and the fact that  $\phi_{r}\in \CC^{\infty}_c(\bbR^d\times\bbR^d)$, we obtain
	\begin{equation} \label{eq:initial_evolution}
	\begin{split}
		&\frac{d}{dt}\iint\phi_{r}(x,y)\,d\rho_t(x,y)
		=\sum_{k=1}^d \iint T^s_{k}(x,y) \,d\rho_t(x,y)\\
		&\quad +\! \sum_{k=1}^d \iint \big(T^c_{1k}(x,y) \!+\! T^c_{2k}(x,y) \!+\! T^\ell_{1k}(x,y) \!+\! T^\ell_{2k}(x,y) \!+\! T^g_{1k}(x,y) \!+\! T^g_{2k}(x,y)\big) \,d\rho_t(x,y),
	\end{split}
	\end{equation}
	where $T^s_{k}(x,y) := -\kappa S^{\beta,\theta}(x,y) \left(y-x\right)_k \partial_{y_k}\phi_r(x,y)$ and
	\begin{equation}
	\begin{aligned}
		&\, T^c_{1k}(x,y)
			\!:= \!-\lambda_1\left(x\!-\!\conspoint{\rho_{Y,t}}\right)_k \partial_{x_k}\phi_r(x,y), 
		&\!\!\!\, T^c_{2k}(x,y)
			&\!:= \!\frac{\sigma_1^2}{2} \left(x\!-\!\conspoint{\rho_{Y,t}}\right)_k^2 \partial^2_{x_kx_k} \phi_{r}(x,y), \nonumber \\
		&\, T^\ell_{1k}(x,y)
			\!:= \!-\lambda_2\left(x\!-\!y\right)_k \partial_{x_k}\phi_r(x,y),
		&\!\!\!\, T^\ell_{2k}(x,y)
			&\!:= \!\frac{\sigma_2^2}{2} \left(x\!-\!y\right)_k^2 \partial^2_{x_kx_k} \phi_{r}(x,y), \nonumber \\
		&\, T^g_{1k}(x,y)
			\!:= -\lambda_3\partial_{x_k}\CE(x) \partial_{x_k}\phi_r(x,y), 
		&\!\!\!\, T^g_{2k}(x,y)
			&\!:= \!\frac{\sigma_3^2}{2} \left(\partial_{x_k}\CE(x)\right)^2 \partial^2_{x_kx_k} \phi_{r}(x,y)\nonumber
	\end{aligned}
	\end{equation}
	for $k\in\{1,\dots,d\}$.
	Since the mollifier~$\phi_r$ and its derivatives vanish outside of $\Omega_r$ we restrict our attention to $\Omega_r$ and aim for showing for all $k\in\{1,\dots,d\}$ that
	\begin{itemize}[leftmargin=*,labelindent=5ex,labelsep=3ex,topsep=1ex]
		\item $T^s_{k}(x,y) \geq 0$,
		\item $T^c_{1k}(x,y) + T^c_{2k}(x,y) \geq -p^c\phi_r(x,y)$,
		\item $T^\ell_{1k}(x,y) + T^\ell_{2k}(x,y) \geq -p^\ell\phi_r(x,y)$,
		\item $T^g_{1k}(x,y) + T^g_{2k}(x,y) \geq -p^g\phi_r(x,y)$
	\end{itemize}
	pointwise for all $(x,y)\in\Omega_r$ with suitable constants~$0 \leq p^\ell, p^c, p^g < \infty$.
	
	\noindent
	\textbf{Term $T^s_{k}$:}
	Using the expression for $\partial_{y_k} \phi_{r}$ from Lemma~\ref{lem:properties_mollifier} and the fact that $S^{\beta,\theta}\geq\theta/2\geq0$ it is easy to see that
	\begin{equation} \label{eq:Ts_lowerbound}
		T^s_{k}(x,y)
		= \frac{r^2\kappa}{2} S^{\beta,\theta}(x,y) \frac{\left(y-x\right)^2_k}{\left(\left(\frac{r}{2}\right)^2-\left(x-y\right)^2_k\right)^2}\phi_{r}(x,y)
		\geq 0.
	\end{equation}

	\noindent
	\textbf{Terms $T^c_{1k}+T^c_{2k}$, $T^\ell_{1k}+T^\ell_{2k}$ and $T^g_{1k}+T^g_{2k}$:}
	We first note that the third inequality from above holds with $p^\ell=0$ if $\lambda_2=\sigma_2=0$ and the fourth with $p^g=0$ if $\lambda_3=\sigma_3=0$.
	
	Therefore, in what follows we assume that $\lambda_2,\sigma_2,\lambda_3,\sigma_3>0$.
	In order to lower bound the three terms from above, we arrange the summands by using the abbreviations introduced in~\eqref{eq:derivatives_abbreviations} as follows.
	For $T^c_{1k}+T^c_{2k}$ we have
	\begin{subequations}
	\begin{align}
		&T^c_{1k}(x,y)+T^c_{2k}(x,y)\nonumber\\
		&\qquad\,= -\lambda_1\left(x-\conspoint{\rho_{Y,t}}\right)_k \delta^{*}_{x_k} \phi_{r}(x,y) + \frac{\sigma_1^2}{2} \left(x-\conspoint{\rho_{Y,t}}\right)_k^2 \delta^{2,*}_{x_kx_k} \phi_{r}(x,y) \label{eq:Tca} \\
			&\qquad\quad\,\,-\lambda_1\left(x-\conspoint{\rho_{Y,t}}\right)_k \delta^Y_{x_k} \phi_{r}(x,y) + \frac{\sigma_1^2}{2} \left(x-\conspoint{\rho_{Y,t}}\right)_k^2 \delta^{2,Y}_{x_kx_k} \phi_{r}(x,y), \label{eq:Tcb}
	\end{align}
	\end{subequations}
	for $T^\ell_{1k}+T^\ell_{2k}$ we have
	\begin{subequations}
	\begin{align}
		&T^\ell_{1k}(x,y)+T^\ell_{2k}(x,y)\nonumber\\
		&\qquad\,=-\lambda_2\left(x-y\right)_k \delta^{*}_{x_k} \phi_{r}(x,y) + \frac{\sigma_2^2}{2} \left(x-y\right)_k^2 \delta^{2,*}_{x_kx_k} \phi_{r}(x,y) \label{eq:Tla} \\
			&\qquad\quad\,\,-\lambda_2\left(x-y\right)_k \delta^Y_{x_k} \phi_{r}(x,y) + \frac{\sigma_2^2}{2} \left(x-y\right)_k^2 \delta^{2,Y}_{x_kx_k} \phi_{r}(x,y) \label{eq:Tlb}
	\end{align}
	\end{subequations}
	and for $T^g_{1k}+T^g_{2k}$ we have
	\begin{subequations}
	\begin{align}
		&T^g_{1k}(x,y)+T^g_{2k}(x,y)\nonumber\\
		&\qquad\,=-\lambda_3\partial_{x_k}\CE(x) \delta^{*}_{x_k} \phi_{r}(x,y) + \frac{\sigma_3^2}{2} (\partial_{x_k}\CE(x))^2 \delta^{2,*}_{x_kx_k} \phi_{r}(x,y) \label{eq:Tga} \\
			&\qquad\quad\,\,-\lambda_3\partial_{x_k}\CE(x) \delta^Y_{x_k} \phi_{r}(x,y) + \frac{\sigma_3^2}{2} (\partial_{x_k}\CE(x))^2 \delta^{2,Y}_{x_kx_k} \phi_{r}(x,y). \label{eq:Tgb}
	\end{align}
	\end{subequations}
	We now treat each of the two-part sums in~\eqref{eq:Tca}, \eqref{eq:Tcb}, \eqref{eq:Tla}, \eqref{eq:Tlb}, \eqref{eq:Tga} and \eqref{eq:Tgb} separately by employing a technique similar to the one used in the proof of~\cite[Proposition~2]{fornasier2021convergence}, which was developed originally to prove~\cite[Proposition~20]{fornasier2021consensus}.

	\noindent
	\textbf{Terms~\eqref{eq:Tca}, \eqref{eq:Tla} and \eqref{eq:Tga}:}
	Owed to their similar structure (in particular with respect to the denominator of the derivatives~$\delta^{*}_{x_k} \phi_{r}$ and $\delta^{2,*}_{x_kx_k} \phi_{r}$) we can treat the three sums~\eqref{eq:Tca}, \eqref{eq:Tla} and \eqref{eq:Tga} simultaneously.
	Therefore we consider the general formulation
	\begin{equation} \label{eq:TlaTga_}
		-\lambda\Upsilon_k(x,y) \delta^{*}_{x_k} \phi_{r}(x,y) + \frac{\sigma^2}{2} \Upsilon_k^2(x,y) \delta^{2,*}_{x_kx_k} \phi_{r}(x,y) =: T^*_{1k}(x,y)+T^*_{2k}(x,y),
	\end{equation}
	which matches
		\eqref{eq:Tca} when $\Upsilon_k(x,y) = (x-\conspoint{\rho_{Y,t}})_k$, $\lambda=\lambda_1$ and $\sigma=\sigma_1$,
		\eqref{eq:Tla} when $\Upsilon_k(x,y) = (x-y)_k$, $\lambda=\lambda_2$ and $\sigma=\sigma_2$,
		and \eqref{eq:Tga} when $\Upsilon_k(x,y) = \partial_{x_k}\CE(x)$, $\lambda=\lambda_3$ and $\sigma=\sigma_3$.
	
	To achieve the desired lower bound over $\Omega_r$, we introduce the subsets
	\begin{align*}
		K^*_{1k} &:= \left\{(x,y) \in \bbR^d\times\bbR^d : \abs{\left(x-\globmin\right)_k} > \frac{\sqrt{c}}{2}r\right\}
	\end{align*}
	and
	\begin{align*}
	\begin{aligned}
		&K^*_{2k} := \Bigg\{(x,y) \in \bbR^d\times\bbR^d : -\lambda\Upsilon_k(x,y) \left(x-\globmin\right)_k \left(\left(\frac{r}{2}\right)^2-\left(x-\globmin\right)_k^2\right)^2 \\
		&\qquad\qquad\qquad\qquad\qquad\qquad\qquad > \tilde{c}\left(\frac{r}{2}\right)^2\frac{\sigma^2}{2} \Upsilon_k^2(x,y)\left(x-\globmin\right)_k^2\Bigg\},
	\end{aligned}
	\end{align*}
	where $\tilde{c} := 2c-1\in(0,1)$.
	For fixed $k$ we now decompose $\Omega_r$ according to
	\begin{align*} 
		\Omega_r = \big((K_{1k}^{*})^c \cap \Omega_r\big) \cup \big(K^*_{1k} \cap (K_{2k}^{*})^c \cap \Omega_r\big) \cup \big(K^*_{1k} \cap K^*_{2k} \cap \Omega_r\big).
	\end{align*}

	\noindent
	In the following we treat each of these three subsets separately.

	\noindent
	\textit{Subset $(K_{1k}^{*})^c \cap \Omega_r$:}
	We have $\abs{\left(x-\globmin\right)_k} \leq \frac{\sqrt{c}}{2}r$ for each $(x,y) \in (K_{1k}^*)^c$, which can be used to independently derive lower bounds for both summands in~\eqref{eq:TlaTga_}.
	For the first we insert the expression for $\delta^{*}_{x_k} \phi_{r}(x,y)$ to get
	\begin{equation} \label{eq:T*_set1_1_lowerbound}
	\begin{aligned}
		T^*_{1k}(x,y)
		&= \frac{r^2}{2}\lambda\Upsilon_k(x,y) \frac{\left(x-\globmin\right)_k}{\left(\left(\frac{r}{2}\right)^2-\left(x-\globmin\right)^2_k\right)^2}\phi_{r}(x,y) \\
		&\geq -\frac{r^2}{2} \!\lambda \frac{\abs{\Upsilon_k(x,y)}\!\abs{\left(x-\globmin\right)_k}}{\left(\left(\frac{r}{2}\right)^2-\left(x-\globmin\right)^2_k\right)^2}\phi_r(x,y)
		\geq - \frac{2\lambda C_\Upsilon\sqrt{c}}{(1-c)^2\frac{r}{2}}\phi_r(x,y) \\
		&=: -p^{*,\Upsilon}_{1}\phi_r(x,y),
	\end{aligned}
	\end{equation}
	where, in the last inequality, we used that $(x,y)\in\Omega_r$, the definition of $B$ and Assumption~\ref{asm:Lipschitz_gradient} to get the bound
	\begin{equation} \label{eq:C_Upsilon}
	\begin{split}
		\abs{\Upsilon_k(x,y)} 
		&= \begin{cases}
			\abs{(x-\conspoint{\rho_{Y,t}})_k} 
					\leq \frac{r}{2}+B,
				&\text{if }\Upsilon_k(x,y) = (x-\conspoint{\rho_{Y,t}})_k, \\
			\abs{(x-y)_k} \leq \frac{r}{2},
				&\text{if }\Upsilon_k(x,y) = (x-y)_k, \\
			\abs{\partial_{x_k}\CE(x)} \leq \N{\nabla\CE(x)}_2 \leq C_{\nabla\CE}\N{x-\globmin}_2\\
					\qquad\leq C_{\nabla\CE}d\N{x-\globmin}_\infty \leq C_{\nabla\CE}d\frac{r}{2}
				&\text{if }\Upsilon_k(x,y) = \partial_{x_k}\CE(x).
		\end{cases}\\
		&\leq \max\left\{\frac{r}{2}+B, C_{\nabla\CE}d\frac{r}{2}\right\}
		=: C_\Upsilon(r,B,d,C_{\nabla\CE}).
	\end{split}
	\end{equation}
	For the second summand we insert the expression for $\delta^{2,*}_{x_kx_k} \phi_{r}(x,y)$ to obtain
	\begin{equation} \label{eq:T*_set1_2_lowerbound}
	\begin{aligned}
		T^*_{2k}(x,y)
		&=\frac{\sigma^2}{2} \Upsilon_k^2(x,y) \delta^{2,*}_{x_kx_k} \phi_{r}(x,y)\\
		&=\sigma^2 \left(\frac{r}{2}\right)^2 \Upsilon_k^2(x,y) \frac{2\left(2\left(x\!-\!\globmin\right)^2_k\!-\!\left(\frac{r}{2}\right)^2\right)\left(x\!-\!\globmin\right)_k^2 \!-\! \left(\left(\frac{r}{2}\right)^2\!-\!\left(x\!-\!\globmin\right)^2_k\right)^2}{\left(\left(\frac{r}{2}\right)^2-\left(x\!-\!\globmin\right)^2_k\right)^4}
			\phi_{r}(x,y)\\
		&\geq - \frac{\sigma^2C_\Upsilon^2}{(1-c)^4\left(\frac{r}{2}\right)^2}\phi_{r}(x,y) =: -p^{*,\Upsilon}_{2}\phi_r(x,y),
	\end{aligned}
	\end{equation}
	where the last inequality uses $\Upsilon_k^2(x,y) \leq C_\Upsilon^2$.

	\noindent
	\textit{Subset $K^*_{1k} \cap (K_{2k}^{*})^c \cap \Omega_r$:}
	As $(x,y) \in K^*_{1k}$ we have $\abs{\left(x-\globmin\right)_k} > \frac{\sqrt{c}}{2}r$.
	We observe that the sum in~\eqref{eq:TlaTga_} is nonnegative for all $(x,y)$ in this subset whenever
	\begin{equation} \label{eq:aux_term_3_*}
	\begin{aligned}
		&\left(-\lambda\Upsilon_k(x,y)\left(x-\globmin\right)_k + \frac{\sigma^2}{2} \Upsilon_k^2(x,y)\right)\left(\left(\frac{r}{2}\right)^2-\left(x-\globmin\right)^2_k\right)^2 \\
		&\qquad\qquad\qquad\qquad\qquad\, \leq \sigma^2\Upsilon_k^2(x,y) \left(2\left(x-\globmin\right)^2_k-\left(\frac{r}{2}\right)^2\right)\left(x-\globmin\right)_k^2.
	\end{aligned}
	\end{equation}
	The first term on the left-hand side in \eqref{eq:aux_term_3_*} can be bounded from above by exploiting that $v\in (K_{2k}^{*})^c$ and by using the relation $\tilde{c} = 2c-1$.
	More precisely, we have
	\begin{align*}
		&-\lambda\Upsilon_k(x,y)\left(x-\globmin\right)_k \left(\left(\frac{r}{2}\right)^2-\left(x-\globmin\right)^2_k\right)^2
		\leq \tilde{c}\left(\frac{r}{2}\right)^2\frac{\sigma^2}{2} \Upsilon_k^2(x,y)\left(x-\globmin\right)_k^2\\
		&\quad\;\; = (2c\!-\!1)\left(\frac{r}{2}\right)^2\!\frac{\sigma^2}{2} \Upsilon_k^2(x,y)\left(x-\globmin\right)_k^2
		\leq\! \left(2\left(x-\globmin\right)_k^2-\left(\frac{r}{2}\right)^2\right)\!\frac{\sigma^2}{2} \Upsilon_k^2(x,y)\left(x-\globmin\right)_k^2,
	\end{align*}
	where the last inequality follows since $v\in K^*_{1k}$.
	For the second term on the left-hand side in \eqref{eq:aux_term_3_*} we can use $(1-c)^2 \leq (2c-1)c$ as per assumption, to get
	\begin{align*}
		&\frac{\sigma^2}{2} \Upsilon_k^2(x,y)\left(\left(\frac{r}{2}\right)^2-\left(x-\globmin\right)^2_k\right)^2
		\leq \frac{\sigma^2}{2} \Upsilon_k^2(x,y) (1-c)^2\left(\frac{r}{2}\right)^4 \\
		&\qquad \leq \frac{\sigma^2}{2} \Upsilon_k^2(x,y) (2c-1)\left(\frac{r}{2}\right)^2c\left(\frac{r}{2}\right)^2
		\leq \frac{\sigma^2}{2} \Upsilon_k^2(x,y) \left(2\left(x-\globmin\right)_k^2-\left(\frac{r}{2}\right)^2\right)\left(x-\globmin\right)_k^2.
	\end{align*}
	Hence, \eqref{eq:aux_term_3_*} holds and we have that \eqref{eq:TlaTga_} is uniformly nonnegative on this subset.

	\noindent
	\textit{Subset $K^*_{1k} \cap K^*_{2k} \cap \Omega_r$:}
	As $(x,y) \in K_{1k}^*$ we have $\abs{\left(x-\globmin\right)_k} > \frac{\sqrt{c}}{2}r$.
	To start with we note that the first summand of~\eqref{eq:TlaTga_} vanishes whenever $\sigma^2\Upsilon_k^2(x,y) = 0$, provided $\sigma>0$, so nothing needs to be done if $\Upsilon_k(x,y)=0$.
	Otherwise, if $\sigma^2\Upsilon_k^2(x,y) > 0$, we exploit $(x,y)\in K_{2k}^*$ to get
	\begin{equation*}
	\begin{split}
		\frac{\Upsilon_k(x,y)\left(x-\globmin\right)_k}{\left(\left(\frac{r}{2}\right)^2-\left(x-\globmin\right)_k^2\right)^2}
		&\geq\frac{-\abs{\Upsilon_k(x,y)}\abs{\left(x-\globmin\right)_k}}{\left(\left(\frac{r}{2}\right)^2-\left(x-\globmin\right)_k^2\right)^2}\\
		&> \frac{2\lambda\Upsilon_k(x,y) \left(x-\globmin\right)_k}{\tilde{c}\left(\frac{r}{2}\right)^2\sigma^2\abs{\Upsilon_k(x,y)}\abs{\left(x-\globmin\right)_k}}
		\geq -\frac{8\lambda}{\tilde{c}r^2\sigma^2}.
	\end{split}
	\end{equation*}
	Using this, the first summand of~\eqref{eq:TlaTga_} can be bounded from below by
	\begin{equation} \label{eq:T*_set3_1_lowerbound}
	\begin{split}
		T^*_{1k}(x,y)
		=\lambda\frac{r^2}{2}\frac{\Upsilon_k(x,y)\left(x-\globmin\right)_k}{\left(\left(\frac{r}{2}\right)^2-\left(x-\globmin\right)^2_k\right)^2}\phi_{r}(x,y)
		&\geq -\frac{4\lambda^2}{\tilde{c}\sigma^2}\phi_{r}(x,y) =: -p^{*,\Upsilon}_{3}\phi_r(x,y).
	\end{split}
	\end{equation}
	For the second summand, the nonnegativity of $\sigma^2\Upsilon_k^2(x,y)$ implies the nonnegativity, whenever
	\begin{equation*}
		2\left(2\left(x-\globmin\right)^2_k-\left(\frac{r}{2}\right)^2\right)\left(x-\globmin\right)_k^2 \geq \left(\left(\frac{r}{2}\right)^2-\left(x-\globmin\right)^2_k\right)^2.
	\end{equation*}
	This holds for $v\in K_{1k}^*$, if $2(2c-1)c \geq (1-c)^2$ as implied by the assumption.

	\noindent
	\textbf{Term~\eqref{eq:Tlb}:}
	Recall that this term has the structure
	\begin{equation}\label{eq:Tlb_}
		-\lambda_2\left(x-y\right)_k \delta^Y_{x_k} \phi_{r}(x,y) + \frac{\sigma_2^2}{2} \left(x-y\right)_k^2 \delta^{2,Y}_{x_kx_k} \phi_{r}(x,y)=:T^{Y,1}_{1k}(x,y)+T^{Y,1}_{2k}(x,y).
	\end{equation}
	We first note that the first summand of~\eqref{eq:Tlb_} is always nonnegative since
	\begin{equation} \label{eq:TY1_set1_1_lowerbound}
		T^{Y,1}_{1k}(x,y)
		= \lambda_2 \frac{r^2}{2}\frac{\left(x-y\right)_k^2}{\left(\left(\frac{r}{2}\right)^2-\left(x-y\right)^2_k\right)^2}\phi_{r}(x,y)
		\geq0.
	\end{equation}
	For the second summand of~\eqref{eq:Tlb_} a direct computation shows
	\begin{equation*}
	\begin{split}
		T^{Y,1}_{2k}(x,y)
		= \sigma_2^2 \left(\frac{r}{2}\right)^2 \left(x-y\right)_k^2 \frac{3\left(x-y\right)^4_k-\left(\frac{r}{2}\right)^4}{\left(\left(\frac{r}{2}\right)^2-\left(x-y\right)^2_k\right)^4}
			\phi_{r}(x,y),
	\end{split}
	\end{equation*}
	which is nonnegative on the set 
	\begin{align*}
		K^Y_k := \left\{(x,y) \in \bbR^d\times\bbR^d : \abs{\left(x-y\right)_k} > \frac{\sqrt{c}}{2}r\right\}
	\end{align*}
	for any $c\geq1/\sqrt{3}$, as ensured by $(1-c)^2 \leq (2c-1)c$.
	On the complement~$(K^Y_k)^c$ we have $\abs{\left(x-y\right)_k} \leq \frac{\sqrt{c}}{2}r$, which can be used to bound
	\begin{equation} \label{eq:TY1_set2_2_lowerbound}
	\begin{aligned}
		T^{Y,1}_{2k}(x,y)
		&=\sigma_2^2 \left(\frac{r}{2}\right)^2 \!\left(x\!-\!y\right)_k^2 \frac{3\left(x\!-\!y\right)^4_k\!-\!\left(\frac{r}{2}\right)^4}{\left(\left(\frac{r}{2}\right)^2\!-\!\left(x\!-\!y\right)^2_k\right)^4}\phi_{r}(x,y)\\
		&\geq-\frac{\sigma_2^2 c}{\left(1\!-\!c\right)^4}\phi_{r}(x,y)
		=: -p^{Y,\Upsilon_\ell}\phi_r(x,y).
	\end{aligned}
	\end{equation}

	\noindent
	\textbf{Terms~\eqref{eq:Tcb} and~\eqref{eq:Tgb}:}
	The final two terms to be controlled have again a similar structure of the form
	\begin{equation} \label{eq:Tgb_}
		-\lambda \Upsilon_k(x,y) \delta^Y_{x_k} \phi_{r}(x,y) + \frac{\sigma^2}{2} \Upsilon_k^2(x,y) \delta^{2,Y}_{x_kx_k} \phi_{r}(x,y)=:T^{Y,2}_{1k}(x,y)+T^{Y,2}_{2k}(x,y),
	\end{equation}
	where we recycle the notation introduced after~\eqref{eq:TlaTga_}, i.e.,
		$\Upsilon_k(x,y) = (x-\conspoint{\rho_{Y,t}})_k$, $\lambda=\lambda_1$ and $\sigma=\sigma_1$ in the case of \eqref{eq:Tcb}
		and $\Upsilon_k(x,y) = \partial_{x_k}\CE(x)$, $\lambda=\lambda_3$ and $\sigma=\sigma_3$ in the case of \eqref{eq:Tgb}.
	
	The procedure for deriving lower bounds is similar to the one at the beginning with the exception that the denominator of the derivatives~$\delta^Y_{x_k} \phi_{r}$ and $\delta^{2,Y}_{x_kx_k} \phi_{r}$ requires to introduce an adapted decomposition of $\Omega_r$.
	To be more specific, we define the subsets
	\begin{align*}
		K^Y_{1k} &:= \left\{(x,y) \in \bbR^d\times\bbR^d : \abs{\left(x-y\right)_k} > \frac{\sqrt{c}}{2}r\right\}
	\end{align*}
	and
	\begin{align*}
	\begin{aligned}
		&K^Y_{2k} := \Bigg\{(x,y) \in \bbR^d\times\bbR^d : -\lambda\Upsilon_k(x,y) \left(x-y\right)_k \left(\left(\frac{r}{2}\right)^2-\left(x-y\right)_k^2\right)^2 \\
		&\qquad\qquad\qquad\qquad\qquad\qquad\qquad > \tilde{c}\left(\frac{r}{2}\right)^2\frac{\sigma^2}{2} \Upsilon_k^2(x,y)\left(x-y\right)_k^2\Bigg\},
	\end{aligned}
	\end{align*}
	where $\tilde{c} := 2c-1\in(0,1)$.
	For fixed $k$ we now decompose $\Omega_r$ according to
	\begin{align*} 
		\Omega_r = \big((K_{1k}^Y)^c \cap \Omega_r\big) \cup \big(K^Y_{1k} \cap (K_{2k}^Y)^c \cap \Omega_r\big) \cup \big(K^Y_{1k} \cap K^Y_{2k} \cap \Omega_r\big).
	\end{align*}

	\noindent
	In the following we treat again each of these three subsets separately.

	\noindent
	\textit{Subset $(K_{1k}^Y)^c \cap \Omega_r$:}
	We have $\abs{\left(x-y\right)_k} \leq \frac{\sqrt{c}}{2}r$ for each $(x,y) \in (K_{1k}^Y)^c$, which can be used to independently derive lower bounds for both summands in~\eqref{eq:Tgb_}.
	For the first summand we insert the expression for $\delta^Y_{x_k} \phi_{r}(x,y)$ to get
	\begin{equation} \label{eq:TY2_set1_1_lowerbound}
	\begin{aligned}
		T^{Y,2}_{1k}(x,y)
		&= \frac{r^2}{2}\lambda\Upsilon_k(x,y) \frac{\left(x-y\right)_k}{\left(\left(\frac{r}{2}\right)^2-\left(x-y\right)^2_k\right)^2}\phi_{r}(x,y) \\
		&\geq -\frac{r^2}{2} \!\lambda \frac{\abs{\Upsilon_k(x,y)}\!\abs{\left(x-y\right)_k}}{\left(\left(\frac{r}{2}\right)^2-\left(x-y\right)^2_k\right)^2}\phi_r(x,y)
		\geq - \frac{2\lambda C_\Upsilon\sqrt{c}}{(1-c)^2\frac{r}{2}}\phi_r(x,y) \\
		&=: -p^{Y,\Upsilon}_{1}\phi_r(x,y),
	\end{aligned}
	\end{equation}
	where we recall from above that $\Upsilon_k(x,y) \leq C_\Upsilon$, which was used in the last inequality.
	For the second summand we insert the expression for $\delta^{2,Y}_{x_kx_k} \phi_{r}(x,y)$ to obtain
	\begin{equation} \label{eq:TY2_set1_2_lowerbound}
	\begin{aligned}
		T^{Y,2}_{2k}(x,y)
		&=\sigma^2 \left(\frac{r}{2}\right)^2 \Upsilon_k^2(x,y) \frac{2\left(2\left(x\!-\!y\right)^2_k\!-\!\left(\frac{r}{2}\right)^2\right)\left(x\!-\!y\right)_k^2 \!-\! \left(\left(\frac{r}{2}\right)^2\!-\!\left(x\!-\!y\right)^2_k\right)^2}{\left(\left(\frac{r}{2}\right)^2\!-\!\left(x\!-\!y\right)^2_k\right)^4}
			\phi_{r}(x,y)\\
		&\geq - \frac{\sigma^2C_\Upsilon^2}{(1-c)^4\left(\frac{r}{2}\right)^2}\phi_{r}(x,y) =: -p^{Y,\Upsilon}_{2}\phi_r(x,y),
	\end{aligned}
	\end{equation}
	where the last inequality uses $\Upsilon_k^2(x,y) \leq C_\Upsilon^2$.

	\noindent
	\textit{Subset $K^Y_{1k} \cap (K_{2k}^Y)^c \cap \Omega_r$:}
	As $(x,y) \in K^Y_{1k}$ we have $\abs{\left(x-y\right)_k} > \frac{\sqrt{c}}{2}r$.
	We observe that the sum in~\eqref{eq:Tgb_} is nonnegative for all $(x,y)$ in this subset whenever
	\begin{equation} \label{eq:aux_term_3_y}
	\begin{aligned}
		&\left(-\lambda\Upsilon_k(x,y)\left(x-y\right)_k + \frac{\sigma^2}{2} \Upsilon_k^2(x,y)\right)\left(\left(\frac{r}{2}\right)^2-\left(x-y\right)^2_k\right)^2 \\
		&\qquad\qquad\qquad\qquad\qquad\, \leq \sigma^2\Upsilon_k^2(x,y) \left(2\left(x-y\right)^2_k-\left(\frac{r}{2}\right)^2\right)\left(x-y\right)_k^2.
	\end{aligned}
	\end{equation}
	The first term on the left-hand side in \eqref{eq:aux_term_3_y} can be bounded from above exploiting that $v\in (K_{2k}^Y)^c$ and by using the relation $\tilde{c} = 2c-1$.
	More precisely, we have
	\begin{align*}
		&-\lambda\Upsilon_k(x,y)\left(x-y\right)_k \left(\left(\frac{r}{2}\right)^2-\left(x-y\right)^2_k\right)^2
		\leq \tilde{c}\left(\frac{r}{2}\right)^2\frac{\sigma^2}{2} \Upsilon_k^2(x,y)\left(x-y\right)_k^2\\
		&\qquad = (2c-1)\left(\frac{r}{2}\right)^2\frac{\sigma^2}{2} \Upsilon_k^2(x,y)\left(x-y\right)_k^2
		\leq\! \left(2\left(x-y\right)_k^2-\left(\frac{r}{2}\right)^2\right)\frac{\sigma^2}{2} \Upsilon_k^2(x,y)\left(x-y\right)_k^2,
	\end{align*}
	where the last inequality follows since $v\in K^Y_{1k}$.
	For the second term on the left-hand side in \eqref{eq:aux_term_3_y} we can use $(1-c)^2 \leq (2c-1)c$ as per assumption, to get
	\begin{align*}
		&\frac{\sigma^2}{2} \Upsilon_k^2(x,y)\left(\left(\frac{r}{2}\right)^2-\left(x-y\right)^2_k\right)^2
		\leq \frac{\sigma^2}{2} \Upsilon_k^2(x,y) (1-c)^2\left(\frac{r}{2}\right)^4 \\
		&\qquad \leq \frac{\sigma^2}{2} \Upsilon_k^2(x,y) (2c-1)\left(\frac{r}{2}\right)^2c\left(\frac{r}{2}\right)^2
		\leq \frac{\sigma^2}{2} \Upsilon_k^2(x,y) \left(2\left(x-y\right)_k^2-\left(\frac{r}{2}\right)^2\right)\left(x-y\right)_k^2.
	\end{align*}
	Hence, \eqref{eq:aux_term_3_y} holds and we have that \eqref{eq:Tgb_} is uniformly nonnegative on this subset.

	\noindent
	\textit{Subset $K^Y_{1k} \cap K^Y_{2k} \cap \Omega_r$:}
	As $(x,y) \in K_{1k}^Y$ we have $\abs{\left(x-y\right)_k} > \frac{\sqrt{c}}{2}r$.
	To start with we note that the first summand of~\eqref{eq:Tgb_} vanishes whenever $\sigma^2\Upsilon_k^2(x,y) = 0$, provided $\sigma>0$, so nothing needs to be done if $\Upsilon_k(x,y)=0$.
	Otherwise, if $\sigma^2\Upsilon_k^2(x,y) > 0$, we exploit $(x,y)\in K_{2k}^Y$ to get
	\begin{equation*}
	\begin{split}
		\frac{\Upsilon_k(x,y)\left(x-y\right)_k}{\left(\left(\frac{r}{2}\right)^2-\left(x-y\right)_k^2\right)^2}
		&\geq\frac{-\abs{\Upsilon_k(x,y)}\abs{\left(x-y\right)_k}}{\left(\left(\frac{r}{2}\right)^2-\left(x-y\right)_k^2\right)^2}\\
		&> \frac{2\lambda \Upsilon_k(x,y) \left(x-y\right)_k}{\tilde{c}\left(\frac{r}{2}\right)^2\sigma^2\abs{\Upsilon_k(x,y)}\abs{\left(x-y\right)_k}}
		\geq -\frac{8\lambda}{\tilde{c}r^2\sigma^2}.
	\end{split}
	\end{equation*}
	Using this, the first summand of~\eqref{eq:Tgb_} can be bounded from below by
	\begin{equation} \label{eq:TY2_set3_1_lowerbound}
	\begin{split}
		T^{Y,2}_{1k}(x,y)
		=\lambda\frac{r^2}{2}\frac{\Upsilon_k(x,y)\left(x-y\right)_k}{\left(\left(\frac{r}{2}\right)^2-\left(x-y\right)^2_k\right)^2}\phi_{r}(x,y)
		&\geq -\frac{4\lambda^2}{\tilde{c}\sigma^2}\phi_{r}(x,y) =: -p^{Y,\Upsilon}_{3}\phi_r(x,y).
	\end{split}
	\end{equation}
	For the second summand, the nonnegativity of $\sigma^2\Upsilon_k^2(x,y)$ implies the nonnegativity, whenever
	\begin{equation*}
		2\left(2\left(x-y\right)^2_k-\left(\frac{r}{2}\right)^2\right)\left(x-y\right)_k^2 \geq \left(\left(\frac{r}{2}\right)^2-\left(x-y\right)^2_k\right)^2.
	\end{equation*}
	This holds for $v\in K_{1k}^Y$, if $2(2c-1)c \geq (1-c)^2$ as implied by the assumption.

	\noindent
	\textbf{Concluding the proof:}
	Combining the formerly established lower bounds \eqref{eq:T*_set1_1_lowerbound}, \eqref{eq:T*_set1_2_lowerbound}, \eqref{eq:T*_set3_1_lowerbound}, \eqref{eq:TY1_set1_1_lowerbound}, \eqref{eq:TY1_set2_2_lowerbound}, \eqref{eq:TY2_set1_1_lowerbound}, \eqref{eq:TY2_set1_2_lowerbound} and \eqref{eq:TY2_set3_1_lowerbound}, we obtain for the constants $p^c$, $p^\ell$ and $p^g$ defined at the beginning of the proof
	\begin{equation}
	\begin{split}
		p^c 
			&= p^{*,\Upsilon_c}_{1} \!+ p^{*,\Upsilon_c}_{2} \!+ p^{*,\Upsilon_c}_{3} \!+ p^{Y,\Upsilon_c}_{1} \!+ p^{Y,\Upsilon_c}_{2} \!+ p^{Y,\Upsilon_c}_{3}
			= 2\!\left(\frac{2\lambda_1 C_\Upsilon\sqrt{c}}{(1-c)^2\frac{r}{2}} \!+\! \frac{\sigma_1^2C_\Upsilon^2}{(1-c)^4\!\left(\frac{r}{2}\right)^2} \!+\! \frac{4\lambda_1^2}{\tilde{c}\sigma_1^2}\right)\\
		p^\ell 
			&= p^{*,\Upsilon_\ell}_{1} \!+ p^{*,\Upsilon_\ell}_{2} \!+ p^{*,\Upsilon_\ell}_{3} \!+ p^{Y,\Upsilon_\ell}
			= \frac{2\lambda_2 C_\Upsilon\sqrt{c}}{(1-c)^2\frac{r}{2}} \!+\! \frac{\sigma_2^2C_\Upsilon^2}{(1-c)^4\left(\frac{r}{2}\right)^2} \!+\! \frac{4\lambda_2^2}{\tilde{c}\sigma_2^2} \!+\! \frac{\sigma_2^2c}{(1-c)^4}\\
		p^g
			&= p^{*,\Upsilon_g}_{1} \!+ p^{*,\Upsilon_g}_{2} \!+ p^{*,\Upsilon_g}_{3} \!+ p^{Y,\Upsilon_g}_{1} \!+ p^{Y,\Upsilon_g}_{2} \!+ p^{Y,\Upsilon_g}_{3}
			= 2\!\left(\frac{2\lambda_3 C_\Upsilon\sqrt{c}}{(1-c)^2\frac{r}{2}} \!+\! \frac{\sigma_3^2C_\Upsilon^2}{(1-c)^4\!\left(\frac{r}{2}\right)^2} \!+\! \frac{4\lambda_3^2}{\tilde{c}\sigma_3^2}\right)\!.
	\end{split}
	\end{equation}
	Together with \eqref{eq:Ts_lowerbound} and by using the evolution of $\phi_r$ as in \eqref{eq:initial_evolution} we eventually obtain
	\begin{equation*}
	\begin{split}
		&\frac{d}{dt}\iint\!\phi_{r}\,d\rho_t
		\geq -d\left(p^c + p^\ell + p^g\right) \iint\!\phi_{r}\,d\rho_t\\
		&\;\quad\,\geq -d\sum_{i=1}^3\omega_i\!\left(\!\left(1\!+\!\mathbbm{1}_{i\not=2}\right)\!\left(\frac{2\lambda_i C_\Upsilon\sqrt{c}}{(1\!-\!c)^2\frac{r}{2}} \!+\! \frac{\sigma_i^2C_\Upsilon^2}{(1\!-\!c)^4\!\left(\frac{r}{2}\right)^2} \!+\! \frac{4\lambda_i^2}{\tilde{c}\sigma_i^2}\right) \!+\! \mathbbm{1}_{i=2}\frac{\sigma_2^2c}{(1\!-\!c)^4} \!\right) \!\!\iint\!\phi_{r}\,d\rho_t\\
		&\;\quad\,= -q \iint\!\phi_{r}\,d\rho_t,
	\end{split}
	\end{equation*}
	where $q$ is defined implicitly and
	where $\omega_i=\mathbbm{1}_{\lambda_i>0}$ for $i\in\{1,2,3\}$.
	Notice that $\omega_1=1$ since $\lambda_1>0$ by assumption.	
	An application of Gr\"onwall's inequality concludes the proof.
\end{proof}

\subsection{Proof of Theorem~\ref{thm:global_convergence_main}} \label{subsec:proof_main}
We now have all necessary tools at hand to prove the global mean-field law convergence result for CBO with memory effects and gradient information by rigorously combining the formerly discussed statements.

\revised{\begin{proof}[Proof of Theorem~\ref{thm:global_convergence_main}]
	If $\CV(\rho_0)=0$, there is nothing to be shown since in this case $\rho_0=\delta_{(\globmin,\globmin)}$.
	Thus, let $\CV(\rho_0)>0$ in what follows.

	Let us first choose the parameter~$\alpha$ such that
	\begin{align} \label{eq:alpha}
	\begin{split}
		\alpha > \alpha_0
		:= \frac{1}{q_\varepsilon}\Bigg(\log\left(\frac{2^{d+2}\sqrt{d}}{c\left(\vartheta,\chi_1,\lambda_1,\sigma_1\right)}\right) 
		&+ \max\left\{\frac{1}{2},\frac{p}{(1-\vartheta)\chi_1}\right\}\log\left(\frac{\CV(\rho_0)}{\varepsilon}\right)\\
		&- \log\rho_0\big(\Omega_{r_\varepsilon/2}\big)\!\Bigg),
	\end{split}
	\end{align}
	where we introduce the definitions
	\begin{align}
		c\left(\vartheta,\chi_1,\lambda_1,\sigma_1\right)
		:= \min\left\{
			\frac{\vartheta}{2}\frac{\chi_1}{2\sqrt{2}\left(\lambda_1+\sigma_1^2\right)},
			\sqrt{\frac{\vartheta}{2}\frac{\chi_1}{\sigma_1^2}}
			\right\} 
	\end{align}
	as well as
	\begin{align} \label{eq:q_and_r_memory}
		q_\varepsilon := \frac{1}{2}\min\bigg\{\left(\eta\frac{c\left(\vartheta,\chi_1,\lambda_1,\sigma_1\right)\sqrt{\varepsilon}}{2\sqrt{d}}\right)^{1/\nu}\!,\CE_{\infty}\bigg\}
		\quad\text{and}\quad
		r_\varepsilon := \!\max_{s \in [0,R_0]}\left\{\max_{v \in B_s^\infty(\globmin)}\CE(v) \leq q_\varepsilon\right\}\!.
	\end{align}
	Moreover, $p$ is as given in \eqref{eq:def_p_memory} in Proposition~\ref{lem:lower_bound_probability_memoryCBO} with $B=c\left(\vartheta,\chi_1,\lambda_1,\sigma_1\right)\sqrt{\CV(\rho_0)}$ in $C_\Upsilon$ and with $r=r_\varepsilon$.
	By construction, $q_\varepsilon>0$ and $r_\varepsilon\leq R_0$.
	Furthermore, recalling the notation $\CE_{r}=\sup_{v \in B_{r}^\infty(\globmin)}\CE(v)$ from Proposition~\ref{lem:laplace_quant}, we have $q_\varepsilon+\CE_{r_\varepsilon} \leq 2q_\varepsilon \leq \CE_{\infty}$ according to the definition of~$r_\varepsilon$.
	Since $q_\varepsilon>0$, the continuity of $\CE$ ensures that there exists $s_{q_\varepsilon}>0$ such that $\CE(v)\leq q_\varepsilon$ for all $v\in B_{s_{q_\varepsilon}}^\infty(\globmin)$, yielding also $r_\varepsilon>0$.
	
	Let us now define the time horizon $T_\alpha \geq 0$ by
	\begin{align} \label{eq:endtime_T}
		T_\alpha := \sup\big\{t\geq0 : \CV(\rho_{t'}) > \varepsilon \text{ and } \N{\conspoint{\rho_{Y,t'}}-\globmin}_2 < C(t') \text{ for all } t' \in [0,t]\big\}
	\end{align}
	with $C(t):=c\left(\vartheta,\chi_1,\lambda_1,\sigma_1\right)\sqrt{\CV(\rho_t)}$.
	Notice for later use that $C(0)=B$.
	
	Our aim now is to show that $\CV(\rho_{T_\alpha}) = \varepsilon$ with $T_\alpha\in\big[\frac{(1-\vartheta)\chi_1}{(1+\vartheta/2)\chi_2}T^*,T^*\big]$ and that we have at least exponential decay of $\CV(\rho_{t})$ until time $T_\alpha$, i.e., until the accuracy $\varepsilon$ is reached. 
	
	First, however, we verify that $T_\alpha>0$, which is due to the continuity of $t\mapsto\CV(\rho_{t})$ and~$t\mapsto\N{\conspoint{\rho_{Y,t}}-\globmin}_2$ since $\CV(\rho_{0}) > \varepsilon$ and $\N{\conspoint{\rho_{Y,0}}-\globmin}_2 < C(0)$ at time $0$.
	While the former is a consequence of the assumption, the latter follows from Proposition~\ref{lem:laplace_quant} with $q_\varepsilon$ and $r_\varepsilon$ as defined in \eqref{eq:q_and_r_memory}, which allows to show that
	\begin{align*} 
		\N{\conspoint{\rho_{Y,0}} - \globmin}_2
		&\leq \frac{\sqrt{d}\left(q_\varepsilon\!+\!\CE_{r_\varepsilon}\right)^\nu}{\eta} \!+\! \frac{\sqrt{d}\exp\left(-\alpha q_\varepsilon\right)}{\rho_{Y,0}\big(B_{r_\varepsilon}^\infty(\globmin)\big)}\int\N{y-\globmin}_2d\rho_{Y,0}(y)\\
		&\leq \frac{\sqrt{d}\left(q_\varepsilon\!+\!\CE_{r_\varepsilon}\right)^\nu}{\eta} \!+\! \frac{\sqrt{d}\exp\left(-\alpha q_\varepsilon\right)}{\rho_{Y,0}\big(B_{r_\varepsilon}^\infty(\globmin)\big)}\iint\N{y-x}_2\!+\!\N{x-\globmin}_2d\rho_{0}(x,y)\\
		&\leq \frac{c\left(\vartheta,\chi_1,\lambda_1,\sigma_1\right)\sqrt{\varepsilon}}{2} \!+\! \frac{2\sqrt{d}\exp\left(-\alpha q_\varepsilon\right)}{\rho_{Y,0}\big(B_{r_\varepsilon}^\infty(\globmin)\big)}\sqrt{\CV(\rho_0)}\\
		&\leq c\left(\vartheta,\chi_1,\lambda_1,\sigma_1\right)\sqrt{\varepsilon}
		< c\left(\vartheta,\chi_1,\lambda_1,\sigma_1\right)\sqrt{\CV(\rho_0)} = C(0).
	\end{align*}
	The first inequality in the last line holds by the choice of $\alpha$ in \eqref{eq:alpha} and by noticing that $\Omega_{r_\varepsilon/2}\subset\bbR^d\times B_{r_\varepsilon}^\infty(\globmin)$ and thus $\rho_0(\Omega_{r_\varepsilon/2})\leq\rho_{Y,0}\big(B_{r_\varepsilon}^\infty(\globmin)\big)$.

	Next, we show that the functional $\CV(\rho_t)$ is sandwiched between two exponentially decaying functions with rates $(1-\vartheta)\chi_1$ and $(1+\vartheta/2)\chi_2$, respectively.
	More precisely, we prove that, up to time $T_\alpha$, $\CV(\rho_t)$ decays
	\begin{enumerate}[label=(\roman*),labelsep=10pt,leftmargin=35pt]
		\item at least exponentially fast (with rate $(1-\vartheta)\chi_1)$, and \label{item:proof:thm:global_convergence_main:exponential_decay_1}
		\item at most exponentially fast (with rate $(1+\vartheta/2)\chi_2)$. \label{item:proof:thm:global_convergence_main:exponential_decay_2}
	\end{enumerate}

	To obtain \ref{item:proof:thm:global_convergence_main:exponential_decay_1}, recall that Corollary~\ref{cor:evolution_of_functionalV} provides an upper bound on the time derivative of $\CV(\rho_t)$ given by
	\begin{equation} \label{eq:proof:thm:global_convergence_main:evolutionV}
	\begin{split}
		\frac{d}{dt} \CV(\rho_t)
		&\leq 
			-\chi_1 \CV(\rho_t)
			+ 2\sqrt{2}\left(\lambda_1 + \sigma_1^2\right) \sqrt{\CV(\rho_t)}\N{\conspoint{\rho_{Y,t}} - \globmin}_2
			+ \sigma_1^2\N{\conspoint{\rho_{Y,t}} - \globmin}_2^2
	\end{split}
	\end{equation}
	with $\chi_1$ as in~\eqref{eq:chi_1} being strictly positive by assumption.	
	By combining \eqref{eq:proof:thm:global_convergence_main:evolutionV} and the definition of $T_\alpha$ in \eqref{eq:endtime_T}, we have by construction
	\begin{align*}
		\frac{d}{dt}\CV(\rho_t)
		\leq -(1-\vartheta)\chi_1\CV(\rho_t)
		\quad \text{ for all } t \in (0,T_\alpha).
	\end{align*}
	Analogously, for \ref{item:proof:thm:global_convergence_main:exponential_decay_2}, by Corollary~\ref{cor:evolution_of_functionalV_lower}, we obtain a lower bound on the time derivative of $\CV(\rho_t)$ given by
	\begin{equation} \label{eq:proof:thm:global_convergence_main:evolutionV_lower}
	\begin{split}
		\frac{d}{dt} \CV(\rho_t)
		&\geq 
			-\chi_2 \CV(\rho_t)
			- 2\sqrt{2}\left(\lambda_1 + \sigma_1^2\right) \sqrt{\CV(\rho_t)}\N{\conspoint{\rho_{Y,t}} - \globmin}_2 \\
		&\geq 
			-(1+\vartheta/2)\chi_2 \CV(\rho_t)
			\quad \text{ for all } t \in (0,T_\alpha),
	\end{split}
	\end{equation}
	where the second inequality again exploits the definition of $T_\alpha$.
	Gr\"onwall's inequality now implies for all $t \in [0,T_\alpha]$ the upper and lower estimates
	\begin{subequations}
	\begin{align}
		\CV(\rho_t)
		&\leq \CV(\rho_0) \exp\left(-(1-\vartheta)\chi_1 t\right), \label{eq:evolution_V}\\
		\CV(\rho_t)
		&\geq \CV(\rho_0) \exp\left(-(1+\vartheta/2)\chi_2 t\right), \label{eq:evolution_V_lower}
	\end{align}
	\end{subequations}
	thereby proving \ref{item:proof:thm:global_convergence_main:exponential_decay_1} and \ref{item:proof:thm:global_convergence_main:exponential_decay_2}.
	The definition of $T_\alpha$ together with the one of $C(t)$ permits to control
	\begin{align}
		\max_{t \in [0,T_\alpha]}\N{\conspoint{\rho_{Y,t}}-\globmin}_2
		\leq \max_{t \in [0,T_\alpha]} C(t)\leq C(0).
		\label{eq:max_bound_distance}
	\end{align}
	To conclude it remains to prove $\CV(\rho_{T_\alpha}) = \varepsilon$ with $T_\alpha\in\big[\frac{(1-\vartheta)\chi_1}{(1+\vartheta/2)\chi_2}T^*,T^*\big]$.
	To this end, we consider the following three cases separately.
	
	\noindent
	\textbf{Case $T_\alpha \geq T^*$:}
	If $T_\alpha \geq T^*$, the time-evolution bound of $\CV(\rho_t)$ from \eqref{eq:evolution_V} combined with the definition of $T^*$ in \eqref{eq:end_time_star_statement} allows to immediately infer $\CV(\rho_{T^*}) \leq \varepsilon$.
	Therefore, with $\CV(\rho_{t})$ being continuous, $\CV(\rho_{T_\alpha}) = \varepsilon$ and $T_\alpha = T^*$ according to the definition of $T_\alpha$ in \eqref{eq:endtime_T}.
	
	\noindent
	\textbf{Case $T_\alpha < T^*$ and $\CV(\rho_{T_\alpha}) \leq \varepsilon$:}
	By continuity of $\CV(\rho_t)$, it holds for $T_\alpha$ as defined in \eqref{eq:endtime_T}, $\CV(\rho_{T_\alpha}) = \varepsilon$.
	Thus, $\varepsilon = \CV(\rho_{T_\alpha}) \geq \CV(\rho_0) \exp\left(-(1+\vartheta/2)\chi_2 T_\alpha\right)$ as a consequence of the time-evolution bound \eqref{eq:evolution_V_lower}. The latter can be reordered as
	\begin{equation*}
		\frac{(1-\vartheta)\chi_1}{(1+\vartheta/2)\chi_2}T^*
		= \frac{1}{(1+\vartheta/2)\revised{\chi_2}}\log\left(\frac{\CV(\rho_0)}{\varepsilon}\right)
		\leq T_\alpha
		< T^*.
	\end{equation*}
	
	\noindent
	\textbf{Case $T_\alpha < T^*$ and $\CV(\rho_{T_\alpha}) > \varepsilon$:}
	We will prove that this case can actually not occur by showing that $\N{\conspoint{\rho_{Y,T_\alpha}}-\globmin}_2 < C(T_\alpha)$ for the $\alpha$ chosen in~\eqref{eq:alpha}.
	In fact, if both $\CV(\rho_{T_\alpha})>\varepsilon$ and $\N{\conspoint{\rho_{Y,T_\alpha}}-\globmin}_2 < C(T_\alpha)$ held true simultaneously, this would contradict the definition of $T_\alpha$ in \eqref{eq:endtime_T}.
	To obtain this contradiction we apply again Proposition~\ref{lem:laplace_quant} with $q_\varepsilon$ and $r_\varepsilon$ as before to get
	\begin{align} \label{eq:proof_contradiction_1}
	\begin{split}
		\N{\conspoint{\rho_{Y,T_\alpha}}-\globmin}_2
		&\leq \frac{\sqrt{d}\left(q_\varepsilon\!+\!\CE_{r_\varepsilon}\right)^\nu}{\eta} \!+\! \frac{\sqrt{d}\exp\left(-\alpha q_\varepsilon\right)}{\rho_{Y,T_\alpha}\big(B_{r_\varepsilon}^\infty(\globmin)\big)}\int\N{y-\globmin}_2d\rho_{Y,T_\alpha}(y)\\
		&\leq \frac{\sqrt{d}\left(q_\varepsilon\!+\!\CE_{r_\varepsilon}\right)^\nu}{\eta} \!+\! \frac{\sqrt{d}\exp\left(-\alpha q_\varepsilon\right)}{\rho_{Y,T_\alpha}\big(B_{r_\varepsilon}^\infty(\globmin)\big)}\!\iint\!\N{y-x}_2\!+\!\N{x-\globmin}_2d\rho_{T_\alpha}(x,y)\\
		&\leq \frac{c\left(\vartheta,\chi_1,\lambda_1,\sigma_1\right)\sqrt{\varepsilon}}{2} \!+\! \frac{2\sqrt{d}\exp\left(-\alpha q_\varepsilon\right)}{\rho_{Y,T_\alpha}\big(B_{r_\varepsilon}^\infty(\globmin)\big)}\sqrt{\CV(\rho_{T_\alpha})}\\
		&< \frac{c\left(\vartheta,\chi_1,\lambda_1,\sigma_1\right)\sqrt{\CV(\rho_{T_\alpha})}}{2} \!+\! \frac{2\sqrt{d}\exp\left(-\alpha q_\varepsilon\right)}{\rho_{Y,T_\alpha}\big(B_{r_\varepsilon}^\infty(\globmin)\big)}\sqrt{\CV(\rho_{T_\alpha})}.
	\end{split}
	\end{align}
	Since, thanks to \eqref{eq:max_bound_distance}, we have $\max_{t \in [0,T_\alpha]}\Nnormal{\conspoint{\rho_{Y,t}}-\globmin}_2 \leq B$ for $B=C(0)$, which in particular does not depend on $\alpha$, Proposition~\ref{lem:lower_bound_probability_memoryCBO} guarantees the existence of $p>0$ independent of $\alpha$ (but dependent on $B$ and $r_\varepsilon$) with
	\begin{align*}
		\rho_{Y,T_\alpha}(B_{r_\varepsilon}^\infty(\globmin))
		&\geq \left(\iint \phi_{r_\varepsilon}(x,y) \,d\rho_0(x,y)\right)\exp(-pT_\alpha) \\
		&\geq \frac{1}{2^d}\,\rho_0\big(\Omega_{r_\varepsilon/2}\big) \exp(-pT^*) 
		> 0.
	\end{align*}
	Here we use that $(\globmin,\globmin)\in\supp(\rho_0)$ to bound the initial mass $\rho_0$ and the fact that $\phi_{r}$ from Lemma~\ref{lem:properties_mollifier} is bounded from below on $\Omega_{r/2}$ by $1/2^d$.
	With this we can continue the chain of inequalities in~\eqref{eq:proof_contradiction_1} to obtain
	\begin{align*} 
	\begin{split}
		\N{\conspoint{\rho_{Y,T_\alpha}}-\globmin}_2
		&< \frac{c\left(\vartheta,\chi_1,\lambda_1,\sigma_1\right)\sqrt{\CV(\rho_{T_\alpha})}}{2} + \frac{2^{d+1}\sqrt{d}\exp\left(-\alpha q_\varepsilon\right)}{\rho_0\big(\Omega_{r_\varepsilon/2}\big) \exp(-pT^*)}\sqrt{\CV(\rho_{T_\alpha})}\\
		&\leq c\left(\vartheta,\chi_1,\lambda_1,\sigma_1\right)\sqrt{\CV(\rho_{T_\alpha})}
		= C(T_\alpha),
	\end{split}
	\end{align*}
	with the first inequality in the last line holding due to the choice of $\alpha$ in \eqref{eq:alpha}.
	This gives the desired contradiction, again thanks to the continuity of $t\mapsto\CV(\rho_{t})$ and~$t\mapsto\N{\conspoint{\rho_{Y,t}}-\globmin}_2$.
\end{proof}}

\section{Numerical Experiments} \label{sec:numerics}

In the first part of this section we comment on how to efficiently implement a numerical scheme for the CBO dynamics~\eqref{eq:CBO_micro_with_memory} which allows to integrate memory mechanisms without additional computational complexity.
Afterwards we numerically demonstrate the benefit of memory effects and gradient information at the example of interesting real-world inspired applications.

\subsection{Implementational Aspects} \label{sec:implementationalaspects}

Discretizing the interacting particle system~\eqref{eq:CBO_micro_with_memory} in time by means of the Euler-Maruyama method~\cite{higham2001algorithmic} with prescribed time step size~$\Delta t$ results in the implementable numerical scheme
\begin{subequations} \label{eq:CBO_discrete_with_memory}
\begin{align}
	&\begin{aligned} \label{eq:CBO_discrete_with_memory_X}
	\!X_{k+1}^i = \begin{aligned}[t] 
		 &X_{k}^i-\Delta t\lambda_1\!\left(X_{k}^i-y_{\alpha}(\empmeasure{Y,k})\right)
		    -\Delta t\lambda_2\!\left(X_{k}^i-Y_{k}^i\right)
		    -\Delta t\lambda_3\nabla\CE(X_{k}^i) \\
		 &\!\quad\,\,\,+\sigma_1 D\!\left(X_{k}^i-y_{\alpha}(\empmeasure{Y,k})\right) B_{k}^{1,i}
		    +\sigma_2 D\!\left(X_{k}^i-Y_{k}^i\right) B_{k}^{2,i}
		    +\sigma_3 D\!\left(\nabla\CE(X_{k}^i)\right) B_{k}^{3,i},
		\end{aligned}
	\end{aligned}\\
	&\mspace{1mu}Y_{k+1}^i = Y_{k}^i + \Delta t\kappa \left(X_{k+1}^i-Y_{k}^i\right) S^{\beta,\theta}\!\left(X_{k+1}^i, Y_{k}^i\right),
	\label{eq:CBO_discrete_with_memory_Y}
\end{align}
\end{subequations}
where $((B_{k}^{m,i})_{k=0,\dots,K-1})_{i=1,\dots,N}$ are independent, identically distributed Gaussian random vectors in  $\bbR^d$ with zero mean and covariance matrix $\Delta t \Id$ for $m\in\{1,2,3\}$.

We notice that, compared to standard CBO, see, e.g., \cite[Equation~(2)]{fornasier2021consensus}, the way the historical best position is updated in~\eqref{eq:CBO_discrete_with_memory_Y} (recall the definition of $S^{\beta,\theta}$ from Equation~\eqref{eq:S_beta}) requires one additional evaluation of the objective function per particle in each time step, which raises the computational complexity of the numerical scheme substantially if computing~$\CE$ is costly and the dominating part.
However, for the parameter choices~$\kappa=1/\Delta t$, $\theta=0$ and $\beta=\infty$, in place of~\eqref{eq:CBO_discrete_with_memory_Y}, we obtain the update rule
\begin{align} \label{eq:historical_best_efficient}
	Y_{k+1}^i = \begin{cases}
		X_{k+1}^i,	&\quad\text{if } \CE(X_{k+1}^i) < \CE(Y_{k}^i),\\
		Y_{k}^i,	&\quad\text{else},
	\end{cases}
\end{align}
which is how one expects a memory mechanism to be implemented.
This way allows to recycle in time step~$k$ the computations made in the previous step, and thus leads to no additional computational cost as consequence of using memory effects.
The memory consumption, on the other hand, is approximately twice as high as in standard CBO.

\subsection{A Benchmark Problem in Optimization\,---\,The Rastrigin Function} \label{sec:numerics:Rastrigin}

Let us validate in this section the numerical observation made in Figure~\ref{fig:benefits_memory} in the introduction about the benefit of memory effects.
Namely, it has been observed in several prior works that a higher noise level can enhance the success of CBO.
To rule out that the improved performance for~$\lambda_2>0$ in Figure~\ref{fig:benefits_memory} originates solely from the larger present noise as consequence of the additional noise term associated with the memory drift, we replicate in Figure~\ref{fig:benefits_memory_validation} the experiments with the exception of setting~$\sigma_2=0$.
\begin{figure}[htp!] 
	\centering
        \includegraphics[trim=260 220 264 238,clip,height=0.2\textheight]{CBO_Rastrigin_lambda2_forMemory.pdf}%
        \vspace{0.06cm}
        \includegraphics[trim=83 220 65 238,clip,height=0.2\textheight]{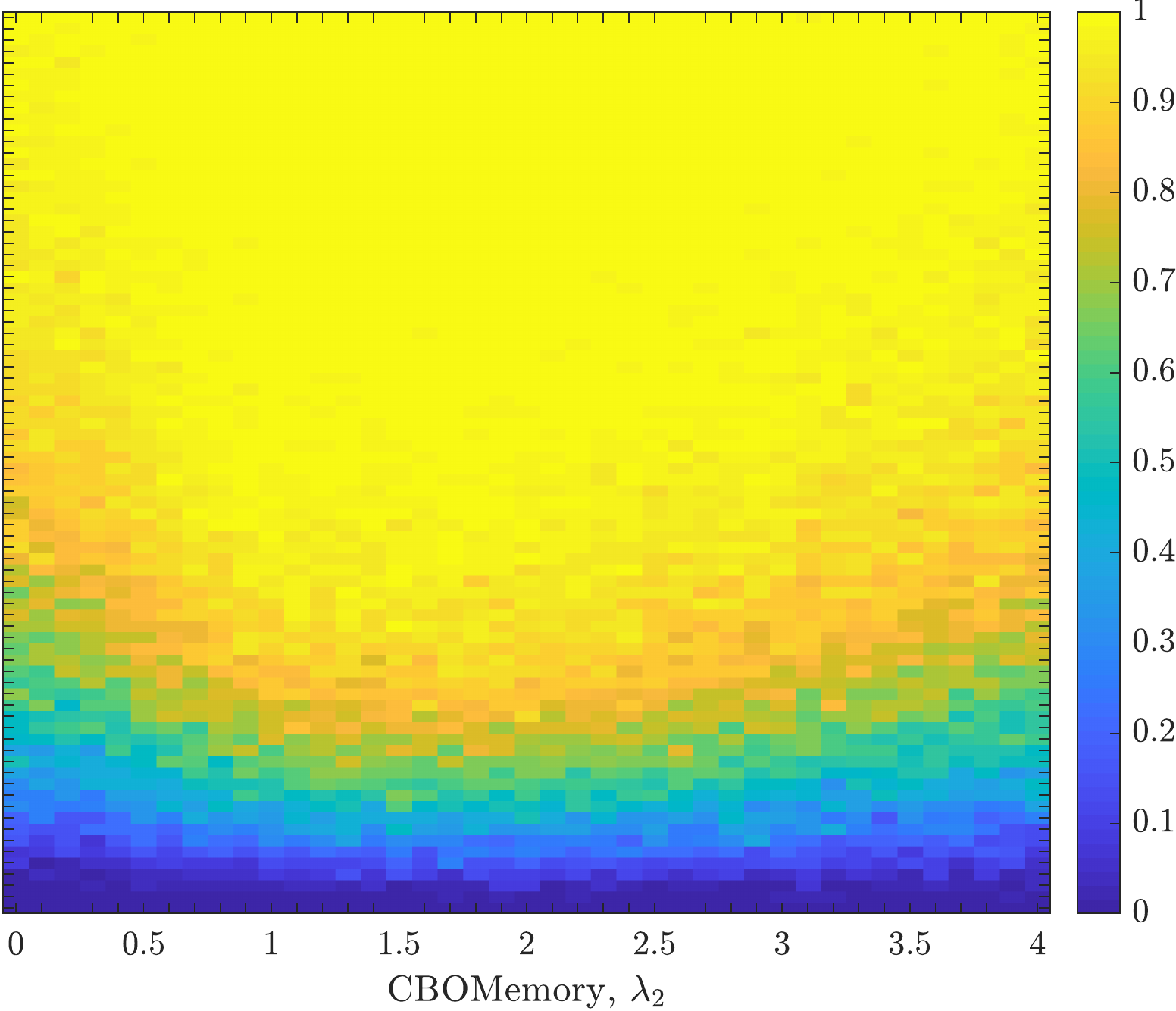}
	\caption{\footnotesize Success probability of CBO without (left separate column) and with memory effects for different values of the parameter~$\lambda_2\in[0,4]$ (right phase diagram) when optimizing the Rastrigin function in dimension~$d=4$ in the setting of Figure~\ref{fig:benefits_memory} with the exception of setting~$\sigma_2=0$.
	In this way we validate that the presence of memory effects is responsible for the improved performance and not just a higher noise level.}
	\label{fig:benefits_memory_validation}
\end{figure}
The obtained results confirm that already the usage of memory effects together with a memory drift improves the performance.
However, we also notice that an additional noise term further increases the success probability.

\subsection{A Machine Learning Example} \label{sec:numerics:NN}

As a first real-world inspired application we now investigate the influence of memory mechanisms in a high-dimensional benchmark problem in machine learning, which is well-understood in the literature, namely the training of a shallow and a convolutional NN~(CNN) classifier for the MNIST dataset of handwritten digits~\cite{MNIST}.

The experimental setting is the one of \cite[Section~4]{fornasier2021convergence} with tested architectures as described in Figure~\ref{fig:architectures}.
\begin{figure}[htp!]
	\centering
	\subcaptionbox{\label{fig:shallowNN} Shallow NN with one dense layer}{\vspace{1.01em}\includegraphics[width=0.24\textwidth, trim=0 0 0 0,clip]{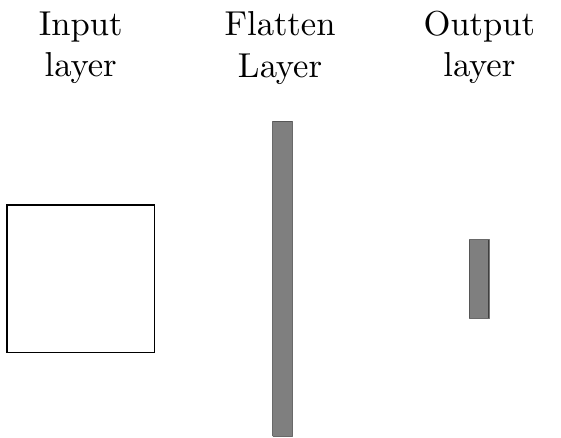}}
	\hspace{2em}
	\subcaptionbox{\label{fig:CNN} CNN~(LeNet-1), cf.\@~\cite[Section III.C.7]{lecun1998gradient}, with two convolutional and two pooling layers, and one dense layer}{\includegraphics[width=0.654\textwidth, trim=0 0 0 0,clip]{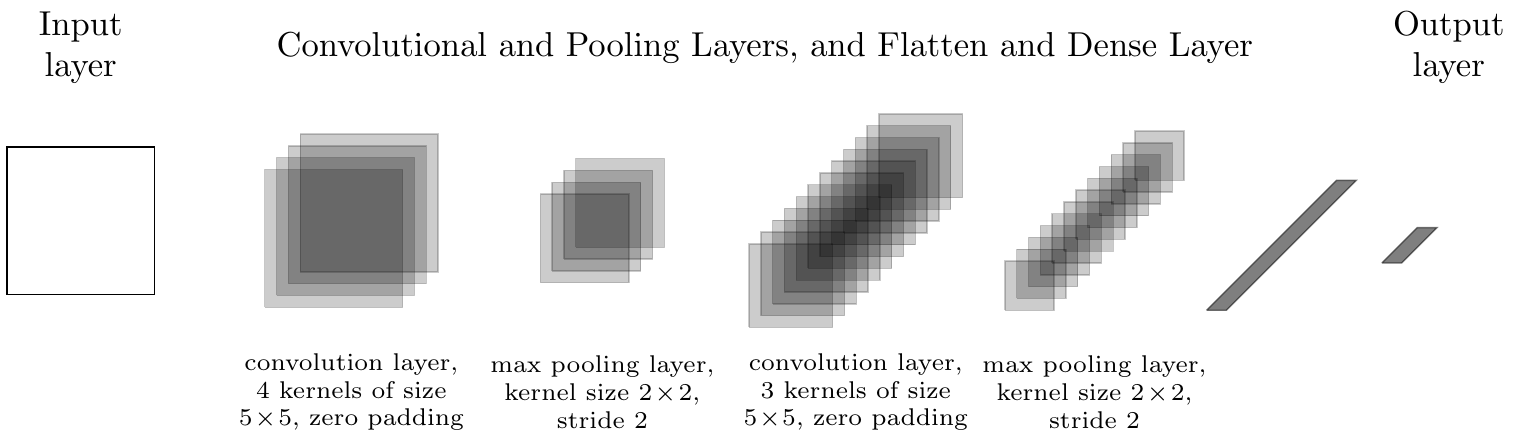}}
	\caption{NN architectures used in the experiments of Section~\ref{sec:numerics:NN}.
	Images are represented as $28\times28$ matrices with entries in $[0,1]$. For the shallow NN in \textbf{(a)} the input is reshaped into a vector~$x\in\bbR^{728}$ which is then passed through a dense layer of the form~$\mathrm{ReLU}(Wx+b)$ with trainable weights~$W\in\bbR^{10\times728}$ and bias~$b\in\bbR^{10}$.
	The learnable parameters of the CNN in \textbf{(b)} are the kernels and the final dense layer.
	Both networks include a batch normalization step after each $\mathrm{ReLU}$ activation function and a softmax activation in the last layer in order to be able to interpret the output as a probability distribution over the digits.
	We denote the trainable parameters of the NN by $\theta$. The shallow NN has $7850$ and the CNN~$2112$.
	(Reprinted by permission from Springer Nature Customer Service Centre GmbH: Springer Nature, \textit{Applications of Evolutionary Computation}, Convergence of Anisotropic Consensus-Based Optimization in Mean-Field Law, M.\@~Fornasier, T.\@~Klock, K.\@~Riedl, \copyright\,2022.)}
	\label{fig:architectures}
\end{figure}
While it is not our aim to challenge the state of the art at this task by employing very sophisticated architectures, we demonstrate that CBO is on par with stochastic gradient descent without requiring time-consuming hyperparameter tuning.

To train the learnable parameters~$\theta$ of the NNs we minimize the empirical risk~$\CE(\theta) = \frac{1}{M} \sum_{j=1}^M \ell(f_\theta(x^j),y^j)$, where $f_\theta$ denotes the forward pass of the NN and $(x^j,y^j)_{j=1}^M$ the $M$ training samples consisting of image and label.
As loss~$\ell$ we choose the categorical crossentropy loss $\ell(\widehat{y},y)=-\sum_{k=0}^9 y_k \log \left(\widehat{y}_k\right)$ with $\widehat{y}=f_\theta(x)$ denoting the output of the NN for a sample~$(x,y)$.

Our implementation is the one of \cite[Section~4]{fornasier2021convergence}, which includes concepts from \cite{carrillo2019consensus} and \cite[Section~2.2]{fornasier2020consensus_sphere_convergence}.
Firstly, mini-batching is employed when evaluating $\CE$ and when computing the consensus point~$\conspointnoarg$, which means that $\CE$ is evaluated on a random subset of size~$n_\CE=60$ of the training dataset and $\conspointnoarg$ is computed from a random subset of size~$n_N=10$ of all \revised{$N=100$} particles.
Secondly, a cooling strategy for $\alpha$ and the noise parameters is used.
More precisely, $\alpha$ is doubled each epoch, while $\sigma_1$ and $\sigma_2$ follow the schedule $\sigma_{i,epoch} = \sigma_{i,0}/\log_2(epoch+2)$ for $i=1,2$.

In Figure~\ref{fig:results_NN}
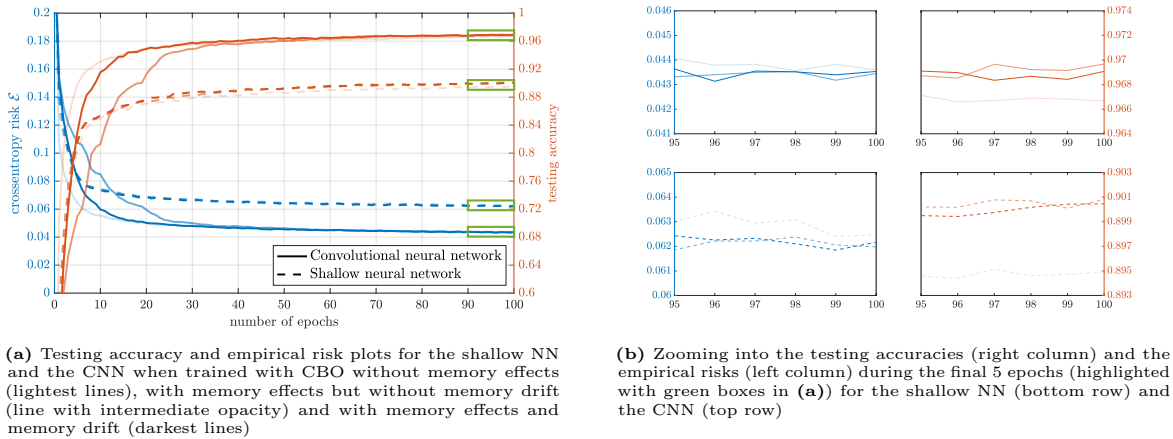
\begin{figure}[htp!]
	\centering
	\begin{subfigure}[b]{0.46\textwidth}
		\centering
		\begin{tikzpicture}
			\definecolor{green2}{rgb}{0.4660    0.6740    0.1880}
    		\node[anchor=south west,inner sep=0] (image) at (0,0) {\includegraphics[trim=28 247 8 266,clip,width=1\textwidth]{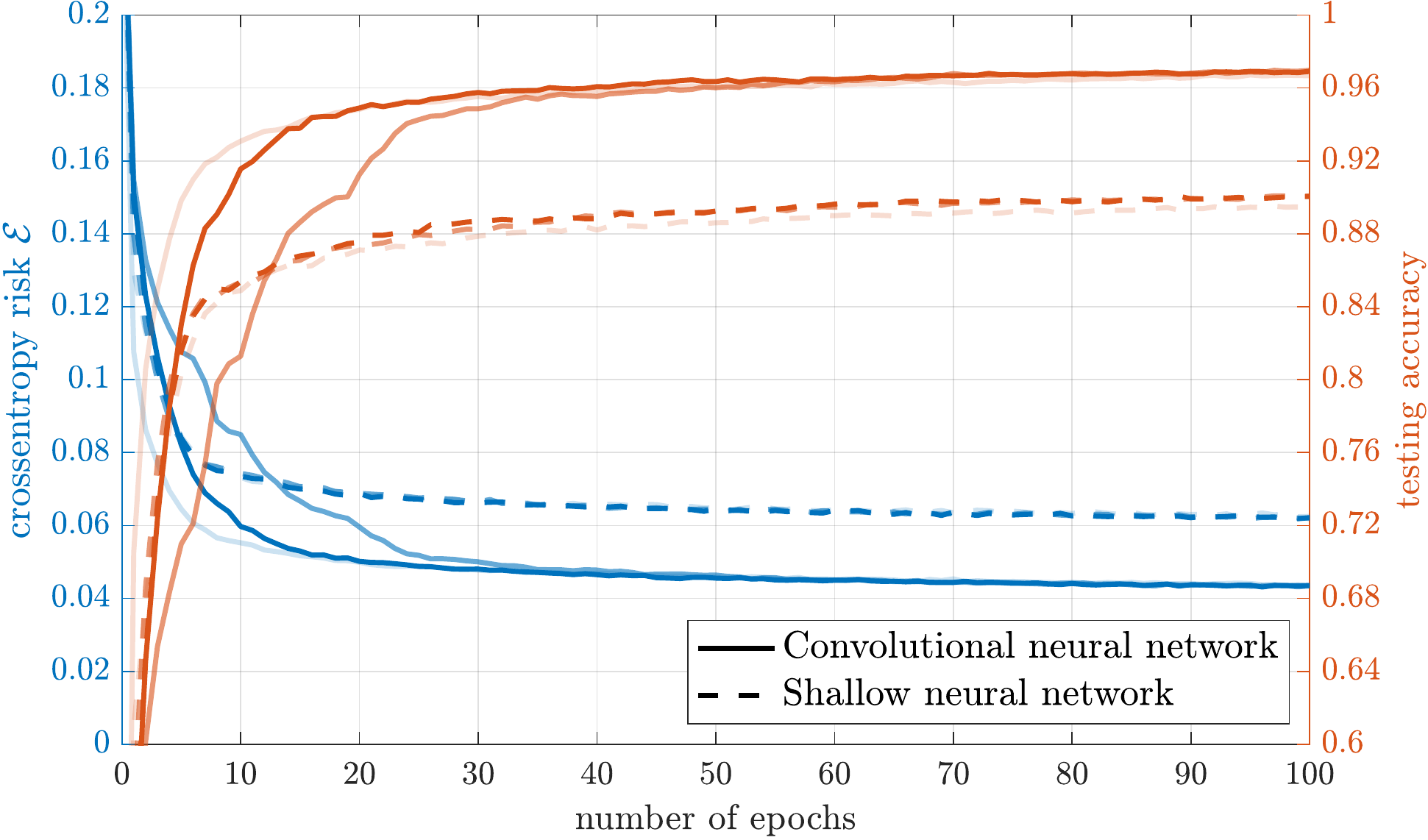}};
    		\begin{scope}[x={(image.south east)},y={(image.north west)}]
        		\draw[green2, thick] (0.833,0.8970) rectangle (0.9166,0.9270);
        		\draw[green2, thick] (0.833,0.7435) rectangle (0.9166,0.7735);
        		\draw[green2, thick] (0.833,0.3715) rectangle (0.9166,0.4015);
        		\draw[green2, thick] (0.833,0.2900) rectangle (0.9166,0.3200);
    		\end{scope}
		\end{tikzpicture}
		\caption{Testing accuracy and empirical risk plots for the shallow NN and the CNN when trained with CBO without memory effects (lightest lines), with memory effects but without memory drift (line with intermediate opacity) and with memory effects and memory drift (darkest lines)}
	\end{subfigure}~\hspace{1em}~
	\begin{subfigure}[b]{0.46\textwidth}
		\centering
		\includegraphics[trim=12 255 10 260,clip,width=0.45\textwidth]{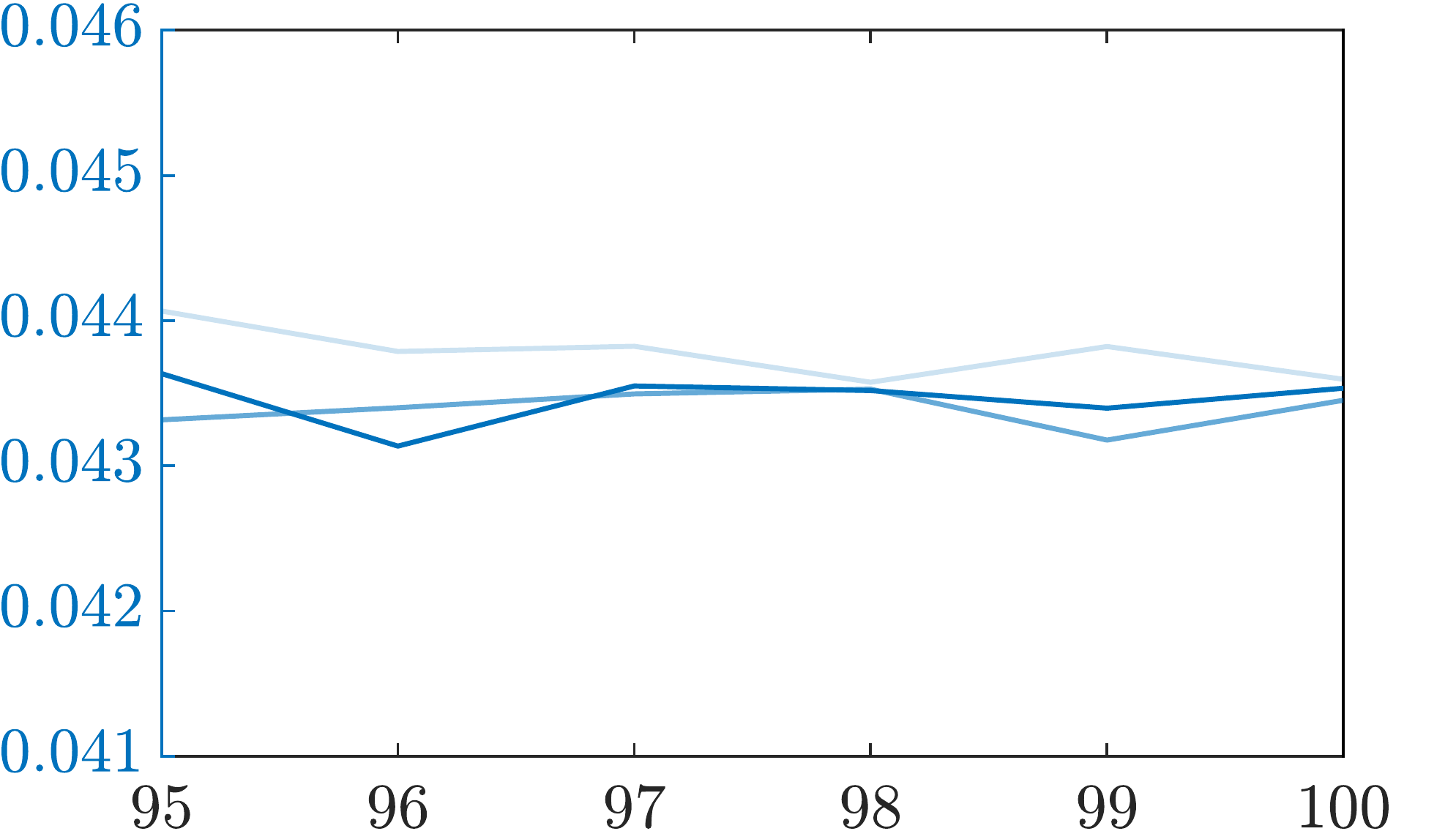}
		~\hspace{-1.4em}~
		\includegraphics[trim=12 255 10 260,clip,width=0.45\textwidth]{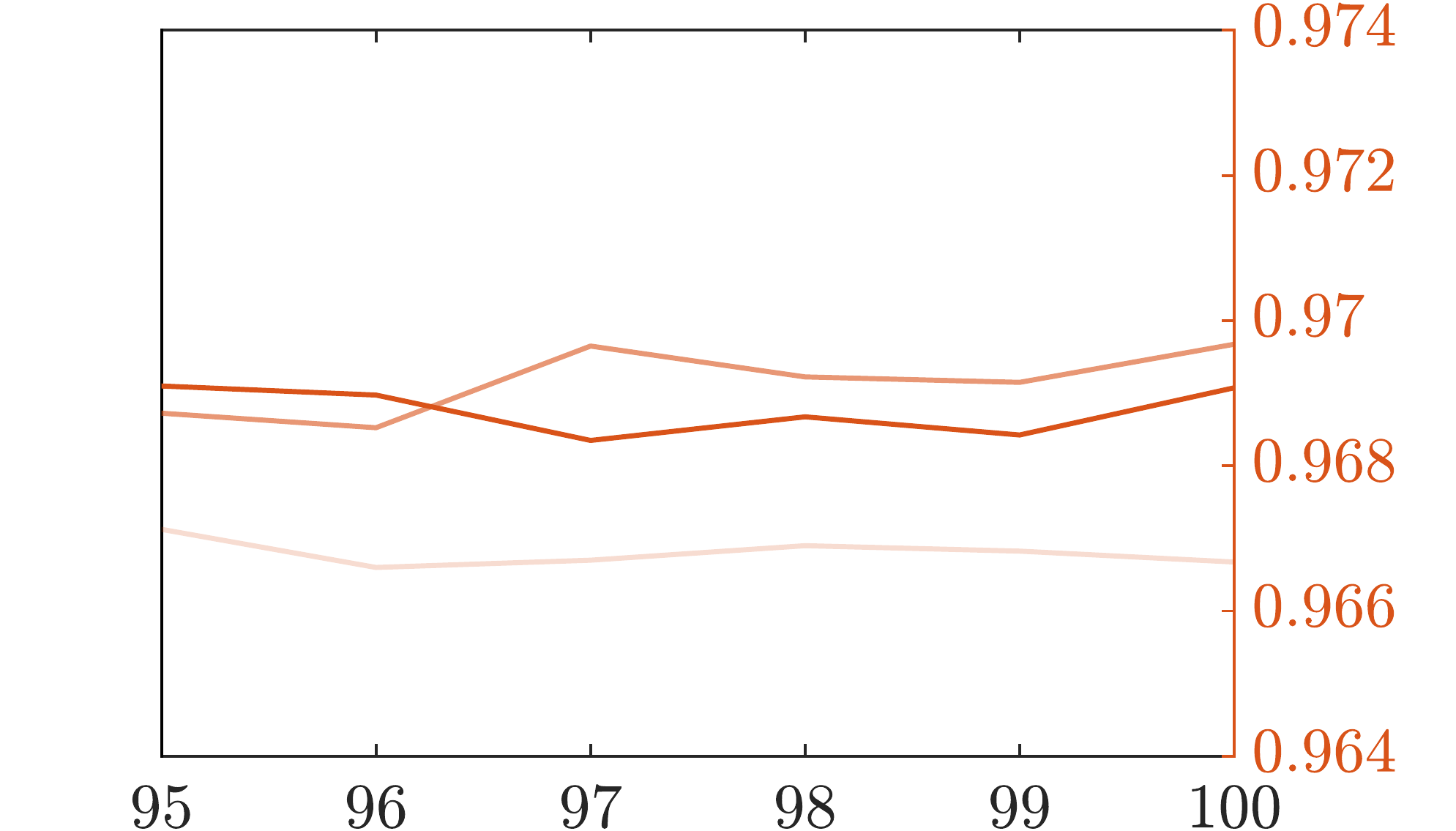}\\
		\vspace{0.63em}
		\includegraphics[trim=12 255 10 260,clip,width=0.45\textwidth]{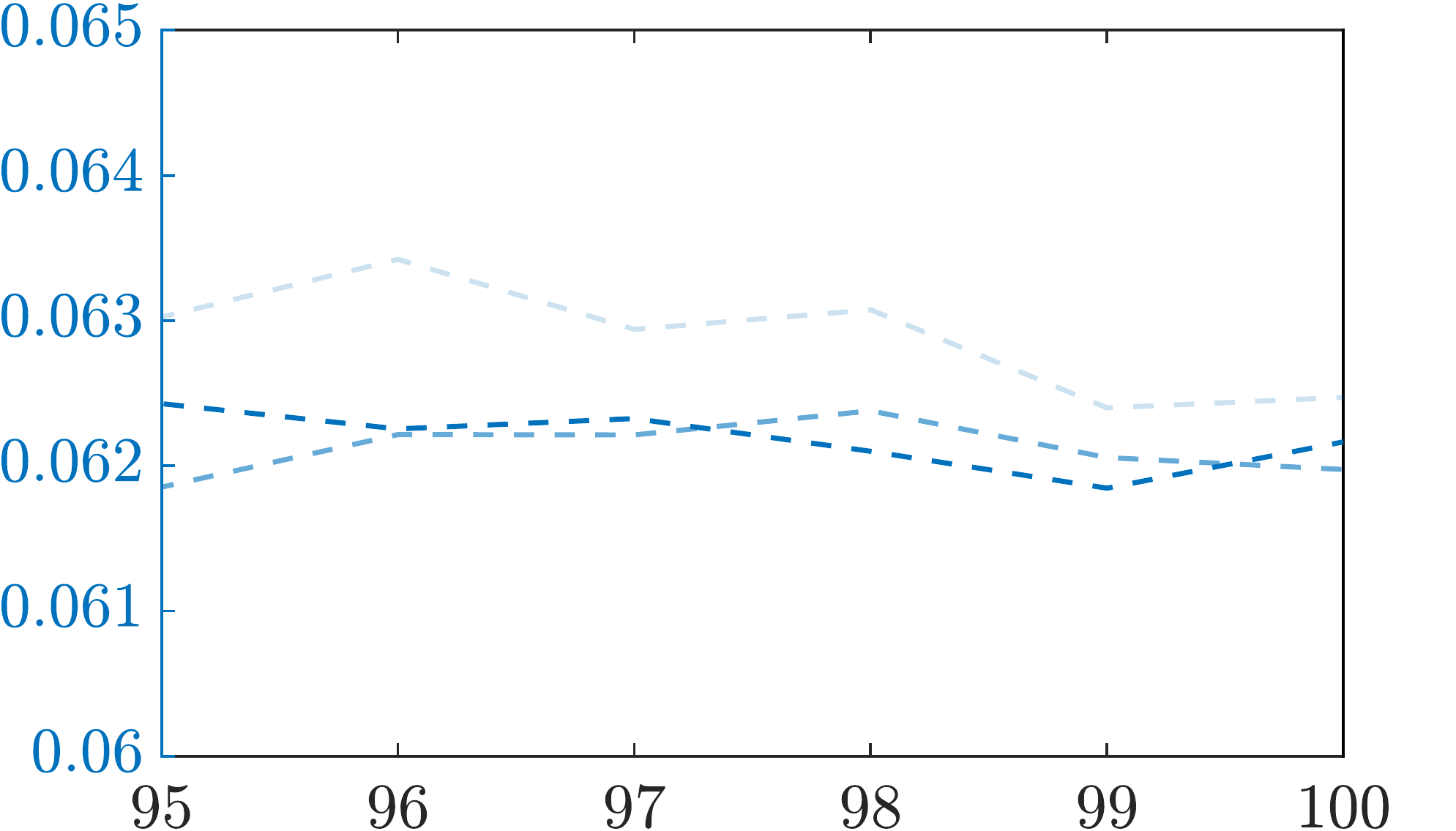}
		~\hspace{-1.4em}~
		\includegraphics[trim=12 255 10 260,clip,width=0.45\textwidth]{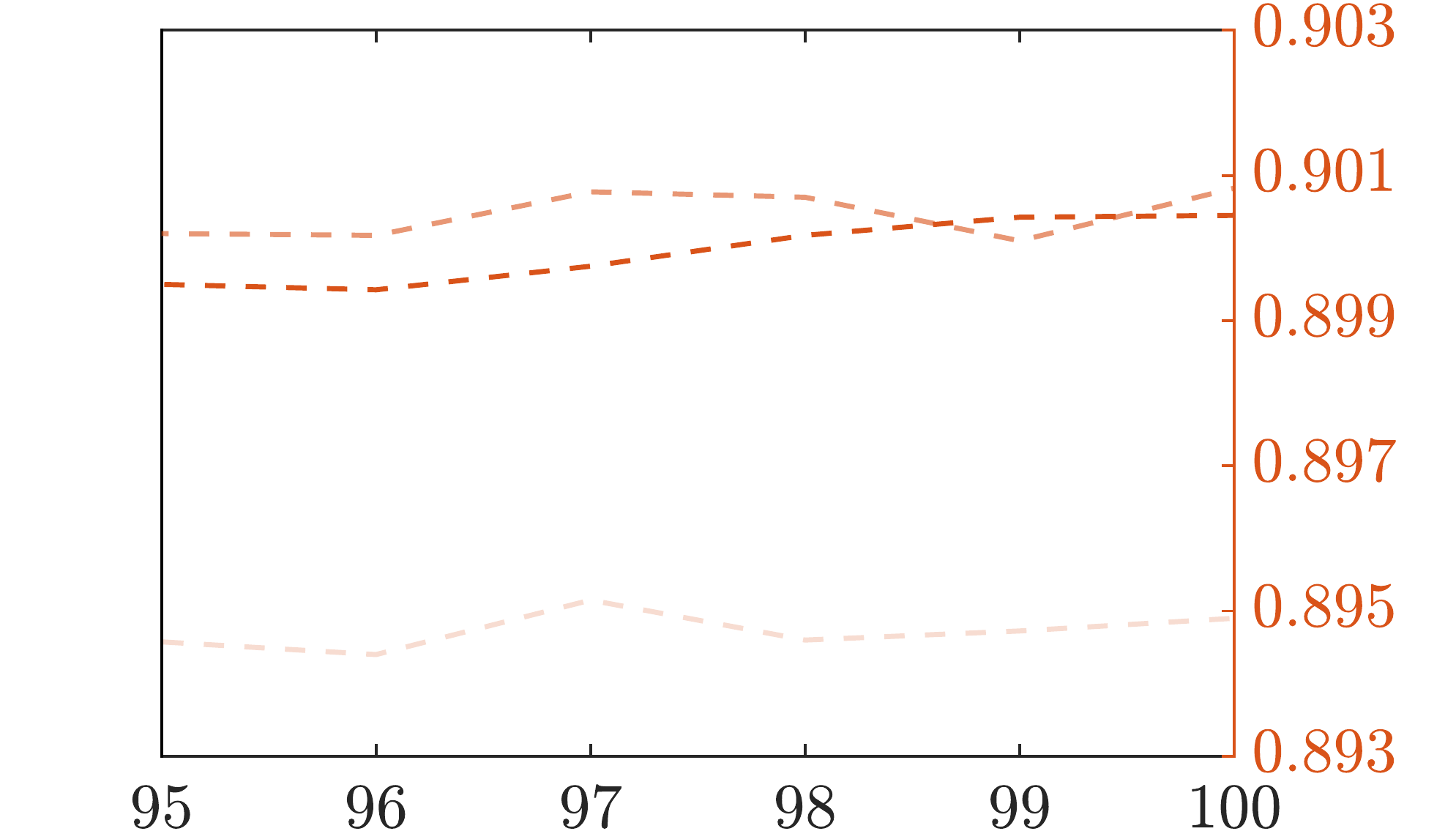}\vspace{0.75em}
		\caption{Zooming into the testing accuracies (right column) and the empirical risks (left column) during the final $5$ epochs (highlighted with green boxes in \textbf{(a)}) for the shallow NN (bottom row) and the CNN (top row)\\}
	\end{subfigure}
	\caption{Comparison of the performances (testing accuracy and training loss) of a shallow NN (dashed lines) and a CNN (solid lines) with architectures as described in Figure~\ref{fig:architectures},
	when trained with CBO without memory effects (lightest lines), 
	with memory effects but without memory drift (line with intermediate opacity) and
	with memory effects and memory drift (darkest lines).
	Depicted are the accuracies on a test dataset (orange lines) and the values of the objective function~$\CE$ evaluated on a random sample of the training set of size~$10000$ (blue lines).
	We observe that memory effects slightly improve the final accuracies while slowing down the training process initially.}
	\label{fig:results_NN}
\end{figure}
we report the testing accuracies and the training risks evaluated \revised{at the consensus point based} on a random sample of the training set of size $10000$ for both the shallow NN and the CNN when trained with one of three algorithms: standard CBO without memory effects as obtained when discretizing \cite[Equation~(1)]{fornasier2021convergence}, CBO with memory effects but without memory drift as in Equation~\eqref{eq:CBO_discrete_with_memory} with $\lambda_2=\sigma_2=0$, and CBO with memory effects and memory drift as in Equation~\eqref{eq:CBO_discrete_with_memory} with $\lambda_2=0.4$ and $\sigma_2=\lambda_2\sigma_1$.
The remaining parameters are $\lambda_1=1$, $\sigma_{1,0}=\sqrt{0.4}$, $\alpha_{initial}=50$, $\beta=\infty$, $\theta=0$, $\kappa=1/\Delta t$, and  discrete time step size $\Delta t=0.1$.
We train for $100$ epochs and use $N=100$ particles, which are initialized according to $\CN\big((0,\dots,0)^T,\Id\big)$.
All results are averaged over $5$ training runs.

As concluded already in \cite[Section~4]{fornasier2021convergence}, we obtain accuracies comparable to SGD, cf.\@~\cite[Figure~9]{lecun1998gradient}.
Moreover, we see slightly improved results when exploiting memory effects.
However, we also notice that memory mechanisms slow down the training process initially.

\subsection{A Compressed Sensing Example} \label{sec:numerics:CS}

In the final numerical section of this paper we showcase an application where gradient information turns out to be indispensable for the success of CBO methods, namely an experiment in compressed sensing~\cite{rauhut2013CS}, which has become a very active and profitable field of research since the seminal works~\cite{MR2230846, MR2241189} about two decades ago.

One of the most common problems encountered in engineering and technology is concerned with the inference of information about an unknown signal~$\globmin\in\bbR^d$ from (linear) measurements~$b\in\bbR^m$.
While classical linear algebra suggests that the number of measurements~$m$ must be at least as large as the dimensionality~$d$ of the signal, in many applications measurements are costly, time-consuming or both, making it desirable to reduce their number to the absolute minimum.
Very often one aims at $m\ll d$, since real-world signals usually live in high-dimensional spaces.
In general this would be an impossible task.
However, in typical practical scenarios additional information about the quantity of interest~$\globmin$ is available, which indeed allows to reconstruct signals from few measurements~$b$.
An empirically observed assumption about real-world signals is compressibility, meaning that they can be well-approximated by sparse vectors, i.e., vectors whose components are for the most part zero.
Exploiting sparsity enables us to solve the underdetermined system $A\globmin = b$ efficiently in both theory and practice.
Compressed sensing is concerned with the task of designing a measurement process~$A\in\bbR^{m\times d}$ together with a reconstruction algorithm capable of recovering the sparse solution $\globmin$ from the set of solutions consistent with the measurements.
This can be formulated as the nonconvex combinatorial optimization
\begin{equation*} 
	\min \N{x}_0
	\quad\text{subject to }
	Ax=b,
\end{equation*}
where $\N{x}_0$ is colloquially referred to as $\ell_0$-``norm'' and denotes the number of non-zero elements of $x$.
Solving $\ell_0$-minimization is however NP-hard in general, which lead researchers to consider the convex relaxation
\begin{equation} \label{eq:CS_objective_ell1}
	\min \N{x}_1
	\quad\text{subject to }
	Ax=b.
\end{equation}
$\ell_1$-minimization is easy to solve by means of established methods from convex optimization and provably recovers the correct solution for a suitable measurement matrix~$A$.
The remaining question is about the correct way of inferring information about the signal through measurements.
Remarkably and responsible for the wide success of compressed sensing is that random matrices enjoy properties such as the null space or restricted isometry property, which guarantee successful recovery, for $m\gtrsim s\log(d/s)$, where $s$ denotes the sparsity of the signal~$\globmin$, i.e., $s=\N{\globmin}_0$.
Up to the logarithmic factor in the ambient dimension~$d$, this is optimal, since in theory $m=2s$ measurements are necessary and sufficient to reconstruct every $s$-sparse vector.

In the numerical experiments following we resort to the regularized variant of the sparse recovery problem
\begin{equation} \label{eq:CS_objective}
	\min \ \CE(x) 
	\quad\text{with } \CE(x) = \frac{1}{2}\N{Ax-b}_2^2 + \mu\N{x}_p^p
\end{equation}
for a suitable regularization parameter~$\mu>0$.
For $p=1$ we obtain the regularization of~\eqref{eq:CS_objective_ell1}, whereas for $p<1$ the optimization~\eqref{eq:CS_objective} becomes nonconvex.
Our results in Figures~\ref{fig:benefits_gradient} and~\ref{fig:CS_results} show that CBO with gradient information is capable of solving the convex but also the nonconvex optimization problem~\eqref{eq:CS_objective} with $p=1/2$ with already very few measurements.
\begin{figure}[htp!]
	\centering
	\begin{subfigure}[b]{0.46\textwidth}
        \centering
        \includegraphics[trim=260 220 264 238,clip,height=0.2\textheight]{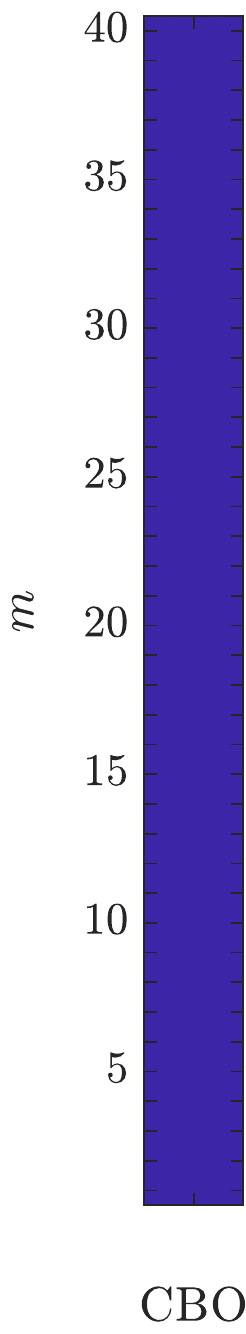}%
        \vspace{0.06cm}
        \includegraphics[trim=83 220 65 238,clip,height=0.2\textheight]{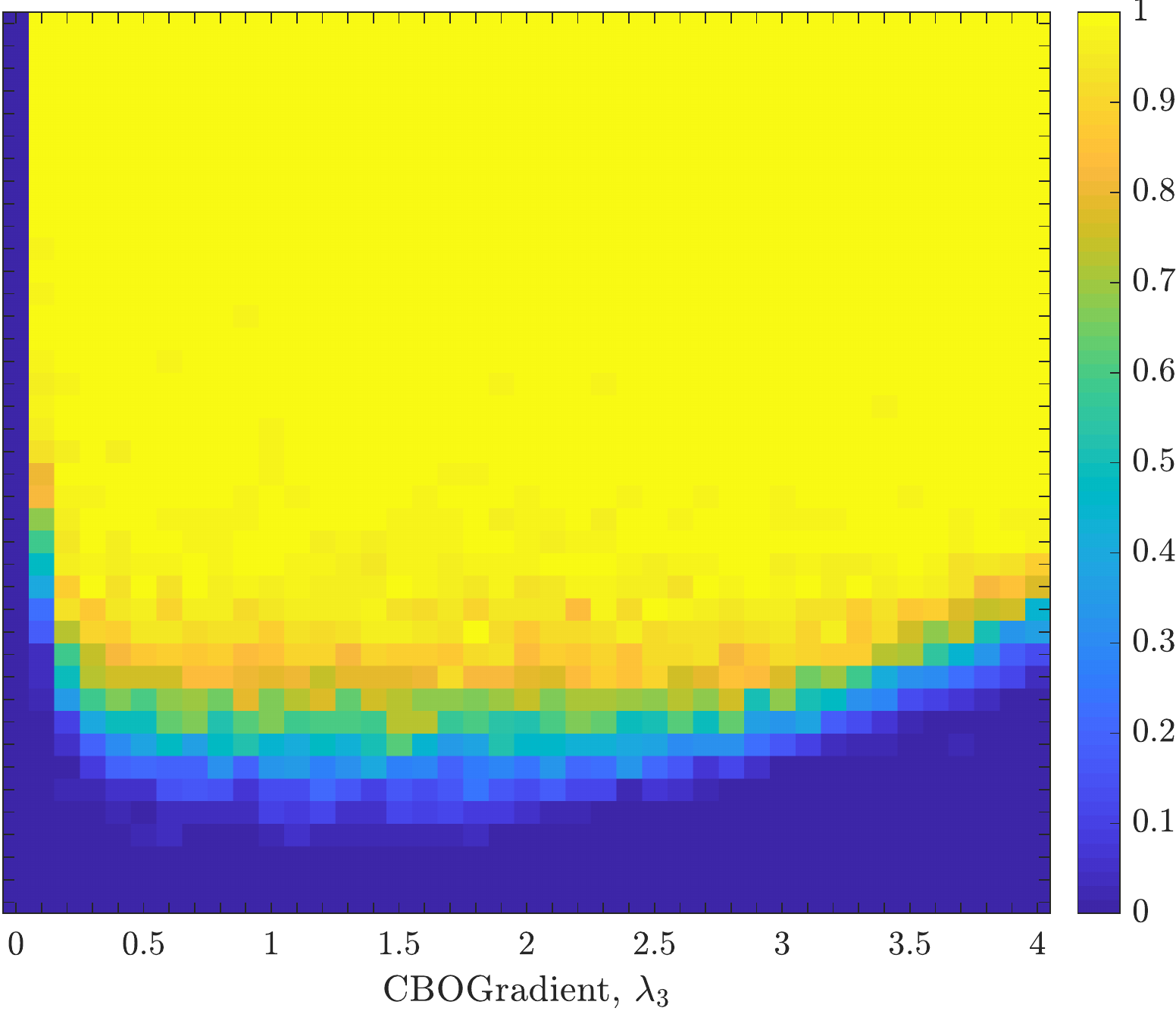}
        \caption{Convex sparse recovery using CBO without and with gradients with $N=10$ particles to solve~\eqref{eq:CS_objective} with $p=1$ from $m$ measurements}
    \end{subfigure}~\hspace{1em}~
	\begin{subfigure}[b]{0.46\textwidth}
        \centering
        \includegraphics[trim=260 220 264 238,clip,height=0.2\textheight]{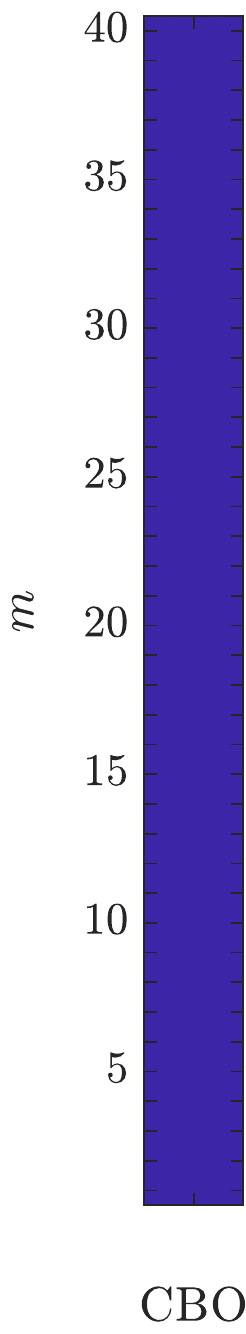}%
        \vspace{0.06cm}
        \includegraphics[trim=83 220 65 238,clip,height=0.2\textheight]{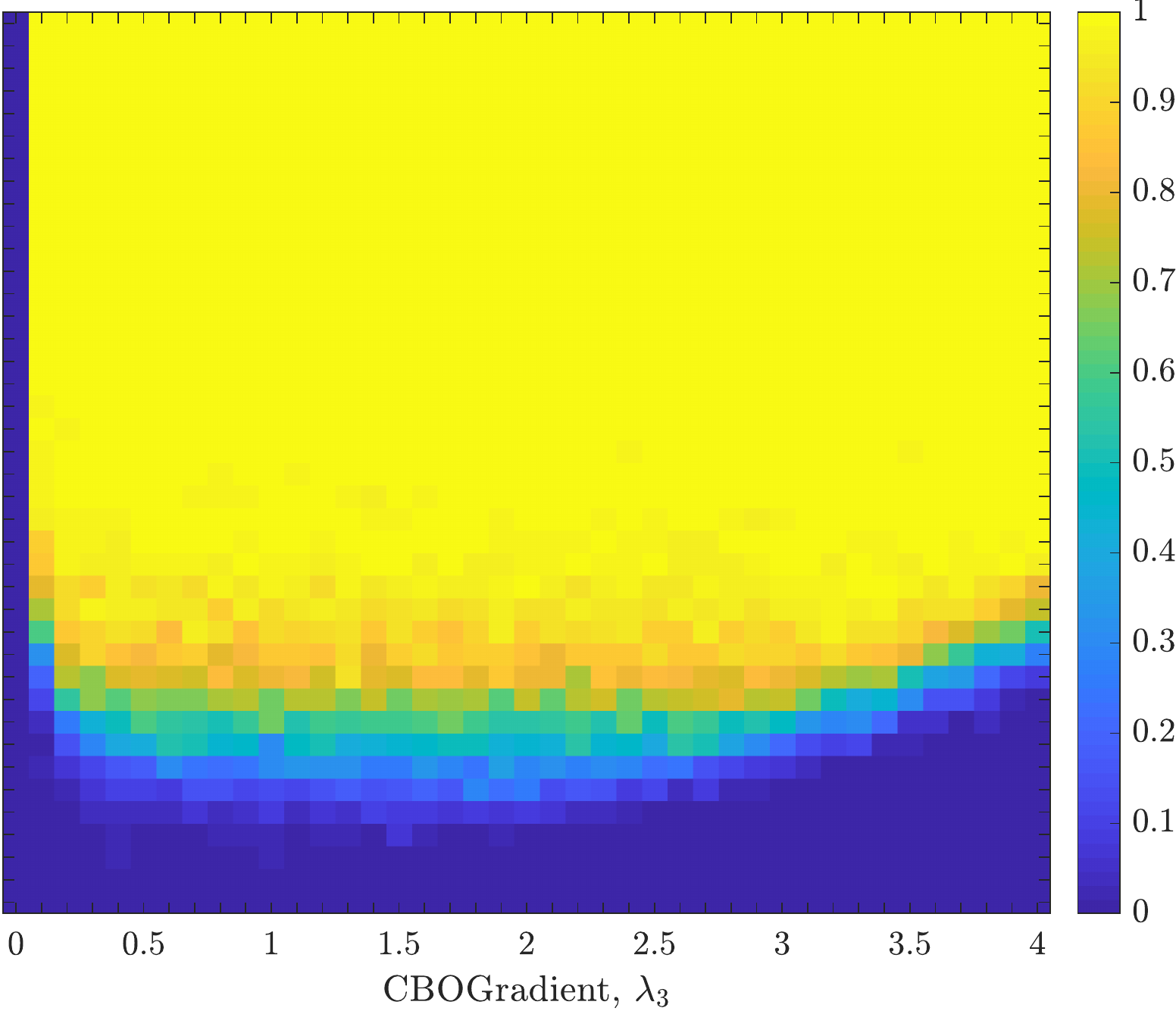}
        \caption{Convex sparse recovery using CBO without and with gradients with $N=100$ particles to solve~\eqref{eq:CS_objective} with $p=1$ from $m$ measurements}
    \end{subfigure}\\\vspace{0.24cm}
    \begin{subfigure}[b]{0.46\textwidth}
        \centering
        \includegraphics[trim=260 220 264 238,clip,height=0.2\textheight]{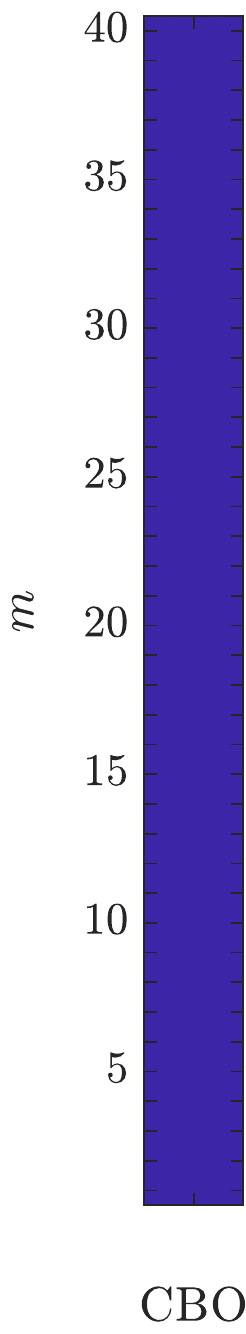}%
        \vspace{0.06cm}
        \includegraphics[trim=83 220 65 238,clip,height=0.2\textheight]{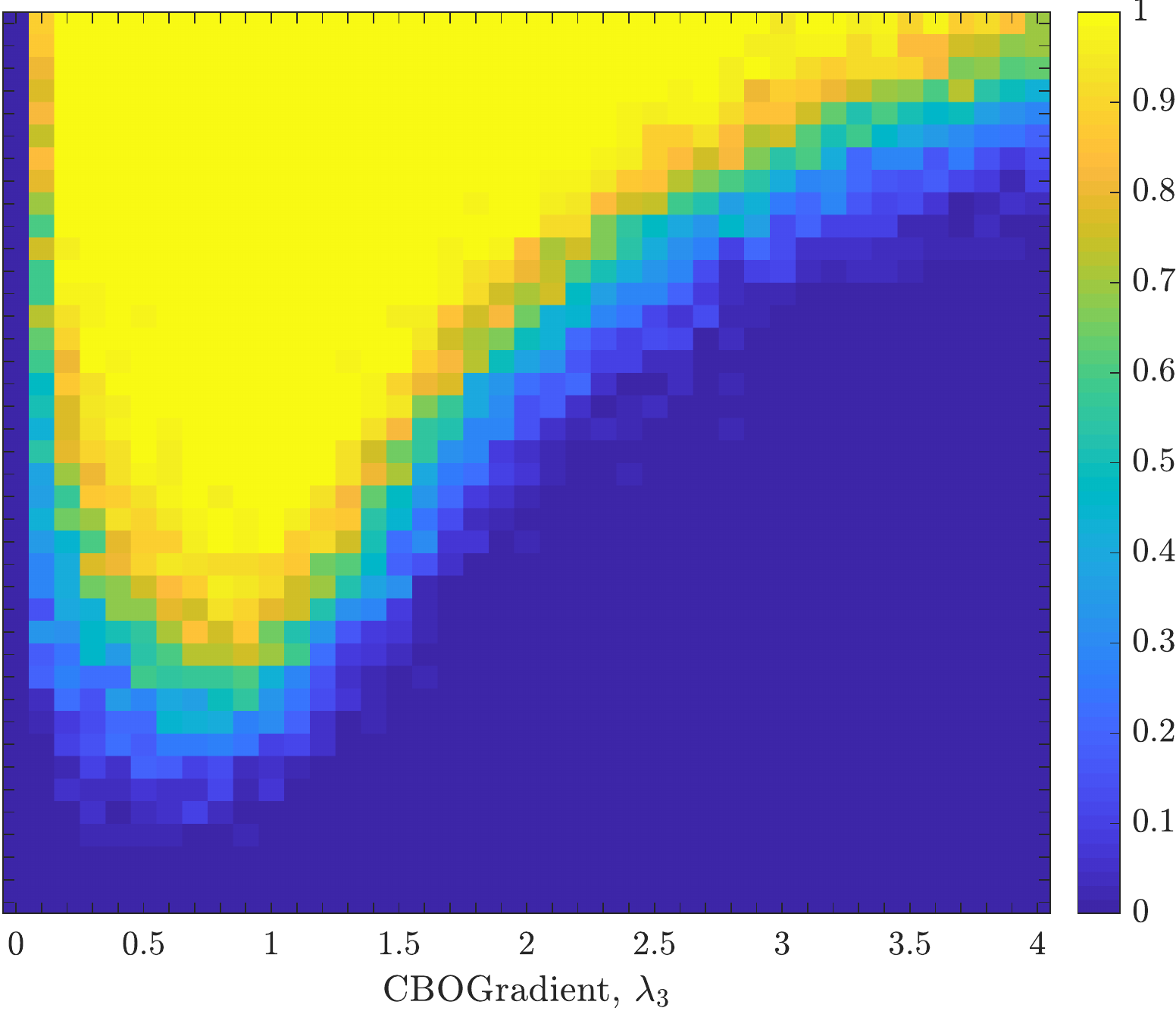}
        \caption{Nonconvex sparse recovery using CBO without and with gradients with $N=10$ particles to solve~\eqref{eq:CS_objective} with $p=0.5$ from $m$ measurements}
    \end{subfigure}~\hspace{1em}~
	\begin{subfigure}[b]{0.46\textwidth}
        \centering
        \includegraphics[trim=260 220 264 238,clip,height=0.2\textheight]{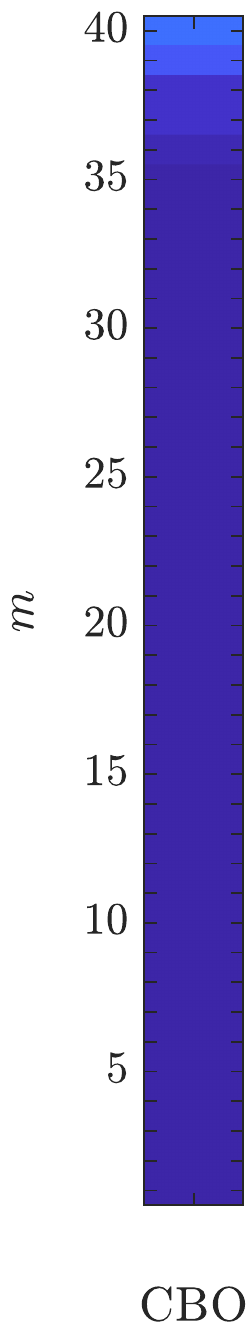}%
        \vspace{0.06cm}
        \includegraphics[trim=83 220 65 238,clip,height=0.2\textheight]{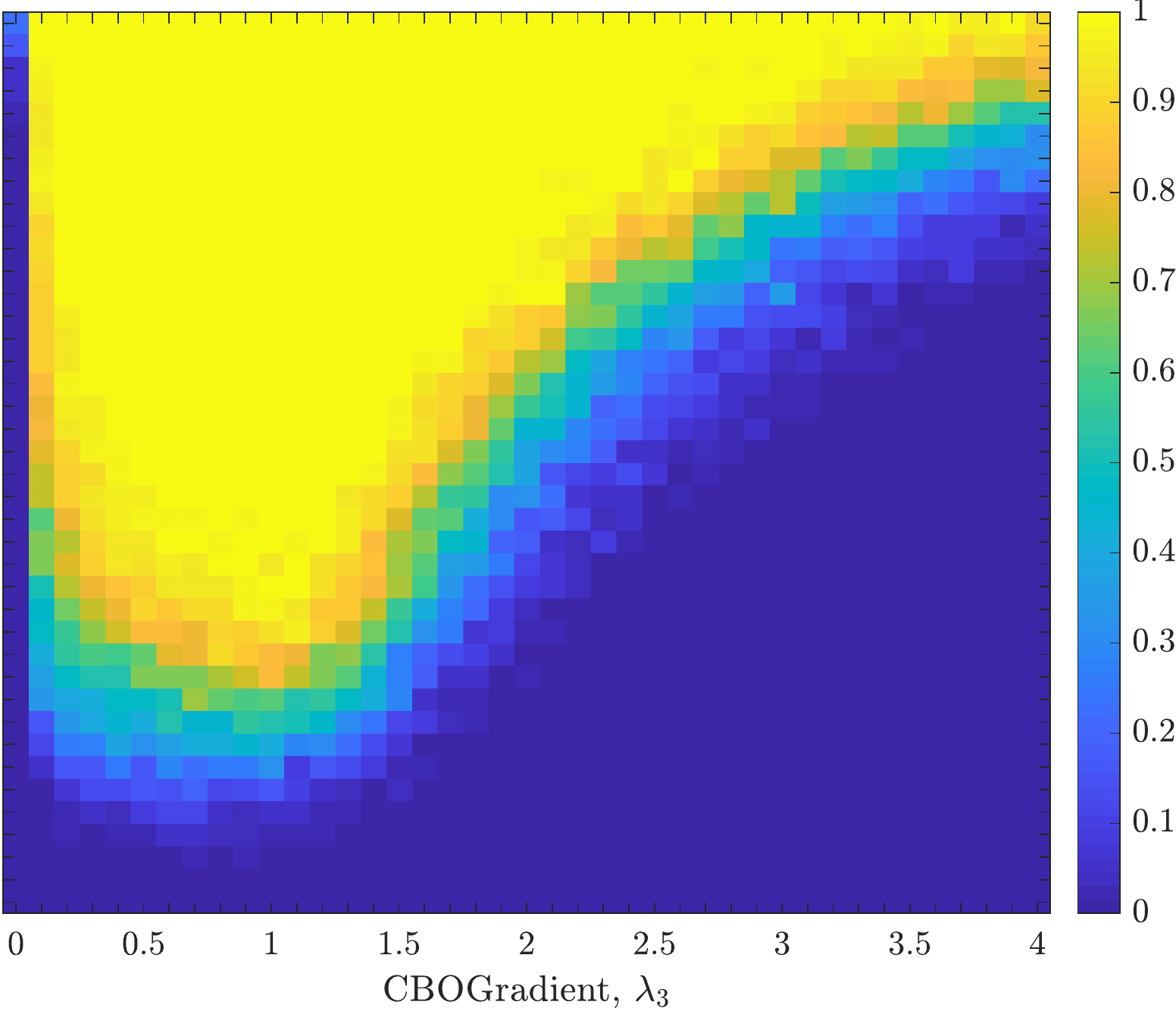}
        \caption{Nonconvex sparse recovery using CBO without and with gradients with $N=100$ particles to solve~\eqref{eq:CS_objective} with $p=0.5$ from $m$ measurements}
    \end{subfigure}
	\caption{\footnotesize 
	Comparison between the success probabilities of CBO without (left separate columns) and with gradient information for different values of the parameter~$\lambda_3\in[0,4]$ (right phase diagrams) with $N=10$ (\textbf{(a)} and \textbf{(c)}) or $N=100$ particles (\textbf{(b)} and \textbf{(d)}) when solving the convex or nonconvex $\ell_p$-regularized least squares problem~\eqref{eq:CS_objective} with $p=1$ and $\mu=$ (\textbf{(a)} and \textbf{(b)}) or $p=0.5$ and $\mu=$ (\textbf{(c)} and \textbf{(d)}), respectively.
	On the vertical axis we depict the number of measurements~$m$, from which we try to recover the $2$-sparse and $50$-dimensional sparse signal.
	As further parameters we choose the time horizon~$T=20$, discrete time step size~$\Delta t=0.01$, $\alpha=100$, $\beta=\infty$, $\theta=0$, $\kappa=1/\Delta t$, $\lambda_1=1$, $\lambda_2=0$ and $\sigma_1=\sigma_2=\sigma_3=0$.
	We discover that in both the convex and nonconvex setting reconstruction is feasible from already very few measurements.
	While increasing the number of optimizing particles has almost no effect for the convex optimization problem, in the nonconvex setting recovery benefits from more particles.
	Note that the separate columns and the left most column of the phase diagrams coincide, and are only depicted in this way to highlight that we compare also CBO.}
	\label{fig:CS_results}
\end{figure}
As parameters of the CBO algorithm, which is obtained as a Euler-Maruyama discretization of Equation~\eqref{eq:CBO_micro_with_memory}, we choose in both cases the time horizon~$T=20$, time step size~$\Delta t=0.01$, $\alpha=100$, $\beta=\infty$, $\theta=0$, $\kappa=1/\Delta t$, $\lambda_1=1$, $\lambda_2=0$ and $\sigma_1=\sigma_2=\sigma_3=0$.
We use either $N=10$ or $N=100$ particles, which is specified in the respective caption.
After running the CBO algorithm, a post-processing step is performed, in which the support of the suspected sparse vector is identified by checking which entries are not smaller than $0.01$ before the final sparse solution is then obtained by solving the linear system constrained to this support.

The depicted success probabilities are averaged over $100$ runs of CBO.
In Figure~\ref{fig:benefits_gradient} we solve the sparse recovery problem in the convex setting for an $8$-sparse $200$-dimensional signal with $p=1$ using CBO without and with gradient information with merely $10$ particles. 
We observe that gradient information is indispensable to be able to identify the correct sparse solution and standard CBO would fail in such task.
In Figure~\ref{fig:CS_results} we conduct a slightly lower-dimensional experiment with a $2$-sparse $50$-dimensional signal. 
Here our focus is to enter the nonconvex recovery regime by comparing the convex $\ell_1$-regularized with the nonconvex $\ell_{1/2}$-regularized problem.
We discover that in either case reconstruction is feasible from already very few measurements.
While increasing the number of optimizing particles has almost no effect for the convex optimization problem, in the nonconvex setting recovery benefits from more particles.
Furthermore, the nonconvex problem demands a more moderate choice of the strength of the gradient.

\section{Conclusion} \label{sec:conclusion}

In this paper we investigate a variant of consensus-based optimization (CBO) which incorporates memory effects and gradient information.
By developing further and generalizing the proof technique devised in~\cite{fornasier2021consensus,fornasier2021convergence}, we establish the global convergence of the underlying dynamics to the global minimizer~$\globmin$ of the objective function~$\CE$ in mean-field law.
To this end we analyze the time-evolution of the Wasserstein distance between the law of the mean-field CBO dynamics and a Dirac delta at the minimizer, and show its exponential decay in time.
Our result holds under minimal assumptions about the initial measure~$\rho_0$ and for a vast class of objective functions.
The numerical benefit of such additional terms, specifically the employed memory effects and gradient information, is demonstrated at the example of a benchmark function in optimization as well as at real-world applications such as compressed sensing and the training of neural networks for image classification.

By these means we demonstrate the versatility, flexibility and customizability of the class of CBO methods, both with respect to potential application areas in practice and modifications to the underlying optimization principles, while still being amenable to theoretical analysis.

\section*{Acknowledgements}
I would like to profusely thank Massimo Fornasier for many fruitful and stimulating discussions while I was preparing this manuscript.

This work has been funded by the German Federal Ministry of Education and Research and the Bavarian State Ministry for Science and the Arts.
I, the author of this work, take full responsibility for its content.
Furthermore, I acknowledge the financial support from the Technical University of Munich -- Institute for Ethics in Artificial Intelligence (IEAI).
Moreover, I gratefully acknowledge the compute and data resources provided by the Leibniz Supercomputing Centre (LRZ).

\bibliographystyle{abbrv}
\bibliography{biblio.bib}

\end{document}